\documentclass[12pt,leqno]{amsart}
\makeatletter
\def\thm@space@setup{%
  \thm@preskip=6pt plus 2pt minus 1pt
  \thm@postskip=0pt
}
\makeatother

\usepackage[margin=1in]{geometry}
\usepackage[T1]{fontenc}
\usepackage{lmodern}
\usepackage{microtype}

\usepackage{parskip}

\usepackage{xcolor}

\usepackage{amsmath,amssymb,mathtools}

\makeatletter
\renewenvironment{proof}[1][\proofname]{%
	\par
	\pushQED{\qed}%
	\normalfont
	\topsep6\p@\@plus6\p@\relax
	\trivlist
	\item[\hskip\labelsep\itshape #1\@addpunct{.}]%
	\ignorespaces
}{%
	\par\nobreak\vspace{2pt}
	\popQED%
	\medskip
	\endtrivlist
	\@endpefalse
}
\makeatother

\allowdisplaybreaks

\usepackage{etoolbox}

\usepackage{enumitem}
\setlist[enumerate]{label=(\arabic*),leftmargin=*,itemsep=2pt}
\setlist[itemize]{leftmargin=*,itemsep=2pt}

\usepackage{array,booktabs,makecell,multirow}
\usepackage{tabularx}
\usepackage{adjustbox}

\usepackage{graphicx}
\usepackage{float}
\usepackage{comment}
\usepackage{tikz}

\usepackage[unicode]{hyperref}
\hypersetup{
	colorlinks=true,
	linkcolor=blue!50!black,
	citecolor=blue!50!black,
	urlcolor=blue!50!black
}
\usepackage[nameinlink,noabbrev]{cleveref}

\usepackage[numbers,sort&compress]{natbib}

\usepackage{diagbox}
\usepackage{rotating}
\usepackage[table]{xcolor}
\newcommand{\be}{\begin{equation}}

\newcommand{\ee}{\end{equation}}
\newtheorem{thm}{Theorem}[section]
\newtheorem{lem}{Lemma}[section]
\newtheorem{prop}{Proposition}[section]
\newtheorem{defn}{Definition}[section]
\newtheorem{rmk}{Remark}[section]
\newtheorem{cor}{Corollary}[section]
\newtheorem{ques}{Question}[section]

\newtheorem{exm}{Example}[section]
\newtheorem{alg}{Algorithm}[section]

\newcommand{\aut}{{\operatorname{Aut}}}
\newcommand{\psl}{{\operatorname{PSL}}}
\newcommand{\ord}{{\operatorname{ord}}}
\newcommand{\val}{{\operatorname{val}}}

\numberwithin{equation}{section}
\numberwithin{figure}{section}

\begin{document}

\title[Two-Generator Discrete and Faithful Subgroups of Tree Automorphisms]{Two-Generator Discrete and Faithful Subgroups of Tree Automorphisms}

\author{Yukun Du}
\address[Y. Du]{Department of Mathematics, University of Georgia, Athens, GA 30603}
\email{yukun.du@uga.edu}

\subjclass[2010]{20E08, 05C25, 20E36, 22E35}
\keywords{Tree automorphisms, Graphs of groups, Discrete subgroups, Faithful amalgams}

\begin{abstract}
We present a partial classification of two-generator discrete and faithful subgroups of the trivalent tree automorphism group, specifically for cases where the generators satisfy a restriction on a small geometric quantity. When the restrictions on the geometric quantity or tree valency are relaxed, we discuss the possible reduced quotient graphs for these subgroups and construct infinite families of graphs of groups on each. Additionally, we include a generalized Poincar\'e algorithm that determines whether a given set of tree automorphisms generates a discrete subgroup.
\end{abstract}

\date{\today}
\maketitle
\tableofcontents

\section{Introduction}\label{sec:introduction}
This paper studies discrete two-generator subgroups of automorphism groups of locally finite trees, especially when torsion is present. Using graphs of groups, we obtain a classification in a highly restricted scope; without that restriction, we construct infinite families of such subgroups with every possible underlying graph.

In the classical setting, this question arises from the study of Fuchsian groups, in particular from sufficient conditions for subgroups $\Gamma<\psl(2,\mathbb{R})$ to be discrete. When $\Gamma$ is generated by two elements, $\Gamma = \langle A,B\rangle$, algebraic criteria and classification results were obtained by Rosenberger \cite{rosenberger1986all}, and later given a geometric treatment by Gilman \cite{gilman1995two}. Gilman’s monograph describes discreteness in terms of geometric quantities such as hyperbolic axes and elliptic rotation angles. In the most complicated case, when $A$ and $B$ are hyperbolic with intersecting axes, Gilman gave a geometric algorithm based on Nielsen moves, leading either to a Schottky pair or to a pair in the \emph{intertwining cases}. Based on Rosenberger's and Gilman's work, Kirschmer and R\"uther suggested an algorithm solving the constructive membership problem for two-generator Fuchsian groups \cite{kirschmer2017constructive}.

The analogous problem for two-generator discrete subgroups in the non-archimedean setting is less understood. For hyperbolic pairs $A,B\in\psl(2,K)$ over a non-archimedean local field $K$, Conder introduced a similar Nielsen-move algorithm, producing either a Schottky pair or an elliptic-hyperbolic pair \cite{conder2020discrete}. More recent work of Conder and Schillewaert \cite{conder2022discrete} included elliptic generators as well, and obtained a partial classification analogous to the classical hyperbolic case. If $p$ is the residue characteristic of $K$, their classification applies under the condition that no element of $\Gamma$ has order $p$; in particular, it is complete when $K = \mathbb{Q}_p$.

Conder’s reduction algorithm remains meaningful in the broader setting where the generators lie in the full automorphism group $\aut(X)$ of a locally finite tree $X$. In fact, our earlier work \cite{Du_Hersonsky_Tree_Aut} showed that the Schottky case can be characterized more directly in terms of geometric quantities of the generating pair $(A,B)$, such as their translation lengths and the length of the intersection of their axes. We also computed the corresponding geometric quantities for the resulting equivalent pair, whether Schottky or elliptic-hyperbolic. These quantities appear to be closely related to the continued fraction of the ratio $\ell(A)/\ell(B)$.

By contrast, understanding discrete subgroups of $\aut(X)$ generated by elliptic-hyperbolic or two-elliptic pairs is more difficult than in the $\psl(2,K)$ case, and even more so than in the $\psl(2,\mathbb{R})$ setting:
\begin{ques}\label{ques}
    Let $A,B \in \mathrm{Aut}(X)$ be automorphisms of a locally finite tree $X$, and assume that the subgroup $\Gamma = \langle A,B \rangle$ is discrete. What isomorphism types can $\Gamma$ have?
\end{ques}
While a complete answer is far beyond the scope of this paper, we investigate how far such discrete two-generator subgroups can be described. By a theorem of Bass, every finitely generated discrete subgroup of $\aut(X)$ contains a torsion-free subgroup of finite index \cite{bass1990uniform}. Combined with Lubotzky’s foundational result that finitely generated torsion-free discrete subgroups of $\aut(X)$ are free and Schottky \cite{lubotzky1991lattices}, this implies that finitely generated discrete subgroups are virtually free. By Bass–Serre theory, such groups admit a description in terms of \emph{finite graphs of finite groups}, which provides the framework for our classification results.

Our first main result gives a classification of faithful HNN $2$-extensions. The resulting family is explicit and resembles the infinite family of $(4,2)$-amalgamated products of dihedral type \cite{djokovic1980class}. In addition, we analyzed the generating pairs of the faithful $(3,2)$-amalgamated products by Djokovi\'c--Miller \cite{djokovic1977regular}, and of faithful $(3,3)$-amalgamated products by Goldschmidt \cite{goldschmidt1980automorphisms}. Combining these results, we determined the list of two-generator discrete subgroups of $\aut(X_3)$ with \emph{small} geometric quantities and acting faithfully on their minimal invariant subtree. In particular, we considered pairs $g,h\in\aut(X_3)$ in two cases: both generators elliptic with $d(T_g,T_h)=1$, or $g$ elliptic and $h$ hyperbolic with $d(T_g,A_h)=0$ and $\ell(h)\leq 2$. We showed that particular discrete subgroups $\Gamma=\langle g,h\rangle$ arise in these settings: they either fall among finitely many amalgamated products of finite groups, or belong to a single explicit family of HNN extensions. By contrast, subgroups without the faithful condition form a more extensive family.

This classification yields further rigidity. The order of an elliptic element is restricted to $2,3,4,6,8,$ or $12$; and, except in the HNN case, one has $\ell(A_g\cap T_h)\leq 8$. Combined with our reduction result for hyperbolic generating pairs, this also gives a rigidity statement for hyperbolic $g,h\in\aut(X_3)$ under the assumptions $\gcd(\ell(g),\ell(h))\leq 2$ and $\ell(A_g\cap A_h)\geq \ell(g)+\ell(h)-\gcd(\ell(g),\ell(h))$.


Our second main result concerns the complementary regime, where the small-geometry restrictions are relaxed: either $\ell(h)>2$, or $d(T_g,A_h)\geq 1$, or the ambient tree $X$ has vertices of valency at least $4$. Even retaining the faithfulness requirement, these discrete groups $\Gamma=\langle g,h\rangle$ form genuinely infinite and much less rigid families, so a complete classification is no longer realistic. Instead, we determine the conditions under which a finite graph can occur as the reduced underlying graph of a discrete subgroup with prescribed geometric quantities. Except in the rigid trivalent amalgamated-product case, every admissible graph gives rise to infinitely many reduced graphs of groups with the same quantities. In the special case $d(T_g,A_h)=0$, this criterion can be written explicitly.

This contrast is illustrated by the appearance of higher-index amalgams such as $(4,3)$- and $(6,2)$-amalgams, which indicate the wideness of the general case and explain why the full classification lies beyond the scope of the paper.

While these results are motivated by Rosenberger and Gilman's work on two-generator Fuchsian groups, a different approach by Riley proposed a more geometric semi-decidable algorithm for Fuchsian and Kleininan groups \cite{riley1983applications}. Riley's algorithm decides if any finite subset in $\psl(2,\mathbb{R})$ or $\psl(2,\mathbb{C})$ generates a discrete subgroup. The algorithm is based on Poincar\'e’s Fundamental Polyhedron Theorem, deciding the discreteness by constructing and verifying a \emph{fundamental polytope} for the given subgroup. Both Riley's and Kirschmer-R\"uther's approaches run on a Blum–Shub–Smale (BSS) machine \cite{blum1989theory}, which stores arbitrary real numbers and computes rational functions in a single time step.

As a combinatorial analogue, fundamental domains of discrete $\aut(X)$-subgroups naturally behave like fundamental polytopes, making a generalization of Riley's algorithm feasible. In this paper we discuss such a generalized algorithm, which determines the discreteness of the group from generators in $\aut(X)$ and gives their graph-of-groups presentations. Our algorithm uses a generalized BSS machine, which better handles the tree automorphisms and their actions on vertices and edges.

The paper is organized as follows. In Section \ref{sec:prelim}, we review the Bass–Serre theoretic notions used throughout the paper. In Section \ref{sec:rigidity}, we present our first set of results concerning rigidity for two-generator discrete subgroups of $\aut(X_3)$ with small geometric quantities. In Section \ref{sec:flex}, we present our second set of results, describing the possible underlying graphs in the non-rigid cases. Explicit constructions of infinite families of groups are given in Section \ref{sec:constr}. The generalized Poincar\'e algorithm for tree automorphisms is discussed in Section \ref{sec:Poincare}. In Appendix \ref{sec:program}, we present and explain the \emph{GAP} code of a Stalling folding process used to obtain certain classification results in Section \ref{sec:rigidity}. 

\section{Preliminaries: Graphs of Groups}\label{sec:prelim}
Throughout this paper, we assume that $X$ is a locally finite tree and $\aut(X)$ is its group of automorphisms, such that the quotient $X/\aut(X)$ is a finite graph. Let $\mathcal{V}(X)$ and $\mathcal{E}(X)$ denote the sets of its vertices and edges, respectively.

A tree automorphism is either an \emph{inversion}, \emph{elliptic}, or \emph{hyperbolic}. If $g$ is an elliptic automorphism, its fixed subset $T_g \subset X$ is a subtree of $X$. If $h$ is hyperbolic, then the displacement $d(v, h \cdot v)$ is always strictly positive; it attains a minimum $\ell(h)$, called its \emph{translation length}, precisely on a subset $A_h \subset X$ called its \emph{axis}, which is a bi-infinite geodesic in $X$ \cite{serre2002trees}.

The total automorphism group $\aut(X)$ is equipped with a topology, where the pointwise stabilizers of finite induced subtrees serve as the basis open sets containing $\mathrm{Id}_{\aut(X)}$. This leads to the notion of \emph{discreteness} for $\aut(X)$-subgroups:
\begin{defn}
    A subgroup $G< \aut(X)$ is \emph{discrete} if there exists a finite tree $T\subset X$ such that the pointwise stabilizer $G_T = 1$.
\end{defn}
In particular, if $G<\aut(X)$ is discrete and finitely generated, then all stabilizer subgroups $G_v$ and $G_e$ are finite, and their orders are uniformly bounded.

Unless specified, when mentioning any subgroup $G<\aut(X)$ in this paper, we assume that $G$ is discrete and finitely generated, and is inversion-free.

To study the action of a tree automorphism group in an algebraic framework, one introduces the notion of a \emph{graph of groups}\cite{dicks2006groups}:
\begin{defn}
    A \textbf{graph of groups} $(Y,\mathcal{G})$ consists of the following data:
    \begin{itemize}
        \item A connected and symmetric directed graph $Y$ with vertex set $\mathcal{V}(Y)$ and edge set $\mathcal{E}(Y)$.
        \item A family $\mathcal{G}$ of groups, including \emph{vertex groups} $G(v)$ for $v\in\mathcal{V}(Y)$, and \emph{edge groups} $G(e)$ for $e\in\mathcal{E}(Y)$.
        \item The \emph{boundary monomorphisms}
        \[
        \iota_{e,\alpha}: G(e)\to G(\alpha_e),\ \iota_{e,\omega}: G(e)\to G(\omega_e)
        \]
        for each edge $e\in\mathcal{E}(Y)$ and the incident vertices $\alpha_e$ and $\omega_e$.
    \end{itemize}
\end{defn}
We regard a graph of groups as a combinatorial generalization of hyperbolic orbifolds. The \emph{fundamental group} of a graph of groups is defined analogously:
\begin{defn}
    Let $(Y,\mathcal{G})$ be a graph of groups, and fix a spanning tree $Y_0\subset Y$. Its \emph{fundamental group} (or \emph{amalgam}) $\pi_1(Y,\mathcal{G};Y_0)$ is defined as
    \[
    \pi_1(Y,\mathcal{G};Y_0) = \left(\left(\ast_{v\in\mathcal{V}(Y)}G(v)\right)\ast F_{\mathcal{E}(Y)\setminus\mathcal{E}(Y_0)}\right)/\sim,
    \]
    where $F_{\mathcal{E}(Y)\setminus\mathcal{E}(Y_0)}$ is the free group generated by symbols $t_e$ corresponding to (oriented) edges $e\in \mathcal{E}(Y)\setminus\mathcal{E}(Y_0)$, subject to the relations:
    \[
    \iota_{e,\omega}(g) = t_e\iota_{e,\alpha}(g)t_e^{-1}.
    \] 
    One shows that this definition is independent of the choice of the spanning tree $Y_0$, up to a group isomorphism. Thus, the fundamental group is also denoted by $\pi_1(Y,\mathcal{G})$.
\end{defn}
Parallel to the manifold scenario, one considers the \emph{universal covering} (or \emph{Bass-Serre tree}) and \emph{fundamental group} of a graph of groups.
\begin{defn}
    Let $(Y,\mathcal{G})$ be a graph of groups, and denote its fundamental group by $\pi_1(Y,\mathcal{G})$. Define a graph $\widetilde{X} = \widetilde{X}(Y,\mathcal{G})$, called the \emph{Bass-Serre covering tree}, whose vertices and edges are cosets:
    \[
    \mathcal{V}(\widetilde{X}) = \{gG(v)\,|\, g\in \pi_1(Y,\mathcal{G}), v\in \mathcal{V}(Y)\},
    \]
    and
    \[
    \mathcal{E}(\widetilde{X}) =\{gG(e)\,|\, g\in \pi_1(Y,\mathcal{G}), e\in \mathcal{E}(Y)\}.
    \]
    Moreover, the incidence is determined as follows:
    \begin{itemize}
        \item If $e\in \mathcal{E}(Y_0)$,
        \[
        \alpha_{gG(e)} = gG(\alpha_e),\ \omega_{gG(e)} = gG(\omega_e).
        \]
        \item If $e\notin \mathcal{E}(Y_0)$,
        \[
        \alpha_{gG(e)} = gG(\alpha_e),\ \omega_{gG(e)} = (gt_e)G(\omega_e).
        \]
    \end{itemize}
    The resulting graph $\widetilde{X}$ is a tree, and is independent of the choice of the spanning tree up to isomorphism. The group $\pi_1(Y,\mathcal{G})$ acts on $\widetilde{X}$ by left multiplication on cosets.
\end{defn}
\begin{defn}\label{def:fundamental_domain}
    Let $X$ be a tree, and $\Gamma<\aut(X)$ a discrete subgroup. We say that a subset $\mathcal{Y}$ of vertices and edges, together with a family $\mathcal{G}$ of groups, is a \emph{fundamental domain} for $\Gamma$ if the following conditions hold:
    \begin{itemize}
        \item The set $\mathcal{Y}$ is $\Gamma$-transversal, meaning $\mathcal{Y}$ meets each $\Gamma$-orbit of vertex and edge exactly once.
        \item The maximal subtree $Y_0\subset \mathcal{Y}$ contains every vertex in $\mathcal{V}(\mathcal{Y})$ and the origin of every edge in $\mathcal{E}(\mathcal{Y})$.
        \item The vertex groups $G(v)$ and edge groups $G(e)$ are the stabilizers $\Gamma_v$ and $\Gamma_e$ for $v\in \mathcal{V}(\mathcal{Y})$ and $e\in \mathcal{E}(\mathcal{Y})$.
    \end{itemize}
\end{defn}
\begin{prop}
    Let $X$ be a tree, $\Gamma<\aut(X)$ be a discrete subgroup, and $(\mathcal{Y},\mathcal{G})$ be a fundamental domain of $\Gamma$. By identifying the unpaired termini of $\mathcal{Y}$-edges with the $\mathcal{Y}$-vertices of the same $\Gamma$-orbit, we recognize $\mathcal{Y}$ as the quotient graph $Y = X/\Gamma$, and $(\mathcal{Y},\mathcal{G})$ as a graph of groups $(Y,\mathcal{G})$. Moreover:
    \begin{itemize}
        \item The fundamental group $\pi_1(Y,\mathcal{G})$ is isomorphic to $\Gamma$.
        \item The natural inclusion $j: \mathcal{Y}\hookrightarrow X$ extends to an isomorphism from the Bass-Serre covering tree $\widetilde{X}(Y,\mathcal{G})$ to $X$.
    \end{itemize}
\end{prop}
The following examples from \cite{serre2002trees} are key constructions of graph of groups in our paper.
\begin{exm}[Amalgamated Products]
    Let $Y$ be a line segment, with vertices $v^+$, $v^-$ and edge $e$. Define a graph of groups $(Y,\mathcal{G})$ such that $G_1 = \mathcal{G}(v^+)$, $G_2 = \mathcal{G}(v^-)$, and $H = \mathcal{G}(e)$, where $H<G_1$ and $\iota_{e,v^-} = \varphi: H\hookrightarrow G_2$ is an injective homomorphism. Then the fundamental group,
    \[
    \pi_1(Y,\mathcal{G})\cong G_1*_{H} G_2\coloneqq (G_1*G_2)/(h\sim \varphi(h),\ h\in H)
    \]
    is an \emph{amalgamated free product} of $G_1$ and $G_2$ over $H$.

    If both $\lvert G_1: H\rvert = n$ and $\lvert G_2:H\rvert = m$ are finite, then the Bass-Serre covering tree of $(Y,\mathcal{G})$ is an $(n,m)$-biregular tree.
\end{exm}
\begin{exm}[HNN Extensions]
    Let $Y$ be a loop, with vertex $v$ and edge $e$. Define a graph of groups $(Y,\mathcal{G})$ such that $G = \mathcal{G}(v)$ and $H = \mathcal{G}(e)$. Both endpoints of $e$ are equal to $v$, so we can let $H<G$ and $\iota_{e,\alpha}$ be identity. Let $\iota_{e,\omega} = \varphi: H\hookrightarrow G$, then the fundamental group,
    \[
    \pi_1(Y,\mathcal{G})\cong G*_\varphi\coloneqq(G*\langle x\rangle)/(\varphi(h)\sim xhx^{-1},\ h\in H),
    \]
    is called an \emph{HNN extension} of $G$. Here, $x$ is a symbol not in $G$ of infinite order, called the \emph{stable letter}.
    
    If $\lvert G: H\rvert = n$ is finite, then the Bass-Serre covering tree of $(Y,\mathcal{G})$ is an $2n$-regular tree. The stable letter $x$ acts on the covering tree as a hyperbolic automorphism of translation length $\ell(x) = 1$.
\end{exm}

If the quotient $Y = X/\Gamma$ by a subgroup $\Gamma < \aut(X)$ is infinite, it is more convenient to consider its action on a smaller subtree:
\begin{defn}[\cite{dicks2006groups}]
    Let $\Gamma < \aut(X)$ be a subgroup containing a hyperbolic element. The union of all axes $A_g$ of hyperbolic elements $g \in \Gamma$ is a subtree of $X$, called the minimal $\Gamma$-subtree of $X$.
    
    The minimal $\Gamma$-subtree $T$ is the unique minimal subtree preserved by $\Gamma$. Moreover, if $\Gamma$ is finitely generated, the quotient $T/\Gamma$ is a finite graph.
\end{defn}
We call $Y_0 = T/\Gamma$ the \emph{minimal quotient graph} for $\Gamma$, and the fundamental domain $(Y_0, \Gamma)$ (per Definition \ref{def:fundamental_domain}) the \emph{minimal quotient graph of groups} for $\Gamma$.

Intuitively, the minimal $\Gamma$-subtree is obtained by shrinking all subtrees of $X$ on which $\Gamma$ acts trivially relative to their root vertices. The boundary monomorphisms within such subtrees are bijective and subject to reduction without changing the fundamental group. Bijective boundary monomorphisms may persist in the minimal quotient graph of groups of $\Gamma$, suggesting that further reductions are possible:
\begin{defn}
    Let $e$ be a non-loop edge in a graph of groups $(Y, \mathcal{G})$ and let $v_1, v_2$ be the incident vertices. If the edge group $\mathcal{G}(e)$ is isomorphic to the vertex group $\mathcal{G}(v_1)$, we can reduce the graph of groups as follows:
    \begin{itemize}
        \item Shrink the edge $e$ and merge vertices $v_1$ and $v_2$ into a single vertex $v$, with $\mathcal{G}(v) = \mathcal{G}(v_2)$.
        \item The edges incident to $v$ are those that were incident to either $v_1$ or $v_2$ before the reduction (excluding $e$ itself).
        \item For edges incident to $v_1$, the new boundary monomorphism is $\iota_{e', v} = \iota_{e, v} \circ \iota_{e, v_1}^{-1} \circ \iota_{e', v_1}$. For edges incident to $v_2$, the new boundary monomorphism is $\iota_{e', v} = \iota_{e', v_2}$.
    \end{itemize}
    The fundamental group $\pi_1(Y, \mathcal{G})$ remains unchanged after such a reduction. A graph of groups $(Y, \mathcal{G})$ is called \emph{reduced} if it admits no further reductions - that is, if each boundary monomorphism is proper.
\end{defn}
For $\Gamma < \aut(X)$ with a minimal quotient graph of groups $(Y_0, \Gamma)$, any reduction $(Y, \mathcal{G})$ of $(Y_0, \Gamma)$ that is itself reduced is called a \emph{reduced graph of groups} associated with $\Gamma$. The underlying graph $Y$ is called a \emph{reduced graph} associated with $\Gamma$. When the minimal fundamental domain $(Y, \mathcal{G})$ of a subgroup $\Gamma < \aut(X)$ is itself reduced, we simply say that the quotient graph $Y = T/\Gamma$ is a \emph{reduced quotient} of $\Gamma$ when it is unambiguous.

\begin{exm}[Amalgamated products acting non-faithfully on the minimal subtree]\label{exm:non_faithful}
    Let $K$ be a group with $\lvert K\rvert = 2^k$, $k>0$. Let $G$ and $H$ be extensions of $K$ with order $2^{k+1}$. Every $2$-group admits a descending chains of subgroups:
    \[
    G = G_{k+1}\triangleright G_k\triangleright \dots \triangleright G_1\triangleright G_0 = \{1\},\quad H = H_{k+1}\triangleright H_k\triangleright \dots \triangleright H_1\triangleright H_0 = \{1\},\quad \lvert G_i\rvert = \lvert H_i\rvert = 2^i.
    \]
    Consider an amalgamated product:
    \[
    \Gamma = G_0 *_{G_0} G_1*_{G_1}\dots *_{G_k}G *_{G} G*_K H*_{H} H*_{H_k}\dots *_{H_1} H_1 *_{H_0} H_0.
    \]
    Then $\Gamma$ is the fundamental group of a path of groups. From the edge-vertex indices, its Bass-Serre universal covering tree is contained in a trivalent tree, and the existence of trivial edge groups $G_0$ and $H_0$ implies that the action on $X_3$ is faithful. Therefore, $\Gamma$ is a discrete subgroup of $\aut(X_3)$.

    On the other hand, the edge reduction of the aforementioned path of groups yields a path of unit length, corresponding to the fact that $\Gamma\cong G*_K H$. Consequently, the minimal $\Gamma$-subtree is a bi-infinite path, and $K$ is the kernel of the action of $\Gamma$ restricted to this minimal subtree.
\end{exm}
Following a similar argument, one sees that the HNN extension,
\[
\Gamma = G_0 *_{G_0} G_1 *_{G_1} \dots *_{G_k}G *_{G} G*_\varphi,\quad \lvert G\rvert = 2^{k+1},\quad \varphi\in \operatorname{Out}(G),
\]
is isomorphic to $G\rtimes_\varphi \mathbb{Z}$ and is also a non-faithful discrete subgroup of $\aut(X_3)$.

\section{Rigidity of Two-generated Discrete Subgroups in \texorpdfstring{$\aut(X_3)$}{Aut(X3)}}\label{sec:rigidity}
Example \ref{exm:non_faithful} shows that discrete subgroups of tree automorphisms admit far more freedom without a faithful action on the minimal subtree. On the other hand, after taking the quotient by the action kernel, it results in a faithful subgroup, maintaining the key group structure while being subject to more restrictions. In particular, when the underlying tree is trivalent and the geometric quantities for the generators are sufficiently small, a complete classification of two-generator faithful discrete subgroups is possible. This rigidity phenomenon is reflected in the main results below.
\begin{thm}\label{thm:classification}
    Let $g\in\aut(X_3)$ be an elliptic automorphism of the trivalent tree and $h\in \aut(X_3)$ be a hyperbolic automorphism with translation length $\ell(h) = 2$, such that $g$ fixes a vertex of $A_h$. Suppose $\Gamma = \langle g,h\rangle$ is discrete and acts faithfully on the $\Gamma$-minimal subtree of $X_3$. Then, $\Gamma$ must be one of the following:
    \begin{itemize}
        \item The infinite dihedral group $D_\infty = C_2*C_2$,
        \item One of the seven free amalgamated products of Djokovi\'c-Miller type (including the free product $C_3*C_2$),
        \item One of the fifteen free amalgamated products of Goldschmidt type, except for $\mathrm{G}_1^1$ (but including the free product $C_3*C_3$), or
        \item An HNN extension $G*_\varphi$, where $G$ admits the following presentation for any $n\geq 1$ and $0\leq k\leq \lfloor\frac{n}{3}\rfloor$:
        \[
        G = \langle t_1,\dots, t_n\rangle,\ t_i^2 = 1 \text{ for }i=1,\dots, n,\ [t_i,t_{i+j}] = 1\text{ for } 0<j<n-k,
        \]
        and
        \[
        [t_i,t_{i+j}] = \prod_{l=k}^{j-k}t_{i+l}^{\epsilon(j,l)},\ \epsilon(j,l)\in\{0,1\},\text{ for }n-k\leq j\leq n-1.
        \]
        Moreover, the HNN extension is over the isomorphism $\varphi: H_1\to H_2$ for subgroups
        \[
        H_1 = \langle t_1,\dots, t_{n-1}\rangle,\ H_2 = \langle t_2,\dots, t_{n}\rangle,
        \]
        with $\varphi(t_i) = t_{i+1}$. This includes the direct product $C_2\times \mathbb{Z}$ when $n = 1$.
    \end{itemize}
    In particular, the order of $g$ is in $\{2,3,4,6,8,12\}$. The possible isomorphism types for each $\ord(g)$ are given below:
    \begin{center}
        \begin{tabular}{c|p{12cm}}
            $\ord(g)$ & $\Gamma = \langle g,h\rangle$ \\
            \hline
            $2$ & $D_\infty$, $\mathrm{DjM}_1$, $\mathrm{DjM}_2^1$, $\mathrm{DjM}_2^2$, $\mathrm{DjM}_3$, $\mathrm{DjM}_4^1$, $\mathrm{DjM}_4^2$, $\mathrm{DjM}_5$, $\mathrm{G}_1^2$, $\mathrm{G}_1^3$, $\mathrm{G}_2$, $\mathrm{G}_2^1$, $\mathrm{G}_2^2$, $\mathrm{G}_2^3$, $\mathrm{G}_2^4$, $\mathrm{G}_3$, $\mathrm{G}_3^1$, $\mathrm{G}_4$, $\mathrm{G}_4^1$, $\mathrm{G}_5$, $\mathrm{G}_5^1$, and HNN extensions \\
            \hline
            $3$ & $\mathrm{DjM}_1$, $\mathrm{DjM}_2^2$, $\mathrm{DjM}_4^1$, $\mathrm{DjM}_4^2$, $\mathrm{G}_1$, $\mathrm{G}_1^2$, $\mathrm{G}_2$, $\mathrm{G}_2^1$, $\mathrm{G}_2^2$, $\mathrm{G}_2^3$ \\
            \hline
            $4$ & $\mathrm{DjM}_2^2$, $\mathrm{DjM}_3$, $\mathrm{DjM}_4^1$, $\mathrm{DjM}_4^2$, $\mathrm{DjM}_5$, $\mathrm{G}_2^1$, $\mathrm{G}_2^2$, $\mathrm{G}_2^4$, $\mathrm{G}_3$, $\mathrm{G}_3^1$, $\mathrm{G}_4$, $\mathrm{G}_4^1$, $\mathrm{G}_5$, $\mathrm{G}_5^1$, and HNN extensions (with $k>0$) \\
            \hline
            $6$ & $\mathrm{DjM}_3$, $\mathrm{DjM}_5$, $\mathrm{G}_1^2$, $\mathrm{G}_1^3$, $\mathrm{G}_2$, $\mathrm{G}_2^2$, $\mathrm{G}_2^3$, $\mathrm{G}_2^4$, $\mathrm{G}_3^1$, $\mathrm{G}_4$, $\mathrm{G}_4^1$, $\mathrm{G}_5$, $\mathrm{G}_5^1$ \\
            \hline
            $8$ & $\mathrm{DjM}_4^1$, $\mathrm{DjM}_4^2$, $\mathrm{DjM}_5$, $\mathrm{G}_4$, $\mathrm{G}_4^1$, $\mathrm{G}_5$, $\mathrm{G}_5^1$ \\
            \hline
            $12$ & $\mathrm{G}_2^1$, $\mathrm{G}_2^4$, $\mathrm{G}_4$, $\mathrm{G}_4^1$, $\mathrm{G}_5^1$ \\
        \end{tabular}
    \end{center}
\end{thm}
\begin{rmk}
    Djokovi\'c \cite{djokovic1980class} studied faithful $(4,2)$-amalgams; besides a finite number of amalgams of $A_4$- and $S_4$-type, he showed that the remainder is an infinite family of $D_8$-type amalgams. We note that the vertex stabilizers of the HNN extensions in our classification are described in a manner very similar to the edge stabilizers for those $D_8$-type $(4,2)$-amalgams.
\end{rmk}
We also have similar classification results for translation length $\ell(h) = 1$ (allowing edge inversions) and for two-elliptic generating pairs.
\begin{cor}\label{cor:translation_1}
    Let $g\in\aut(X_3)$ be an elliptic automorphism, and let $h\in \aut(X_3)$ be a hyperbolic automorphism with translation length $\ell(h) = 1$, such that $g$ fixes a vertex of $A_h$. Suppose $\Gamma = \langle g,h\rangle$ is discrete (allowing edge inversions) and faithful; then $\Gamma$ must be $D_\infty$ or one of the seven free amalgamated products of Djokovi\'c-Miller type. In particular, the order of $g$ is in $\{2,3,4,6\}$. The possible isomorphism types for each $\ord(g)$ are given below:
    \begin{center}
        \begin{tabular}{c|p{12cm}}
            $\ord(g)$ & $\Gamma = \langle g,h\rangle$ \\
            \hline
            $2$ & $D_\infty$, $\mathrm{DjM}_2^1$, $\mathrm{DjM}_2^2$, $\mathrm{DjM}_3$, $\mathrm{DjM}_4^1$, $\mathrm{DjM}_4^2$, $\mathrm{DjM}_5$\\
            \hline
            $3$ & $\mathrm{DjM}_1$, $\mathrm{DjM}_2^2$, $\mathrm{DjM}_4^1$, $\mathrm{DjM}_4^2$\\
            \hline
            $4$ & $\mathrm{DjM}_4^1$, $\mathrm{DjM}_4^2$, $\mathrm{DjM}_5$\\
            \hline
            $6$ & $\mathrm{DjM}_3$, $\mathrm{DjM}_5$ \\
        \end{tabular}
    \end{center}
\end{cor}
\begin{cor}\label{cor:two_ellip}
    Let $g,h\in\aut(X_3)$ be elliptic automorphisms with $d(T_g,T_h) = 1$. Suppose $\Gamma = \langle g,h\rangle$ is discrete and faithful; then $\Gamma$ must be one of the following:
    \begin{itemize}
        \item The infinite dihedral group,
        \item One of the seven free amalgamated products of Djokovi\'c-Miller type, except for $\mathrm{DjM}_2^1$, or
        \item One of the fifteen free amalgamated products of Goldschmidt type, except for $\mathrm{G}_1^1$.
    \end{itemize}
    In particular, the order of any elliptic generator here is in $\{2,3,4,6,8,12\}$. Possible isomorphism types for each pair of orders are provided in the appendix.
\end{cor}
Applying Theorem 3.3 in \cite{Du_Hersonsky_Tree_Aut}, we also obtain a classification result for certain two-hyperbolic discrete subgroups of $\aut(X_3)$:
\begin{cor}
    Let $g,h\in\aut(X_3)$ be hyperbolic automorphisms, and suppose $\Gamma = \langle g,h\rangle$ is discrete and faithful. Let the translation lengths be $\ell(g)$, $\ell(h)$, and let $l = \ell(A_g\cap A_h)<\infty$.
    \begin{itemize}
        \item If $\gcd(\ell(g),\ell(h)) = 2$ and $l\geq \ell(g)+\ell(h) - 2$, then $\Gamma$ is isomorphic to one of the amalgamated products or HNN extensions in Theorem \ref{thm:classification}.
        \item If $\gcd(\ell(g),\ell(h)) = 1$ and $l\geq \ell(g)+\ell(h) - 1$, then $\Gamma$ is isomorphic to one of the amalgamated products in Corollary \ref{cor:translation_1}.
    \end{itemize}
    In addition, by testing all elliptic-hyperbolic generating pairs, we further conclude:
    \begin{itemize}
        \item If $\gcd(\ell(g),\ell(h)) = 2$, then either $\Gamma$ is isomorphic to one of the HNN extensions, or $l\leq \ell(g)+\ell(h) + 6$. Equality holds only for amalgamated products of types $\mathrm{G}_5^1$ and $\mathrm{DjM}_5$.
        \item If $\gcd(\ell(g),\ell(h)) = 1$, then $l\leq \ell(g)+\ell(h) + 3$. Equality holds only for the amalgamated product of type $\mathrm{DjM}_5$ that contains an inversion.
    \end{itemize}
\end{cor}
The proof of Theorem \ref{thm:classification} follows from a combination of the following: in Subsection \ref{subsec:HNN}, we classify all possible HNN extensions generated by the pairs in the theorem. In Subsection \ref{subsec:3_3_amalgam}, we refer to the theorems by Djokovi\'c and Miller \cite{djokovic1977regular} and by Goldschmidt \cite{goldschmidt1980automorphisms}, which provide finite lists of possible $(3,2)$- and $(3,3)$-amalgamated products. Using a computer program (see the Appendix) and several manual verifications, we determined whether these amalgamated products are generated by elliptic-hyperbolic or two-elliptic pairs for each elliptic order.
\subsection{Rigidity in HNN extensions}\label{subsec:HNN}
In this subsection, we classify the two-generator discrete HNN extensions from Theorem \ref{thm:classification}, which form a single infinite family. In particular, the order of every elliptic automorphism in such HNN extensions is at most $4$.

Suppose that $n\geq 1$ and $0\leq k\leq \lfloor\frac{n}{3}\rfloor$. Let parameters $\epsilon(j,l)\in \{0,1\}$ for $n-k\leq j\leq n-1$, $k\leq l\leq j-k$, and $\sum_{l}\epsilon(n-k,l)>0$. Define a group with presentation:
\[
G_{n,k,\vec{\epsilon}} = \langle t_1,\dots, t_n\mid t_i^2, [t_i,t_{i+j}],\ 0<j<n-k,\ [t_i,t_{i+j}]t_{i+k}^{\epsilon(j,k)}\dots t_{i+j-k}^{\epsilon(j,j-k)},\ n-k\leq j\leq n-1\rangle.
\]
In particular, $G_{n,0,\varnothing}\cong (C_2)^n$.

The group $G$ has a Heisenberg-like structure: $G = \langle G_p,G_z,G_q\rangle$, where $G_p = \langle t_1,\dots, t_k\rangle\cong (C_2)^k$, $G_z = \langle t_{k+1},\dots, t_{n-k}\rangle\cong (C_2)^{n-2k}$, and $G_q = \langle t_{n-k+1},\dots, t_n\rangle\cong (C_2)^k$. Furthermore, $[G_p,G_z] = [G_z,G_q] = \{1\}$, and $[G_p,G_q] \leq G_z$. Notably, the group is equivalent to the group $N$ in \cite{djokovic1980class} with the change of variables $k = n-m$ and $j$ as their $j-i$.

Define subgroups $H_i = \langle t_i,\dots, t_{i+n-2}\rangle$ for $i \in \{1,2\}$, and an isomorphism $\varphi:H_1\to H_2$ given by $\varphi(t_i) = t_{i+1}$ for $i=1,\dots, n-1$. This defines an HNN extension $\Gamma = \Gamma_{n,k,\vec{\epsilon}}\coloneqq G_{n,k,\vec{\epsilon}}*_\varphi$. Since $H_1$ is core-free in $\Gamma$, the group $\Gamma$ acts faithfully on its Bass-Serre tree. Additionally, $\Gamma$ is generated by $g = t_1$ and the HNN stable letter $h$ with translation length $\ell(h) = 1$.

The fundamental graph of groups of $\Gamma$ is a loop, corresponding to a $4$-valent covering tree. If we replace the vertex of this loop with an edge having two distinct endpoints, we obtain a trivalent covering tree; in this case, the translation length becomes $\ell(h) = 2$.
\begin{exm}
    Let $n\geq 3$, $k=1$, $\epsilon(n-1,1) = 1$, and $\epsilon(n-1,l) = 0$ for $l>1$. Then $G\cong D_8\times (C_2)^{n-3}$ and $H_1,H_2\cong (C_2)^{n-1}$. The group $\Gamma$ is generated by the stable letter $h$ and $g=t_1t_{n-1}$, satisfying $\ord(g) = 2$ and $\ell(T_g\cap A_h) = 1$ (or $3$ in the trivalent case). This yields the minimal values of $\ord(g)$, $\ell(h)$, and $\ell(T_g\cap A_h)$ such that there exist discrete subgroups $\langle g,h\rangle < \aut(X_4)$ generated by an elliptic-hyperbolic pair with arbitrarily large vertex stabilizer orders.
\end{exm}
The vertex stabilizer $G$ can be more complex than a simple direct product of copies of $D_8$ and $C_2$:
\begin{exm}
    Let $n=7$, $k=2$, $\epsilon(5,2) = \epsilon(6,4) = 1$, and all other $\epsilon(j,l) = 0$. The resulting vertex group $G$, identified as \texttt{SmallGroup(128,1135)} in \textit{GAP}, is not a direct product of $D_8$'s and $C_2$'s. Its rank is $4$, while any such direct product of order $2^7 = 128$ must have rank at least $5$.
\end{exm}
Our rigidity theorem states that this construction exhausts all faithful $2$-HNN extensions of finite groups.
\begin{thm}\label{thm:HNN_rigidity}
    Suppose $\Gamma = G*_{\varphi}$ is a faithful $2$-HNN extension of a finite group $G$—that is, $H_1,H_2 < G$ with $\lvert G:H_1\rvert = \lvert G:H_2\rvert = 2$, $\varphi: H_1 \to H_2$ is an isomorphism, and no non-trivial subgroup of $H_1 \cap H_2$ is preserved by $\varphi$. Then $\Gamma$ is isomorphic to an HNN extension constructed as above. In particular, every element in $G$ has order at most $4$.
\end{thm}
We need a few lemmas to proceed:
\begin{lem}\label{lem:pres}
    The group $G$ satisfying the requirements in Theorem \ref{thm:HNN_rigidity} is a $2$-group. In addition, if $\lvert G\rvert = 2^n$, then $G$ is generated by $n$ involutions $t_1,\dots, t_n$, such that the HNN isomorphism satisfies $\varphi(t_i) = t_{i+1}$.
\end{lem}
\begin{proof}
    Suppose a group $G$ together with subgroups $H_1,H_2$ of index $2$ and an isomorphism $\varphi: H_1\to H_2$ defines a faithful HNN extension $\Gamma\coloneqq G*_\varphi$. Then $\Gamma$ is the fundamental group of the $2$-circuit with vertex groups and an edge group $G$, and the other edge group isomorphic to $H_1$, and $\varphi$ corresponds to the HNN stable letter. Primitivity requires that $G$ be a finite $\aut(X_3)$-edge stabilizer subgroup, so $G$ is a $2$-group.

    Inductively define
    \[
    K_0 = G,\quad K_i = \varphi(H_1\cap K_{i-1}).
    \]
    Then $K_{i+1}\leq K_i$, and
    \[
    \lvert G:K_{i+1}\rvert\leq \lvert G:H_1\rvert\cdot \lvert G:K_i\rvert = 2\lvert G:K_i\rvert.
    \]
    If $K_{i+1} = K_i\neq \{1\}$, then $\varphi$ preserves a non-trivial subgroup of $G$, a contradiction to primitivity. Therefore $K_{i+1}<K_i$ with index $2$; assume that $\lvert G\rvert = 2^n$, then
    \[
    \{1\} = K_n<K_{n-1}<\dots<K_1=H_2<K_0=G.
    \]
    Let $t_n$ be the generator of $K_{n-1}$ of order $2$. For $i=1,\dots, n-1$, define $t_i = \varphi^{i-n}(t_{n})$; it follows that
    \[
    t_i\in K_{i-1},\ t_i\notin K_{i},\ i=1,\dots,n,
    \]
    implying $\lvert\langle t_1,\dots, t_n\rangle\rvert\geq 2^n = \lvert G\rvert$. Since $t_1,\dots, t_n\in G$, we conclude that $G$ admits the generating set $(t_1,\dots, t_n)$ with $\varphi(t_i) = t_{i+1}$.
\end{proof}
From the argument above, we see that
\[
K_i = \langle t_{i+1},\dots,t_{n}\rangle,
\]
is a subgroup of $G$ of order $2^{n-i}$. For $0\leq i< j\leq n$, denote
\[
K_{i,j} = \langle t_{i+1},\dots, t_j\rangle,
\]
then $K_{i,j} = \varphi^{j-n}(K_{i-j+n})$ is a subgroup of order $2^{j-i}$.
\begin{lem}\label{lem:comm}
    Each element in $G$ is uniquely expressed as a product,
    \[
    g = t_{i_1}\dots t_{i_k},\ 1\leq i_1<\dots < i_k\leq n.
    \]
    Moreover, the group $G$ admits commutator relations,
    \[
    [t_i,t_{i+l}] = t_{i+1}^{\epsilon(l,1)}t_{i+2}^{\epsilon(l,2)}\dots t_{i+l-1}^{\epsilon(l,l-1)},
    \]
    for $\epsilon(l,1)$,\dots, $\epsilon(l,l-1)\in \{0,1\}$.
\end{lem}
\begin{proof}
    First, we show that for any $i<j$, the commutator
    \[
    [t_i,t_j]\in K_{i,j-1}.
    \]
    Indeed, $K_{i,j-1}$ is a normal subgroup of index $2$ in both of the following:
    \[
    K_{i,j-1}\vartriangleleft K_{i-1,j-1},\text{ and }K_{i,j-1}\vartriangleleft K_{i,j}.
    \]
    Since $K_{i-1,j-1}K_{i,j} = K_{i-1,j}$, $K_{i,j-1}$ is also normal in $K_{i-1,j}$. Consider the quotient groups,
    \[
    K_{i-1,j-1}/K_{i,j-1}, \quad K_{i,j}/K_{i,j-1}<K_{i-1,j}/K_{i,j-1}.
    \]
    The first two are isomorphic to $C_2$ and generated by $[t_i]$ and $[t_j]$; these elements generate the latter group, which is of order $4$. Therefore $[t_i]$ and $[t_j]$ commute in $K_{i-1,j}/K_{i,j-1}$, so $[t_i,t_j]\in K_{i,j-1}$.

    The fact above implies that $[t_i,t_j]$ is a product of certain letters from $t_{i+1}$ through $t_{j-1}$. Since each $t_i$ is an involution, this allows each element $g\in G$ to be expressed as a product of $t_i$'s with ascending indices. Since there are $2^n$ such words and $\lvert G\rvert = 2^n$, the expressions are unique. Such expressions allow the relators to take the form described in the lemma.
\end{proof}
\begin{proof}[Proof of Theorem \ref{thm:HNN_rigidity}]
    We will prove by induction that the relators of such a group satisfy the presentation in Theorem \ref{thm:HNN_rigidity}. The base case is clear: for orders $2^1$ and $2^2$, $G = C_2$ and $G = V_4$ are the only eligible groups.
    
    For convenience, we assume that the theorem holds for groups of order $2^{n-1}$, and we prove the case $\lvert G\rvert = 2^n$ under this assumption. By Lemma \ref{lem:pres}, the subgroups $H_1 = \langle t_1,\dots, t_{n-1}\rangle$ and $H_2 = \langle t_2,\dots, t_n\rangle < G$ satisfy the assumption in the theorem. Thus, we can assume that there exists $0\leq k_0\leq \frac{n-1}{3}$ such that $t_i$ and $t_{i+j}$ commute for any $j\leq n-2-k_0$, and for $n-1-k_0\leq l\leq n-2$,
    \[
    [t_i,t_{i+l}] = t_{i+k_0}^{\epsilon(l,k_0)}\dots t_{i+l-k_0}^{\epsilon(l,l-k_0)},
    \]
    where $\sum_j\epsilon(n-1-k_0,j)>0$. Denote
    \[
    k_+ = \max_l(\min_{\epsilon(l,k)=1}k),\ k_- = \max_l(l - \max_{\epsilon(l,k) = 1} k),
    \]
    the induction hypothesis implies that $k_+,k_-\geq k_0$, while the case $\lvert G\rvert = 2^n$ further requires that $k_+,k_-\geq k_0+1$.
    
    Suppose that $k_+ = k_0$; then there exists some $l_0\geq n-1-k_0$ such that
    \[
    (t_1t_{l_0+1})^2 = t_{k_1'}\dots t_{k_j'},
    \]
    with $k_0+1=k_1'\leq\dots\leq k_j'\leq l-k_0+1$.
    
    Now, since $H_1$ is a normal subgroup of $G$, we have an automorphism:
    \[
    \psi:H_1\to H_1,\ t\mapsto t_ntt_n.
    \]
    The induction hypothesis implies that $\psi(t_j) = t_j$ for $j\geq k_0+2$, and
    \[
    \psi(t_j) = t_j t_{j+k_0}^{\epsilon(n-j,k_0)}\dots t_{n-k_0}^{\epsilon(n-j,n-j-k_0)},
    \]
    for $j=2,\dots, k_0+1$. In particular, the fact that $\sum_l\epsilon(n-1-k_0,l)>0$ in the definition of $k_0$ implies that $[t_{k_0+1},t_n]$ is non-trivial, or
    \[
    \psi(t_{k_0+1}) = t_{k_0+1}t_{k_1}\dots t_{k_i} \neq t_{k_0+1},
    \]
    with $2k_0+1\leq k_1<\dots< k_i\leq n-k_0$ and $i\geq 1$.
    
    Combine the facts above:
    \[
    \begin{split}
        & (\psi(t_1)t_{l_0+1})^2 = (\psi(t_1)\psi(t_{l_0+1}))^2 = \psi((t_1t_{l_0+1})^2) \\
        = & \psi(t_{k_0+1} t_{k_2'}\dots t_{k_j'}) = \psi(t_{k_0+1})t_{k_2'}\dots t_{k_j'} \\
        \neq & t_{k_0+1}t_{k_2'}\dots t_{k_j'} = (t_1t_{l_0+1})^2.
    \end{split}
    \]
    While the induction assumption does not allow $\psi(t_1)$ to fit into the pattern, Lemma \ref{lem:comm} guarantees that $t_1\psi(t_1)$ can be expressed as a word in the letters $t_2,\dots,t_{n-1}$ in ascending order. If $t_1\psi(t_1)$ is a word involving only the letters $t_{k_0+1},\dots,t_{n-1}$, then the commutativity of $t_{l_0+1}$ with $t_1\psi(t_1)$ implies that $(\psi(t_1)t_{l_0+1})^2 = (t_1t_{l_0+1})^2$, a contradiction. Therefore, if we let
    \[
        m = \min_{\epsilon(n-1,k)=1}k,
    \]
    then $m\leq k_0-1$, that is, $m+n-k_0\leq n-1$. The commutator $[t_1,t_{m+n-k_0}]$ can be expressed as a word in the letters $t_{k_0+1},\dots,t_{n-k_0-1}$.

    If this word expression of $[t_1,t_{m+n-k_0}]$ does not contain $t_{k_0+1}$, then it is fixed by the automorphism $\psi$. The element $t_{m+n-k_0}$ is also in this fixed-point subgroup:
    \[
        [t_1,t_{m+n-k_0}] = \psi([t_1,t_{m+n-k_0}]) = [\psi(t_1),\psi(t_{m+n-k_0})] = [\psi(t_1),t_{m+n-k_0}].
    \]
    By the assumption, $\psi(t_1) = t_1t_{m+1}\prod_{j\geq m+1}t_{j+1}^{\epsilon(n-1, j)}$, and $t_{m+n-k_0}$ commutes with every such $t_{j+1}$ as $j+1\geq m+2$. As a result,
    \[
        [t_{m+1},t_{m+n-k_0}] = 1,
    \]
    contradicting the fact that $\sum_l\epsilon(n-1-k_0,l)>0$.
    
    For the case that $t_{k_0+1}$ is contained in the commutator $[t_1,t_{m+n-k_0}]$, we need a further lemma:
    \begin{lem}\label{lem:not_contain}
        In the setting above, we have $\epsilon(n-k_0-1, n-2k_0-1) = 0$. That is, $t_{n-k_0}$ does not appear as a letter in the word expression of $\psi(t_{k_0+1})$. 
    \end{lem}
    \begin{proof}
        Since $t_1$ commutes with $t_{k_0+1}$, we have
        \[
            [\psi(t_1),\psi(t_{k_0+1})] = \psi([t_1,t_{k_0+1}]) = 1.
        \]
        If $t_{n-k_0}$ appears in the word expression of $\psi(t_{k_0+1})$, then any other factor lies between $t_{k_0+1}$ and $t_{n-k_0-1}$, hence lies in the center of $H_1 = \langle t_1,\dots, t_{n-1}\rangle$ and therefore commutes with $\psi(t_1)$. Therefore,
        \[
        [\psi(t_1),t_{n-k_0}] = 1.
        \]
        However, every factor in $\psi(t_1)$ except for $t_1$ itself commutes with $t_{n-k_0}$, implying $[t_1,t_{n-k_0}] = 1$, which contradicts the definition of $k_0$.
    \end{proof}
    Returning to the proof of Theorem \ref{thm:HNN_rigidity}. Since $t_{k_0+1}$ plays a role in $[t_1,t_{m+n-k_0}]$, a slightly different relation holds:
    \[
        [\psi(t_1),t_{m+n-k_0}] = \psi([t_1,t_{m+n-k_0}]) = t_{k_0+1}\psi(t_{k_0+1})[t_1,t_{m+n-k_0}].
    \]
    Similar to the previous case, the fact that $t_{m+n-k_0}$ commutes with every factor in $\psi(t_1)$ except for $t_1$ and $t_{m+1}$ implies
    \[
        [t_{m+1},t_{m+n-k_0}] = t_1t_{k_0+1}\psi(t_{k_0+1})t_1.
    \]
    However, Lemma \ref{lem:not_contain} implies that the index of each letter in the word expression of $\psi(t_{k_0+1})$ is strictly less than $n-k_0$, so $\psi(t_{k_0+1})$ commutes with $t_1$. Hence,
    \[
    [t_{m+1},t_{m+n-k_0}] = t_{k_0+1}\psi(t_{k_0+1}) = [t_{k_0+1},t_n],
    \]
    or
    \[
    \prod_{j=k_0}^{n-2k_0-1}t_{m+1+j}^{\epsilon(n-k_0-1,j)} = \prod_{j=k_0}^{n-2k_0-1}t_{k_0+1+j}^{\epsilon(n-k_0-1,j)}
    \]
    which contradicts the definition of $k_0$, the fact that $m<k_0$, and the unique ascending-letter expression of elements in $G$. This shows that $k_+\geq k_0+1$. A similar argument for the $t_1$-conjugation automorphism on $H_2$ implies that $k_-\geq k_0+1$.
    
    It remains to show that the commutator $[t_1,t_n] = (t_1t_n)^2$ also fits the pattern of the other commutators; that is, it is a word in ascending letters among $t_{k_0+2}$ through $t_{n-k_0-1}$. As mentioned earlier, a rough characterization by Lemma \ref{lem:comm} implies that $[t_1,t_n]$ is a word in letters from $t_2$ to $t_{n-1}$.

    First, suppose the word expression of $[t_1,t_n]$ contains letters from $t_2$ to $t_{k_0}$. In the definition of $m = \min_{\epsilon(n-1,k) = 1} k$ above, this implies $m\leq k_0 - 1$. The fact that $k_+\geq k_0+1$ implies that the word expression of $[t_1,t_{m+n-k_0}]$ does not contain the letter $t_{k_0+1}$, leading to the same contradiction that $[t_{m+1},t_{m+n-k_0}] = 1$. Therefore, $[t_1,t_n]$ does not contain letters from $t_2$ to $t_{k_0}$, and similarly, it does not contain letters from $t_{n-k_0+1}$ to $t_{n-1}$.
    
    Now suppose $[t_1,t_n]$ contains the letter $t_{k_0+1}$, that is,
    \[
        \psi(t_1) = t_1t_{k_0+1}w,
    \]
    for some word $w$ in letters with indices $\geq k_0+2$ and $<n$. Since conjugation by $t_n$, namely $\psi$, has order $2$, we have
    \[
        t_1 = \psi^2(t_1) = \psi(t_1)\psi(t_{k_0+1})\psi(w) = t_1t_{k_0+1}wt_{k_0+1}w'w,
    \]
    where $w' = t_{k_0+1}\psi(t_{k_0+1})$ is non-identity by the definition of $k_0$. Now, the fact that the word $w$ has indices between $k_0+2$ and $n-1$ implies that $w^2 = 1$ and $[w,t_{k_0+1}] = 1$, so
    \[
        1 = t_{k_0+1}wt_{k_0+1}w'w = ww'w\implies w' = w^2 = 1,
    \]
    a contradiction. Similarly, $t_{n-k_0}$ is not in the word expression of $[t_1,t_n]$, proving our claim.
    
    It is now straightforward that the square of every element in $G$ is contained in the Abelian subgroup $\langle t_{k+1},\dots, t_{n-k}\rangle$, and the fourth power equals the identity.
\end{proof}
Any faithful $2$-HNN extension $\Gamma_{n,k,\vec{\epsilon}}$ is two-generated: denote the stable letter by $x$. Then $\Gamma = \langle t_1, x\rangle$, where $x$ is hyperbolic and $t_1$ is elliptic of order $2$.

The group $\Gamma_{n,0,\varnothing}$ does not contain an elliptic element of order $4$, so it falls only into the case $\ord(g)=2$ in Theorem \ref{thm:classification}. By contrast, for any $k>0$ (and $\vec{\epsilon}\neq \vec{0}$), $\Gamma_{n,k,\vec{\epsilon}}$ admits an elliptic-hyperbolic generating pair with elliptic order $4$.
\begin{prop}
    Let $\Gamma = \Gamma_{n,k,\vec{\epsilon}} = G_{n,k,\vec{\epsilon}}*_\varphi$ as defined earlier, where $k>0$ and $\epsilon(n-k,l)=1$ for some $k\leq l\leq n-2k$. Then $\Gamma$ is generated by the stable letter $x$ and an elliptic element $g\in G$, with $\ord(g) = 4$.
\end{prop}
\begin{proof}
    Using the same notation as in Theorem \ref{thm:HNN_rigidity}, suppose that
    \[
    [t_1,t_{1+n-k}] = t_{1+k_1}\dots t_{1+k_m},
    \]
    where $m\geq 1$ and $k\leq k_1<\dots <k_m\leq n-2k$. Since $[t_i,t_{i+j}] = 1$ for any $j< n-k$, it follows that
    \[
    (t_1t_{1+k_2-k_1}\dots t_{1+k_m-k_1}t_{1+n-k})^2 = (t_1t_{1+n-k})^2 = t_{1+k_1}\dots t_{1+k_m}.
    \]
    On the other hand,
    \[
    \varphi^{-k_1}(t_{1+k_1}\dots t_{1+k_m}) = t_1t_{1+k_2-k_1}\dots t_{1+k_m-k_1},
    \]
    so we obtain that the stable letter $x$, together with $g = t_1t_{1+k_2-k_1}\dots t_{1+k_m-k_1}t_{1+n-k}$, generates $\Gamma$. The equations above imply that $\ord(g) = 4$.
\end{proof}

\subsection{Two-generator \texorpdfstring{$(3,2)$}{(3,2)}- and \texorpdfstring{$(3,3)$}{(3,3)}-amalgamated products}\label{subsec:3_3_amalgam}
According to Djokovi\'c-Miller \cite{djokovic1977regular} and Goldschmidt \cite{goldschmidt1980automorphisms}, there are only finitely many amalgamated products that satisfy the requirements needed for Theorem \ref{thm:classification}:
\begin{thm}\label{thm:3_3_amalgam}
    Let $\Gamma = G_1*_H G_2$ be an amalgamated product of finite groups satisfying the following:
    \begin{itemize}
        \item The indices $\lvert G_1: H\rvert$ and $\lvert G_2: H\rvert$ are $2$ or $3$.
        \item No non-trivial normal subgroup of $\Gamma$ is contained in $H$: $\bigcap_{\gamma\in\Gamma} \gamma^{-1}H\gamma = \{1\}$.
    \end{itemize}
    Then, the pair $(G_1,G_2)$ is one of the following:
    \begin{itemize}
        \item If $\lvert G_1: H\rvert = \lvert G_2: H\rvert = 2$, then $(G_1,G_2) = (C_2,C_2)$.
        \item If $\lvert G_1: H\rvert = 3$ and $\lvert G_2: H\rvert = 2$, then $(G_1,G_2)$ is one of the seven pairs of Djokovi\'c-Miller type \cite{djokovic1977regular}.
        \item If $\lvert G_1: H\rvert = \lvert G_2: H\rvert = 3$, then $(G_1,G_2)$ is one of the fifteen pairs of Goldschmidt type \cite{goldschmidt1980automorphisms}.
    \end{itemize}
\end{thm}
Throughout this subsection, we realize the amalgamated products using the presentations given in \cite{conder2025edge} (Sections 2.1 and 2.2). In the Djokovi\'c-Miller cases, the letters $h$ and $a$ in the group presentations are replaced with $x$ and $y$ to avoid ambiguity with the usage of $h$ in this paper and to remain consistent with the presentations in the Goldschmidt cases.

Among these types, $\mathrm{G}_1^1$ is distinguished by the fact that it is not two-generated.
\begin{prop}\label{prop:g1_1}
    The rank of $D_6*_{C_2}D_6$ of type $\mathrm{G}_1^1$ is three; thus, it is neither elliptic-hyperbolic generated nor two-elliptic generated.
\end{prop}
\begin{proof}
    The amalgamated product
    \[
    D_6*_{C_2}D_6 = \langle c,x,y\mid c^2, x^3,y^3, (cx)^2, (cy)^2\rangle
    \]
    is clearly generated by $c$, $x$, and $y$. On the other hand, it admits a surjective homomorphism
    \[
    D_6*_{C_2}D_6\twoheadrightarrow (C_3\times C_3)\rtimes C_2 = \langle x,y,c\mid x^3,y^3,[x,y],c^2,(cx)^2,(cy)^2\rangle,
    \]
    where $x$ and $y$ map to the generators of the two $C_3$ factors, and $c$ maps to the generator of $C_2$. The resulting group is the generalized dihedral group of $C_3\times C_3$, and it is not two-generated. Indeed, the elements in $(C_3\times C_3)\rtimes C_2 - (C_3\times C_3)$ are involutions. For a pair of elements $g,h\in C_3\times C_3$, the subgroup $\langle g,h\rangle$ lies in $C_3\times C_3$. For $g\in C_3\times C_3$ and $h$ an involution, the subgroup $\langle g,h\rangle\cong D_6 = C_3\rtimes C_2$. For $g,h$ two different involutions, we have $gh\in C_3\times C_3$, so again $\langle g,h\rangle\cong D_6$.
    
    Consequently, two generators are insufficient to generate $(C_3\times C_3)\rtimes C_2$. Since the amalgamated product $D_6*_{C_2}D_6$ admits a surjective homomorphism onto $(C_3\times C_3)\rtimes C_2$, it also requires at least $3$ generators.
\end{proof}
Another restriction on two-elliptic generating pairs comes from their orders. Viewing each amalgamated product as a discrete subgroup acting on its Bass-Serre tree (either $X_3$ or $X_{2,3}$), if a group admits an elliptic generating pair $(g,h)$ with $d(T_g,T_h) = 1$, then edge-transitivity implies that, up to conjugation and swapping $g$ and $h$, we have $g\in G_1$ and $h\in G_2$. On the one hand, after adopting the convention of swapping the two factor groups in the five amalgams of $\mathrm{G}_2$-type, the orders of elements in those $G_1$'s and $G_2$'s imply that $\ord(g)\in\{2,3,4,6,8\}$ and $\ord(h)\in\{2,3,4,6,8,12\}$. On the other hand, we have the following result:
\begin{prop}[\cite{conder2022discrete}, Theorem C, case (3)]\label{prop:prime_order}
    Let $g,h$ be elliptic tree automorphisms of prime orders $p$ and $q$ such that the fixed trees $T_g$ and $T_h$ are disjoint. Then $\Gamma = \langle g,h\rangle$ is isomorphic to the free product $C_p*C_q$. The vertex stabilizers in $T_g$ and $T_h$ are $\langle g\rangle\cong C_p$ and $\langle h\rangle\cong C_q$, respectively.
\end{prop}
When $\Gamma$ is not the free product $C_3*C_2$ of type $\mathrm{DjM}_1$ or $C_3*C_3$ of type $\mathrm{G}_1$, the proposition above implies that either $\ord(g)\geq 4$ or $\ord(h)\geq 4$. This further narrows the candidate order pairs $(\ord(g),\ord(h))$ for the generators. In particular, the group $D_6*_{C_2}V_4$ of type $\mathrm{DjM}_2^1$ is not two-elliptic generated, as neither $D_6$ nor $V_4$ contains elements of order greater than $3$.

It is worth considering, for each of the remaining eighteen amalgams $\Gamma = G_1*_HG_2$ and for each pair of orders $(n,m)$—where $n\in \{2,3,4,6,8\}$, $m\in \{2,3,4,6,8,12\}$, and $\max\{n,m\}\geq 4$—whether there exist $g\in G_1$ and $h\in G_2$ such that $(\ord(g),\ord(h)) = (n,m)$ and $\langle g,h\rangle = \Gamma$. In the appendix, we provide a computer program that tests whether any candidate order pair of generators is attained.

As described in another part of the appendix, elliptic-hyperbolic generating pairs with $\ell(h) = 2$ and $d(T_g,A_h) = 1$ can also be converted into a finite search. Nevertheless, for our elliptic order problem, many of the possible values of $\ord(g)$ in elliptic-hyperbolic generating pairs are already covered by the two-elliptic case: if $(g,h)$ is a two-elliptic generating pair with $d(T_g,T_h) = 1$, then $(g,gh)$ is an elliptic-hyperbolic generating pair with translation length $\ell(gh) = 2$. It turns out that this determines the answer to the elliptic-hyperbolic generating-pair existence problem for almost every value of $\ord(g)$ and for every amalgamated product, except for:
\begin{enumerate}
    \item $\ord(g) = 2$ and $\Gamma$ of types $\mathrm{DjM}_2^1$, $\mathrm{G}_2$, and $\mathrm{G}_3$,
    \item $\ord(g) = 3$ and $\Gamma$ of types $\mathrm{DjM}_2^1$, $\mathrm{DjM}_3$, $\mathrm{DjM}_5$, $\mathrm{G}_1^3$, and $\mathrm{G}_2^4$, and
    \item $\ord(g) = 6$ and $\Gamma$ of type $\mathrm{G}_2^1$.
\end{enumerate}
A manual case study would be more effective for treating these remaining cases. Computation suggests that generating pairs in case (1) are still possible:
\begin{prop}\label{prop:ell_hyp}
    There exist elliptic-hyperbolic pairs $g,h\in\aut(X_3)$ with $\ord(g) = 2$, $\ell(h) = 2$, and $d(T_g, A_h) = 0$ that generate the amalgamated products of types $\mathrm{DjM}_2^1$, $\mathrm{G}_2$, and $\mathrm{G}_3$.
\end{prop}
\begin{proof}
    \textbf{Case $\mathrm{DjM}_2^1$.} Use the standard presentation,
    \[
    D_6*_{C_2}V_4 = \langle x,p,y\mid x^3,p^2,y^2,(xp)^2,(yp)^2\rangle,
    \]
    Then $(p,xy)$ is an eligible generating pair:
    \begin{center}
        \begin{tabular}{c||c|c}
            Element & $x$ & $y$ \\
             \hline
            Recovery & $p(xy)p(xy)^{-1}$ & $x^{-1}(xy)$
        \end{tabular}
    \end{center}
    
    \textbf{Case $\mathrm{G}_2$.} Use the standard presentation,
    \[
    D_{12}*_{V_4}A_4 = \langle c, d, x, y \mid c^2, d^2, [c, d], x^3, y^3,(cx)^2, [d, x],(cy)^3, dy^{-1}cy \rangle.
    \]
    Then $(g,h) = (d,xy)$ is an eligible generating pair:
    \begin{center}
        \begin{tabular}{c||c|c|c}
            Element & $c$ & $x$ & $y$ \\
             \hline
            Recovery & $(xy)^{-1}d(xy)d$ & $c(xy)d(xy)^{-1}$ & $x^{-1}(xy)$
        \end{tabular}
    \end{center}
    \textbf{Case $\mathrm{G}_3$.} Use the standard presentation,
    \[
    S_4*_{D_8}S_4 = \langle c, d, x, y \mid c^2, d^4, (cd)^2, x^3,(dx^{-1})^2, cxdx, y^3,(dy^{-1})^2, cdydy \rangle.
    \]
    Then $(g,h) = (c,yx)$ is an eligible generating pair:
    \begin{center}
        \begin{tabular}{c||c|c|c}
            Element & $d$ & $x$ & $y$ \\
             \hline
            Recovery & $c(yx)^{-1}c(yx)$ & $d(xy)d^2(xy)^{-1}$ & $x^{-1}(xy)$
        \end{tabular}
    \end{center}
\end{proof}
By contrast, the conditions in cases (2) and (3) cannot be satisfied.
\begin{prop}\label{prop:ell_hyp_3}
    There are no elliptic-hyperbolic pairs $g,h\in\aut(X_3)$ with $\ord(g) = 3$ (or $\ord(g) = 6$), $\ell(h) = 2$, and $d(T_g, A_h)=0$ that generate the amalgamated products of types $\mathrm{DjM}_2^1$, $\mathrm{DjM}_3$, $\mathrm{DjM}_5$, $\mathrm{G}_1^3$, and $\mathrm{G}_2^4$ (or type $\mathrm{G}_2^1$, respectively).
\end{prop}
\begin{proof}
    Suppose such a pair $(g,h)$ exists. Since $\ord(g) = 3$ (or $6$, in the $\mathrm{G}_2^1$ case), $g$ fixes only a single vertex on $A_h$ and permutes the three adjacent vertices. Consequently, either $gh$ or $gh^{-1}$ fixes a vertex adjacent to $T_g$. Thus, $(g,gh)$ (or $(g,gh^{-1})$) is a two-elliptic generating pair of the amalgamated product, with $d(T_g,T_{gh}) = 1$ and $\ord(g) = 3$. However, these two-elliptic scenarios were excluded earlier by our program.
\end{proof}
This completes the discussion of the existence of elliptic-hyperbolic generating pairs of every elliptic order for every faithful amalgamated product.
\subsection{Proof of Theorem \ref{thm:classification}}
We first prove the classification theorem:
\begin{proof}[Proof of Theorem \ref{thm:classification}]
    The group $\Gamma$ acts on the subforest $\bigcup_{\gamma\in\Gamma}\gamma.A_h\subset X_3$. Since $d(T_g,A_h) = 0$, it is indeed a subtree. Its quotient graph by $\Gamma$ is the circuit $A_h/\langle h\rangle$ of length $\ell(h) = 2$, up to further edge identifications. Therefore, it is either a circuit of length $2$, or an edge with two vertices.

    If the quotient graph is a loop of length $2$, then the indices of the edge groups contained in each vertex group must be either $(1,1)$ or $(1,2)$: their sum is the vertex valency of the universal covering tree, which is $\leq 3$. If both indices are equal to $1$, then the fundamental group action would not be faithful. Therefore, the indices must be $1$ and $2$, respectively. After reducing the edge with index $1$ and merging the two vertices, the result is a loop with an edge group of index $2$. Consequently, the fundamental group is a faithful $2$-HNN extension $G*_\varphi$. By Theorem \ref{thm:HNN_rigidity}, it is one of the HNN extensions in the family, and the order of each nontrivial element in $G$ is either $2$ or $4$.
    
    If the quotient graph is an edge with two vertices, then $\Gamma\cong G_1*_H G_2$, the amalgamated product of the vertex groups $G_1$ and $G_2$ over the edge group $H$. Furthermore, these groups satisfy:
    \begin{itemize}
        \item The indices $\lvert G_1: H\rvert, \lvert G_2: H\rvert\leq 3$, as the underlying tree is trivalent.
        \item The edge group acts faithfully, so it does not contain any subgroup normal in $\Gamma$.
    \end{itemize}
    By Theorem \ref{thm:3_3_amalgam}, the amalgam of the vertex groups $G_1*_H G_2$ is either the infinite dihedral group $C_2*C_2 = D_\infty$, or one of those outlined in \cite{djokovic1977regular} and \cite{goldschmidt1980automorphisms}. The amalgamated product of type $\mathrm{G}_1^1$ is ruled out by Proposition \ref{prop:g1_1}. The two-elliptic generating pairs provided in the appendix give elliptic-hyperbolic generating pairs of most elliptic orders. Then Proposition \ref{prop:ell_hyp} fills in a few of the gaps, while Proposition \ref{prop:ell_hyp_3} implies that the remaining cases are not possible.
\end{proof}
Next, we consider the case $\ell(h) = 1$:
\begin{proof}[Proof of Corollary \ref{cor:translation_1}]
    If $\ell(h) = 1$ instead, then the group $\Gamma$ contains edge inversions. After an edge barycentric division, the quotient graph is a ``half-edge'', and the index $\lvert G_2: H\rvert = 2$ for the barycenter stabilizer $G_2$. This narrows the possibilities to either the infinite dihedral group, or the amalgamated products of the Djokovi\'c-Miller types, where the elliptic generator must come from the group $G_1$ with $\lvert G_1: H\rvert = 3$. The possible groups follow from this.
\end{proof}
Finally, we turn to the two-elliptic case.
\begin{proof}[Proof of Corollary \ref{cor:two_ellip}]
    If both $g$ and $h$ are elliptic automorphisms with $d(T_g,T_h) = 1$, then the group $\Gamma = \langle g,h\rangle$ acts on the subtree $\bigcup_{\gamma\in\Gamma}\gamma.[u,v]$, where $u,v$ are the vertices fixed by $g$ and $h$, with $d(u,v) = 1$. Its quotient graph by $\Gamma$ is an edge with two vertices, representing the equivalence classes of $u$ and $v$. Therefore, $\Gamma$ is one of the faithful $(2,2)$-, $(3,2)$-, or $(3,3)$-amalgamated products. Besides the free product cases, the possible orders $\ord(g)$ and $\ord(h)$ cannot both be prime numbers, as indicated by Proposition \ref{prop:prime_order}. For the other cases, whether such generators exist is discussed in Appendix \ref{sec:program}.
\end{proof}
\subsection{The Lower Bound of Vertex Group Orders}
As an aside in this section, we note that the family of HNN extensions constructed in Theorem \ref{thm:HNN_rigidity} has the ``smallest'' vertex group orders among elliptic-hyperbolic-generated discrete subgroups of tree automorphisms, as the estimate below shows.
\begin{prop}\label{thm:ellip_lower}
    Let $X$ be a tree, $g\in \aut(X)$ be an elliptic automorphism, and $h\in \aut(X)$ be a hyperbolic automorphism such that:
    \[
    \ell(A_h\cap T_g)\geq \ell(h),\ \ell(A_h\cap T_g)<\infty.
    \]
    Then the largest vertex stabilizer of $\Gamma = \langle g,h\rangle$ satisfies the lower bound:
    \[
        \max_{v\in\mathcal{V}(X)} \lvert \Gamma_v\rvert \geq 2^{1+\left\lfloor\frac{\ell(A_h\cap T_g)}{\ell(h)}\right\rfloor},
    \]
    where $\Gamma_v$ denotes the stabilizer of the vertex $v$.
\end{prop}
\begin{proof}
    Let $v$ and $w$ be the endpoints of $A_h\cap T_g$, and let $h$ translate from $v$ to $w$. Let
    \[
    n = \left\lfloor\frac{\ell(A_h\cap T_g)}{\ell(h)}\right\rfloor.
    \]
    Then, for $i=0,1,\dots,n$, $h^i.v$ lies on $[v,w]$. Consider the conjugate elements
    \[
    t_i\coloneqq h^igh^{-i}.
    \]
    Then each $t_i$ is elliptic and fixes the segment $[h^i.v,h^i.w]$ vertex-wise. Define $K_i$ to be the subgroup fixing the segment $[h^i.v,w]$ pointwise:
    \[
    K_i = \{\gamma\in\Gamma\mid \gamma.u = u,\forall u\in\mathcal{V}([h^i.v,w])\}.
    \]
    Then these groups yield a nested sequence of subgroups in $\Gamma$:
    \[
    \{1\}\leq K_0\leq K_1\leq\dots\leq K_n.
    \]
    We observe that for $i=0,\dots,n$, $t_i$ fixes the segment $[h^i.v,w]$, so $t_i\in K_i$. On the other hand, for $i=1,\dots,n$, $t_i$ moves off an edge in $[h^{i-1}.v,h^i.v]$, so $t_i\notin K_{i-1}$. These inclusion relations show that each inclusion $\{1\}<K_0$ and $K_i<K_{i+1}$ is proper. Thus, $\lvert K_i\rvert\geq 2\cdot\lvert K_{i-1}\rvert$, and by induction,
    \[
    \lvert K_n\rvert \geq 2^{n+1}.
    \]

    Furthermore, each group $K_i$, in particular $K_n$, is a subgroup of the vertex stabilizer $\Gamma_w$:
    \[
    \lvert\Gamma_w\rvert \geq \lvert K_n\rvert \geq 2^{n+1}.
    \]
    This completes the proof.
\end{proof}
We now identify the equality case in Proposition \ref{thm:ellip_lower}, which is precisely the HNN family in Theorem \ref{thm:HNN_rigidity}.
\begin{prop}
    If equality in Proposition \ref{thm:ellip_lower} holds for $\Gamma$, then $\Gamma$ is isomorphic to one of the HNN extensions described in Theorem \ref{thm:HNN_rigidity}.
\end{prop}
\begin{proof}
    Suppose $\ell(h) = m$ and $\ell(A_h\cap T_g) = nm+k$, with $0\leq k<m$. For every vertex $u$ in the path $[h^n.v,w]$ of length $k$, the stabilizer
    \[
    \Gamma_u\geq K_n = \langle t_0,\dots, t_n\rangle,
    \]
    where $t_i = h^igh^{-1}$ as above. If equality holds, then $\lvert \Gamma_u\rvert = 2^{n+1}$, so $\Gamma_u = K_n$.
    
    The stabilizer $\Gamma_{h^{-1}.w}$ is $\langle t_{-1},\dots, t_{n-1}\rangle$. The vertex $h^{-1}.w$ is incident with an edge $e$ in $[h^{-1}.w,h^n.v]$, and the stabilizer $\Gamma_e = K_{n-1} = \langle t_{0},\dots, t_{n-1}\rangle$. Conjugation by the HNN stable letter $h$ gives the injection
    \[
    \varphi: \Gamma_e\to \Gamma_w,\ \varphi(t_i) = ht_ih^{-1} = t_{i+1}, \ i=0,\dots, n-1.
    \]
    For every vertex $u$ in $[h^{-1}.w,h^n.v]$, excluding the two endpoints, the vertex stabilizer
    \[
    \Gamma_u\geq K_{n-1} = \langle t_0,\dots, t_{n-1}\rangle,
    \]
    and is either equal to $K_{n-1}$ or an extension of index $2$. Suppose $\lvert \Gamma_u: K_{n-1}\rvert = 2$, $s\in \Gamma_u - K_{n-1}$, and $\Gamma_u = \langle K_{n-1},s\rangle$. Since $s$ is not involved in the HNN injection $\varphi$, the normal form for HNN extensions implies that any word of the form
    \[
    s^{-1}gh^{i_1}gh^{i_2}\dots h^{i_l}g
    \]
    is reduced for $i_1,\dots, i_l\neq 0$. This implies that $s$ cannot be expressed as a word in $g$ and $h$, contradicting our assumption that $s\in \Gamma = \langle g,h\rangle$.

    Therefore, $\Gamma_u = K_{n-1}$ for those vertices, and the quotient graph of the action of $\Gamma$ reduces to a loop of vertex group $K_n$ and edge group $K_{n-1}$, subject to the boundary injection $\varphi(t_i) = ht_ih^{-1} = t_{i+1}$. Its fundamental group is exactly the HNN extension in Theorem \ref{thm:HNN_rigidity}.
\end{proof}
\section{General Two-generated Discrete Subgroups}\label{sec:flex}
General two-generator discrete subgroups of $\aut(X)$ can have complicated vertex stabilizers, even when requiring faithful action on the minimal subtree:
\begin{itemize}
    \item If $\ell(h) = 4$, $d(T_g,A_h) = 0$, and $X$ is trivalent, then $\Gamma = \langle g,h\rangle$ could be an amalgam of three finite groups:
    \[
    \Gamma = G_0 *_{H_1} G_1*_{H_2} G_2,
    \]
    where $\{\lvert G_1:H_1\rvert, \lvert G_2:H_2\rvert\} = \{1,2\}$, and $\lvert G_0:H_1\rvert,\lvert G_2:H_2\rvert\leq 3$.
    \item If $\ell(h) = 3$ and $d(T_g,A_h) = 0$, or $\ell(h) = 1$ and $d(T_g, A_h) = 1$, when inversions are allowed, we obtain amalgams of the same form as in the previous case.
    \item If $\ell(h) = 2$, $d(T_g,A_h) = 0$, and $X$ is four-regular or $(3,4)$-biregular, then $\langle g,h\rangle$ includes a family of two-generator faithful $(4,3)$-amalgamated products, and this further contains a family of three-group amalgams $G_0 *_{H_1} G_1*_{H_2} G_2$ equivalent to the above.
    \end{itemize}
This family of amalgamated products is infinite and appears to be impractical to fully classify. In contrast, once the geometric quantities $\ell(h)$ and $d(T_g,A_h)$ are fixed for an elliptic $g\in\aut(X)$ and a hyperbolic $h\in\aut(X)$, the minimal quotient graph for $\Gamma = \langle g,h\rangle$ becomes restricted, as we show in the following subsection.
\subsection{Restrictions for the Minimal Quotient Graph}
The restriction is exhibited in the following main result:
\begin{thm}\label{prop:ellip_gog}
    Let $X$ be a tree, $g,h\in\aut(X)$, and let $\Gamma = \langle g,h\rangle$ be a discrete and faithful subgroup. Suppose $g$ is elliptic with fixed tree $T_g$, and $h$ is hyperbolic with translation length $l = \ell(h)$, axis $A_h$, and distance $d = d(T_g,A_h)$. Then the quotient graph $Y$ of the minimal $\Gamma$-subtree by $\Gamma$ has cyclomatic number $\mathrm{b}_1(Y)\leq 1$; that is, $Y$ is either a tree or unicyclic.

    In addition, $Y$ must be obtained by identifying edges of the $(l,d)$-tadpole graph $Y'$ (a circuit of length $l$ with an attached path of length $d$). Furthermore, if $Y$ is unicyclic, then the circuit in $Y$ must be homotopic to the image of the circuit in the tadpole graph $Y'$ under the edge identifications.
\end{thm}
\begin{proof}
    For the first claim, we consider the fundamental group homomorphism
    \[
    \varphi: \Gamma = \pi_1(Y,\mathcal{G})\to \pi_1(Y) = F_{\mathrm{b}_1(Y)},
    \]
    obtained by forgetting the finite vertex groups of the quotient graph of groups $(Y,\mathcal{G})$. 
   
    On the one hand, this forgetful homomorphism is surjective. On the other hand, $\Gamma = \langle g,h\rangle$, with $g$ elliptic. Thus $\varphi(g) = 1$, and $\varphi(\Gamma)$ is either trivial or an infinite cyclic group $\langle \varphi(h)\rangle$. In particular, the rank of the free group $\pi_1(Y)$ is at most one, so $\mathrm{b}_1(Y)\leq 1$.

    For the assertion about the quotient graph, let $P$ be the path joining $T_g$ and $A_h$. Then the group $\Gamma$ acts on the subtree
    \[
    X_0 = \bigcup_{\gamma\in\Gamma}(\gamma.A_h\cup \gamma.P).
    \]
    Its quotient graph, $X_0/\Gamma$, is obtained by identifying edges in the vertex sum of $A_h/\langle h\rangle$, a circuit of length $\ell(h)$, with $P$, a path of length $d(T_g,A_h)$. Thus it is obtained from the tadpole graph $Y'$ by edge identifications.
    
    In this construction, $h$ corresponds to the circuit of length $l$ in $Y'$, and after the further edge identifications and forgetting of the vertex groups, $\varphi(h)\in\pi_1(Y)$ corresponds to the image of that circuit in $Y'$. Therefore, if $Y$ is unicyclic, then its circuit must be the image of the circuit in the tadpole graph $Y'$ under the edge identifications.
\end{proof}
Theorem \ref{prop:ellip_gog} implies that $Y$ is indeed obtained from \emph{folding} the edges of the $(l,d)$-tadpole graph, i.e., repetitively identifying adjacent edges, as well as their endpoints that are not incident with the other edge (if any). The theorem further gives an upper bound on the size of the quotient graph:
\begin{cor}\label{cor:ellip_gog}
    If $\mathrm{b}_1(Y) = 0$, then $\lvert\mathcal{E}(Y)\rvert\leq d(T_g,A_h) + \frac{1}{2}\ell(h)$.
    
    If $\mathrm{b}_1(Y) = 1$, then the girth of $Y$ is at most $\ell(h)$, and
    \[
    \lvert\mathcal{E}(Y)\rvert\leq d(T_g,A_h) +\frac{1}{2}(\ell(h)+\mathrm{girth}(Y)).
    \]
    When $d(T_g,A_h) = 0$, this bound is sufficient.
\end{cor}
\begin{proof}
    If $\mathrm{b}_1(Y) = 0$, then the circuit collapses to a tree, so each edge coming from the circuit arises from identifying at least two edges. If $\mathrm{b}_1(Y) = 1$, then the length of the remaining circuit is at most $\ell(h)$, and each edge not lying on the circuit also arises from identifying at least two edges. This gives the required upper bound on the size of the graph.
\end{proof}
On the other hand, except for the case where $Y$ is a single edge (as discussed in Section \ref{sec:rigidity}), $Y$ arises as a reduced quotient of an infinite family of elliptic-hyperbolic generated discrete subgroups $\Gamma = \langle g,h\rangle <\aut(X_{2\Delta(Y)})$ with arbitrarily large vertex stabilizers.
\begin{thm}\label{thm:ellip_gog}
    Let $n\geq 2$, and let $Y$ be a graph obtained from the $(l,d)$-tadpole graph according to Theorem \ref{prop:ellip_gog} with maximal valency $\Delta(Y)\leq n$. Then there exist infinitely many elliptic-hyperbolic generated discrete and faithful subgroups $\Gamma = \langle g,h\rangle$ in $\aut(X_{2n})$, with $\ell(h) = l$, distance $d(T_g,A_h) = d$, and arbitrarily large vertex stabilizers, such that $Y$ is the reduced quotient of $\Gamma$.
\end{thm}
We illustrate Theorem \ref{thm:ellip_gog} with explicit constructions of graphs of groups. Since the general construction is fairly extensive, we postpone the proof to the next section, where constructions over paths, tadpole graphs, and star graphs are presented to support the main idea.
\subsection{Criteria and Inequalities for Reduced Graphs with \texorpdfstring{$d(T_g,A_h) = 0$}{d(Tg,Ah)=0}}
While Theorem \ref{thm:ellip_gog} is elegant in showing the richness of discrete subgroups of tree automorphisms through their quotient graphs, the valency of the tree on which the subgroup acts is constrained. When the ambient group is fixed to be $\aut(X_n)$ for certain $n$, discrete subgroups with a reduced quotient of larger valency must arise from the reduction procedure.

Let us first introduce the shorthand: if the associated group $\mathcal{G}(e)$ of an edge $e\in\mathcal{E}(Y)$ is a proper subgroup of the incident vertex group $\mathcal{G}(v)$, we double the half-edge closer to the vertex $v$. 
\begin{figure}[H]
    \centering
    \includegraphics[width=0.75\linewidth]{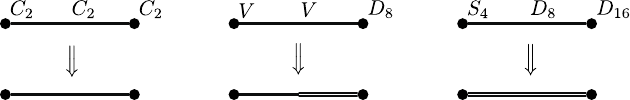}
    \caption{The shorthand used to specify when the edge groups are properly contained in the vertex groups. Any edge in a graph of groups falls into one of the three cases illustrated above.}
\end{figure}
The reduction of a graph of groups can then be understood through this doubled graph. Namely, if either half of an edge $e$ is not doubled, we may shrink the edge and merge the incident vertices $v_1$ and $v_2$. It is clear that the reduction of the other edges incident with $v_1$ and $v_2$ follows the rules below.
\begin{figure}[H]
    \centering
    \includegraphics[width=0.55\linewidth]{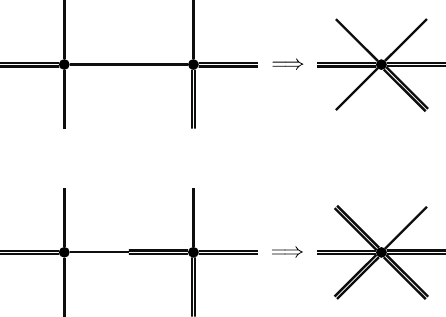}
    \caption{If the edge $e$ is single, then after reduction we keep the style of the other edges incident with $v_1$ and $v_2$. If the half of $e$ near $v_2$ is doubled, then after reduction we double the half-edges coming from edges incident with $v_1$, regardless of their styles before reduction.}
\end{figure}
To obtain graphs of groups subject to reduction, we require a strengthened version of Theorem \ref{thm:ellip_gog}:
\begin{prop}\label{prop:non_reduced}
    Let $n\geq 2$, and let $Y$ be a graph obtained from the $(l,d)$-tadpole graph. For each edge $e\in \mathcal{E}(Y)$ incident with a vertex $v\in\mathcal{V}(Y)$, assign a number $n_{e,v}\in \{1,2,3\}$, such that:
    \begin{itemize}
        \item If $\val(v) = 1$, then $n_{e,v} = 3$.
        \item If $\val(v)\geq 2$, then $n_{e,v} = 1$ or $2$.
        \item If $Y$ is a path or a circuit, then at least one $n_{e,v} > 1$.
        \item If $Y$ is unicyclic and contains circuit $C$, let $v^+(e)$ and $v^-(e)$ be the left and right endpoints of $e\in\mathcal{E}(C)$. Then
        \[
        \prod_{e\in\mathcal{C}}n_{e,v^+(e)} = \prod_{e\in\mathcal{C}}n_{e,v^-(e)}.
        \]
        \item If two edges incident with a vertex $v$ fold into the same edge $e$, then $n_{e,v}>1$.
    \end{itemize}
    Let $n = \max_{v\in\mathcal{V}(Y)}\sum_{e}n_{e,v}$. Then there exist infinitely many elliptic-hyperbolic generated discrete and faithful subgroups $\Gamma = \langle g,h\rangle$ in $\aut(X_n)$, with $\ell(h) = l$ and $d(T_g,A_h) = d$, such that $Y$ is the minimal quotient graph of groups for $\Gamma$, and $\lvert \Gamma_v:\Gamma_e\rvert = n_{e,v}$ for each edge $e$ incident with vertex $v$.
\end{prop}
We note that these valency requirements are natural: the index condition at leaf vertices ensures that the action on the minimal subtree is faithful; the condition for paths avoids the rigid scenario discussed in Section \ref{sec:rigidity}; the condition for unicyclic graphs ensures that there are no contradictions in the vertex group orders on the circuit. For the last requirement, the assumption implies the existence of an automorphism in $\Gamma$ that preserves $e$ and maps one incident edge to another, resulting in a proper inclusion.

The following proposition characterizes the possible reduced graphs for discrete subgroups $\langle g,h\rangle$ purely from graph theory, and follows naturally from Theorems \ref{prop:ellip_gog} and \ref{thm:ellip_gog}, and Proposition \ref{prop:non_reduced}.
\begin{prop}\label{char:ellip_gog}
    Let $X_n$ be the $n$-valent tree, and let $l\geq 1$ and $d\geq 0$. There exists a two-generator discrete and faithful subgroup $\Gamma = \langle g,h\rangle$ with geometric quantities $\ell(h) = l$ and $d(T_g,A_h) = d$ and reduced graph $Y$, if and only if $Y$ is obtained from the following construction:
    \begin{itemize}
        \item Iteratively folding adjacent edges of an $(l,d)$-tadpole graph, such that the resulting graph $Y_0$ has valencies $\val(v)\leq n$ for each vertex.
        \item Assigning numbers $n_{e,v}$ following Proposition \ref{prop:non_reduced}, using the short-hand notions for double edges, and reduce the graph $Y_0$ with double edges following the rules above. Let this reduction result be the desired $Y$.
    \end{itemize}
    In addition, unless $n = 3$ and $Y$ is a path of unit length, there exists an infinite family of such two-generator discrete subgroups.
\end{prop}
\begin{proof}
    If $Y_0$ is obtained by folding the edges of a $(l,d)$-tadpole graph, then for each admissible $n_{e,v}$-assignment, Proposition \ref{prop:non_reduced} guarantees two-generator subgroups $\Gamma = \langle g,h\rangle< \aut(X_n)$ whose minimal quotient graph is $Y_0$. Following the reduction rules, the graph of groups $(Y_0,\Gamma)$ can be reduced to certain $(Y,\mathcal{G})$ with the desired underlying graph $Y$.

    On the other hand, if $\Gamma = \langle g,h\rangle< \aut(X_n)$ is a two-generator discrete subgroup with geometric quantities $l = \ell(h)$ and $d = d(T_g,A_h)$, then by Theorem \ref{prop:ellip_gog}, the minimal quotient graph $Y_0$ is the result of folding certain adjacent edges of the $(l,d)$-tadpole graph. The corresponding vertex-edge indices $n_{e,v} = \lvert \Gamma_v:\Gamma_e\rvert$ for $(Y_0,\Gamma)$ satisfies the requirements (3) to (5) in Proposition \ref{prop:non_reduced} as we have explained. Even when $n_{e,v}$ are larger, the relation $n\geq \max_v\sum_e n_{e,v}$ implies the same holds for the $n_{e,v}$-assignment satisfying criteria (1) and (2) in Proposition \ref{prop:non_reduced}. Therefore, the reduced graph of groups $(Y,\mathcal{G})$ follows the description in the proposition.
\end{proof}
The criterion (5) in Proposition \ref{prop:non_reduced} may occur when the vertex is not a leaf of $Y$, making a simpler characterization for Proposition \ref{char:ellip_gog} less feasible. Nevertheless, this is not the case under the restriction $d(T_g,A_h) = 0$, and we are able to classify the possible reduced graphs explicitly with inequalities in this scenario.
\begin{thm}
    Let $g,h\in \aut(X_n)$, where $g$ is elliptic, $h$ is hyperbolic, $d(A_h,T_g) = 0$, and $\Gamma = \langle g,h\rangle$ is a discrete and faithful subgroup. Denote $l = \ell(h)$. Then the reduced graph $Y$ of $\Gamma$ can be any of the following:
    \begin{itemize}
        \item If $l$ is even, the reduced graph $Y$ can be any tree satisfying
        \[
        \sum_{v\in\mathcal{V}(Y)}2\left\lceil\frac{\val(v)-1}{n-2}+\frac{1}{2}\right\rceil \leq l+2.
        \]
        Except for the case where $l=2$ and $n=3$, the order of the vertex groups associated with $Y$ can be arbitrarily large.
        \item Whether $l$ is even or odd, the reduced graph $Y$ can be any unicyclic graph satisfying
        \[
        \sum_{v\in\mathcal{V}(Y)}2\left\lceil\frac{\val(v)-1}{n-2}+\frac{1}{2}\right\rceil - \sum_{v\in\mathcal{V}(C)}\left\lceil\frac{\val(v)-1}{n-2}+\frac{1}{2}\right\rceil \leq l.
        \]
        where $C$ denotes the unique circuit contained in $Y$. The associated vertex group orders can be arbitrarily large.
    \end{itemize}
\end{thm}
\begin{proof}
    Let $k = \lceil \frac{n}{2}\rceil\geq 2$. For a vertex $v\in\mathcal{V}(Y,\mathcal{G})$, consider the number $\mathrm{s}(v)$ of edges $e$ incident with $v$ for which $\mathcal{G}(e)< \mathcal{G}(v)$:
    \begin{itemize}
        \item If $\mathrm{s}(v)<\val(v)-k$, we say that the vertex $v$ is \emph{deficient}, and assign the weight $\mathrm{w}(v) = \val(v)-2$ to the vertex.
        \item If $\mathrm{s}(v)\geq \val(v)-k$, we say that the vertex $v$ is \emph{adequate}, and assign the weight $\mathrm{w}(v) = \mathrm{s}(v)+k-2$ to the vertex. In particular, if $\mathrm{s}(v) = \val(v)$, we say that the vertex is \emph{robust}.
    \end{itemize}
    We have two observations. First, we have $\mathrm{w}(v)\geq \val(v) - 2$ in both cases. Second, $\mathrm{s}(v) = \val(v)-k$ can be viewed as a boundary case: while the vertex is adequate, the equation $\mathrm{w}(v) = \val(v)-2$ still holds.
    
    When the graph of groups $(Y,\mathcal{G})$ is reduced to $(Y',\mathcal{G}')$ by shrinking an edge $e$, the vertices $v_1$ and $v_2$ of valencies $n_1$ and $n_2$ incident with $e$ merge into a new vertex $v$ of valency $n = n_1+n_2-2$. Let the corresponding proper boundary monomorphism numbers be $s_1$, $s_2$, and $s$. If the edge $e$ is single, we have $s_1\leq n_1-1$, $s_2\leq n_2 - 1$, and $s = s_1+s_2$. If the half of $e$ near $v_2$ is doubled, then $s_2\geq 1$ and $s =n_1+s_2-2$.
    \begin{lem}
        In all cases, we have $\mathrm{w}(v_1) + \mathrm{w}(v_2)\geq \mathrm{w}(v)$.
    \end{lem}
    \begin{proof}
        \textbf{Case (1): $e$ is single.} If one of the vertices, say $v_2$, is deficient, then
        \[
        s = s_1+s_2\leq n_1-1+n_2-k-1 = n-k,
        \]
        implying that $v$ is either deficient or lies on the boundary case:
        \[
        \mathrm{w}(v_1) + \mathrm{w}(v_2) \geq (n_1-2)+(n_2-2) = \mathrm{w}(v).
        \]
        
        If both $v_1$ and $v_2$ are adequate, then either the same argument holds if $v$ is deficient, or
        \[
        \mathrm{w}(v_1) + \mathrm{w}(v_2) = s_1+s_2+2k-4\geq s + k-2 = \mathrm{w}(v)
        \]
        if $v$ is adequate.

        \textbf{Case (2): $e$ is halfway-doubled.} Suppose the half of $e$ near $v_2$ is doubled. If $v_2$ is deficient, then
        \[
        s = n_1 + s_2 - 2< n_1+n_2 - 2 - k = n-k,
        \]
        implying that $v$ is deficient:
        \[
        \mathrm{w}(v_1) + \mathrm{w}(v_2) \geq (n_1-2)+(n_2-2) = \mathrm{w}(v).
        \]
        If $v_2$ is adequate, then either the same argument holds if $v$ is deficient, or
        \[
        \mathrm{w}(v_1) + \mathrm{w}(v_2) \geq (n_1-2) + (s_2 + k-2) = (n_1+s_2 - 2) + k-2 = s + k-2 = \mathrm{w}(v),
        \]
        if $v$ is adequate.
    \end{proof}
    Now suppose that a discrete subgroup $\Gamma$ acts on the $n$-valent tree, and that $(Y_0,\mathcal{G}_0)$ is the quotient graph of groups of the $\Gamma$-minimal subtree by $\Gamma$, while $(Y,\mathcal{G})$ is a reduced graph of groups obtained from $(Y_0,\mathcal{G}_0)$. Since the covering tree of $(Y_0,\mathcal{G}_0)$ has valency at most $n$, it follows that $\val(v)+\mathrm{s}(v)\leq n$. It is not difficult to see that $\mathrm{w}(v)\leq n-2$ for each $v\in \mathcal{V}(Y_0)$, with equality holding if and only if $(\val(v),\mathrm{s}(v)) = (n,0)$ or $(k,n-k)$.

    On the other hand, if a vertex $v\in\mathcal{V}(Y)$ is robust, then we have $\mathrm{w}(v) = \val(v)+k-2$. Since the weight sum of the graph of groups is non-increasing during the reduction process, $v$ must be merged from at least
    \[
    \left\lceil\frac{\val(v)+k-2}{n-2}\right\rceil = \left\lceil\frac{\val(v)-1}{n-2}+\frac{1}{2}\right\rceil
    \]
    vertices. As a result,
    \[
    \sum_{v\in\mathcal{V}(Y)}2\left\lceil\frac{\val(v)-1}{n-2}+\frac{1}{2}\right\rceil\leq \lvert \mathcal{V}(Y_0)\rvert.
    \]
    If $Y_0$ is a tree, then Corollary \ref{cor:ellip_gog} implies that
    \[
    \lvert \mathcal{V}(Y_0)\rvert = 1+\lvert \mathcal{E}(Y_0)\rvert\leq 1+\frac{l}{2},
    \]
    which gives the desired inequality.

    If $Y_0$ is unicyclic, we sum the vertex-wise inequalities differently, by counting the vertices not in the circuit $C_0$ twice:
    \[
    \sum_{v\in\mathcal{V}(Y)}2\left\lceil\frac{\val(v)-1}{n-2}+\frac{1}{2}\right\rceil - \sum_{v\in\mathcal{V}(C)}\left\lceil\frac{\val(v)-1}{n-2}+\frac{1}{2}\right\rceil \leq 2\lvert \mathcal{V}(Y_0)\rvert - \lvert \mathcal{V}(C_0')\rvert.
    \]
    Here, $C_0'$ is the set of vertices in $Y_0$ that ultimately merged into $C$. Since $C_0'$ contains the circuit $C_0$, we have
    \[
    2\lvert \mathcal{V}(Y_0)\rvert - \lvert \mathcal{V}(C_0')\rvert\leq 2\lvert \mathcal{V}(Y_0)\rvert - \lvert \mathcal{V}(C_0)\rvert = 2\lvert \mathcal{E}(Y_0)\rvert -\mathrm{girth}(Y_0)\leq l,
    \]
    which gives the desired inequality.

    We show that each tree or unicyclic graph $Y$ satisfying the inequality can be attained. Indeed, if $n\geq 4$, we replace a vertex $v\in\mathcal{V}(Y)$ of valency $k+(i-1)(n-2)< \val(v)\leq k+i(n-2)$ with a path consisting of a vertex $v_0\in \mathcal{V}(Y_0)$ with $(\val(v_0),\mathrm{s}(v_0)) = (k,n-k)$ and $i$ vertices $v_1,\dots, v_i$ with $\val(v_j)\leq n$ and $\mathrm{s}(v_j) = 0$. If $n = 3$, we replace $v$ with a path consisting of two vertices of type $(2,1)$ and $\val(v)-2$ vertices of type $(3,0)$. If $v$ is on the circuit in $Y$, then let all the vertices $v_0,\dots, v_i$ lie on the circuit in $Y_0$. As shown in the diagram below, our assembly of the vertices is also compatible with the boundary monomorphism indices on the circuit.
    \begin{figure}[H]
        \centering
        \includegraphics[width=0.65\linewidth]{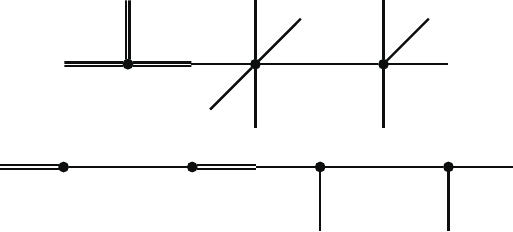}
        \caption{Upper diagram: $n = 6$ and $\val(v) = 10$. Bottom diagram: $n=3$ and $\val(v) = 4$.}
    \end{figure}
    
    The resulting graph $Y_0$ satisfies the inequality requirement, and by Theorem \ref{thm:ellip_gog}, we are able to construct an elliptic-hyperbolic generated subgroup $\Gamma = \langle g,h\rangle <\aut(X_n)$ such that $l = \ell(h)$, $d(T_g,A_h) = 0$, and its minimal quotient graph of groups is $(Y_0,\mathcal{G}_0)$, with the indices for the boundary monomorphisms agreeing with the selected $\mathrm{s}(v)$. After all, $(Y,\mathcal{G})$ is a reduced graph of groups associated with $\Gamma$, as claimed.
\end{proof}

\section{Constructions for Theorem \ref{thm:ellip_gog}}\label{sec:constr}
Before proceeding to the construction, it is worth mentioning that the edge groups are often constructed as elementary abelian $2$-groups:
\[
H = \langle t_1,\dots,t_k\mid t_i^2 = (t_it_j)^2 = 1\rangle\cong (C_2)^k.
\]
We identify $H$ with the vector space $V_H = \mathbb{F}_2^k$, which has dual space $V_H^* = \text{Hom}(V_H, \mathbb{F}_2)$. For any $f \in V_H^*$, we denote its kernel by $\ker(f) = \{ h \in H \mid f(h) = 0 \}$.

The following lemma provides conditions ensuring that $H$ is core-free in a larger group. In our constructions, this is usually sufficient to demonstrate the faithfulness of the $\Gamma$-action on its Bass-Serre tree, thereby ensuring that $\Gamma$ is a subgroup of $\aut(X)$.
\begin{lem}\label{lem:faithful}
    Suppose $H < \Gamma$ for some ambient group $\Gamma$. Assume the following:
    \begin{itemize}
        \item There exist elements $r_1,\dots, r_l\in \Gamma$, such that $(\bigcap r_i H r_i^{-1}) \cap H = \ker(t_1^*,\dots, t_n^*)$ for some non-zero $t_1^*,\dots, t_n^* \in V_H^*$.
        \item There exist elements $x_1, \dots, x_j\in \Gamma$, and for each $i=1,\dots,j$, there exists $\varphi_i \in \mathrm{GL}(k, 2)$ such that $x_i t x_i^{-1} = \varphi_i(t)$ for all $t \in H \cap x_i^{-1} H x_i$.
    \end{itemize}
    Let $\mathcal{W} \subset \mathrm{GL}(k, 2)$ be the set of automorphisms induced by words in the generators $\{x_i\}$. If the set of pulled-back functionals $\{ W^T t_i^* \mid W \in \mathcal{W},\ i=1,\dots,n \}$ spans $V_H^*$, then $H$ is \emph{core-free} in $\Gamma$, i.e.,
    \[
    \bigcap_{\gamma\in\Gamma}\gamma H\gamma^{-1} = \{1\}.
    \]
\end{lem}
\begin{proof}
    Let $K = \bigcap_{\gamma \in \Gamma} \gamma H \gamma^{-1}$ be the core of $H$ in $\Gamma$. Let $W \in \mathcal{W}$ be the linear transformation corresponding to the conjugation action of $w^{-1}$, where
    \[
    w = x_{i_m}^\pm\dots x_{i_1}^\pm.
    \]
    Denote by $w_i$ the partial product $x_{i_k}^\pm\dots x_{i_1}^\pm$. For any $t$ in
    \[
    \left(\bigcap_{i=1}^l(wr_i)H(wr_i)^{-1}\right)\cap\left(\bigcap_{k=0}^m w_kHw_k^{-1}\right),
    \]
    each intermediate conjugate $w_k^{-1}tw_k$ lies in $H$, so the successive conjugation relations $\mathrm{ad}_{x_{i_{k+1}}} = \varphi_{i_{k+1}}$ applies. Hence $w^{-1}tw = W(t)$ as an element of $V_H$. Moreover,
    \[
    w^{-1}tw \in H\cap\left(\bigcap_{i=1}^l r_iHr_i^{-1}\right),
    \]
    so $W(t)\in \ker(t_1^*,\dots,t_n^*)$, meaning that
    \[
    \langle t_i^*, W(t) \rangle = 0 \iff \langle W^T t_i^*, t \rangle = 0,\ i=1,\dots,n.
    \]
    Thus, $t \in \ker(W^T t_1^*,\dots,W^T t_n^*)$ for every $W \in \mathcal{W}$. Since the collection $\{ W^T t_i^* \mid W \in \mathcal{W},\ i=1,\dots,n \}$ spans $V_H^*$, an element $t$ lies in the intersection of these kernels if and only if $t = 1$ (the identity in $H$). Specifically,
    \[
    K \subseteq \bigcap_{W \in \mathcal{W},1\leq i\leq n} \ker(W^T t_i^*) = \{ 1 \}.
    \]
    We conclude that the core of $H$ is trivial, and thus $H$ is core-free in $\Gamma$.
\end{proof}
Another lemma is required to recover a generating pair from a graph of groups with the desired geometric quantities $\ell(h)$ and $d(T_g,A_h)$:
\begin{lem}\label{lem:geom_quant}
    We say a collection of subtrees $\{T_i\}$ is \emph{bridge-reduced} if (i) each $T_i$ contains exactly one vertex $v_i$ minimizing the distance to $\bigcup_{j\neq i} T_j$, and (ii) each $v_i$ is a leaf of the subtree induced by them.
    
    (1) If $h_i\in \aut(X)$, $i=1,\dots, m$ are elliptic automorphisms, and $\{T_{h_1},\dots, T_{h_m}\}$ is bridge-reduced, then their product
    \[
    h = h_1\dots h_m
    \]
    is hyperbolic, and
    \[
    \ell(h) = \sum_{i=1}^m d(T_{h_i},T_{h_{i+1}}),\ h_{m+1} = h_1.
    \]

    (2) In addition to (1), if $T^-$, $T^+$ are subtrees of $X$ such that $\{T^-, T^+, T_{h_1},\dots, T_{h_m}\}$ is bridge-reduced, then
    \[
    d(T^-, h.T^+) = d(T^-,T_{h_1}) + \sum_{i=1}^{m-1} d(T_{h_i},T_{h_{i+1}}) + d(T_{h_m}, T^+).
    \]

    (3) In addition to (2), if $T^- = v^-$, $T^+ = v^+$ are vertices on the axis of hyperbolic automorphism $h_0\in\aut(X)$, with $d(v^-,v^+) = \ell(h_0) + 1$, then
    \[
    h' = h_0h_1\dots h_m
    \]
    is hyperbolic with translation length
    \[
    \ell(h') = d(v^-,h.v^+) -1.
    \]
\end{lem}
The proof is a straightforward adaptation of \cite{bestvina2002r}, Exercise 2.6.

\subsection{Two-generator Subgroups with Elementary Quotient Graphs}
We will first describe the construction of two-generator subgroups of $\aut(X)$ with \emph{elementary} quotient graphs, meaning a path, a circuit, or a star graph. The framework for these quotient graphs carries the necessary idea for the general construction of two-generator $\aut(X)$-subgroups.

\subsubsection{Paths of Length \texorpdfstring{$\geq 2$}{>=2}}
By Theorem \ref{thm:3_3_amalgam}, there are only finitely many faithful amalgamated products of indices $\leq 3$, so finitely many $\aut(X)$-subgroups satisfying Theorem \ref{thm:ellip_gog} whose quotient graph is a path of length $1$. For this reason, the infinite family we will construct will apply to paths of length at least $2$.

\begin{prop}\label{thm:path}
    For any $n\geq 2$, each $j=1,\dots, n$, and every assignment of indices $n_{e,v}$ on a path $Y$ of length $n$ satisfying the requirements of Proposition \ref{prop:non_reduced}, there exists an infinite family of two-generator subgroups $\Gamma = \langle g,h\rangle< \aut(X_4)$ containing members with arbitrarily large vertex stabilizers such that:
    \begin{itemize}
        \item The generator $g$ is elliptic of order $\ord(g) \in \{3, 6\}$.
        \item The generator $h$ is hyperbolic, with $\ell(h) = 2j$ and $d(T_g,A_h) = n-j$.
        \item The minimal quotient graph for $\Gamma$ is the path $Y$, and the associated graph of groups satisfies $[\Gamma_v : \Gamma_e] = n_{e,v}$ for each edge $e$ incident with a vertex $v$. In particular, this includes the case when $Y$ is a reduced quotient.
    \end{itemize}
\end{prop}
\textbf{The construction.} Let the vertices of $Y$ be $v_0,v_1,\dots, v_n$, and the edges be $e_0,e_1,\dots, e_{n-1}$. We aim to construct a family of vertex groups $G_i = \mathcal{G}(v_i)$ and edge groups $H_i = \mathcal{G}(e_i)$, such that the fundamental group of the graph of groups $(Y,\mathcal{G})$,
\[
    \Gamma = \pi_1(Y,\mathcal{G}) = G_0 *_{H_0} G_1 *_{H_1}\dots *_{H_{n-1}} G_n
\]
satisfies the requirements in the proposition.

Let us first define the vertex groups $G_k$ for $k=1$,\dots, $n-1$: Let
\[
    m_k = \#\{1\leq i\leq k\mid n_{e_{i-1},v_i} = 2\},\ m_k' = \#\{1\leq i\leq k-1\mid n_{e_{i},v_i} = 2\},
\]
so
\[
    m_{k+1}-m_k,\ m'_{k+1}-m'_k\in \{0,1\},
\]
depending on the values $n_{e_{k},v_{k+1}}$ and $n_{e_{k},v_{k}}$. For $N$ large enough, let
\[
    G_k = \langle t_1,\dots, t_{2N-1}, r_{m'_k+1},\dots, r_{m'_{n}}, s_1,\dots, s_{m_k}\rangle,
\]
which is Abelian and generated by involutions. For $k=0,\dots, n-1$, let
\[
    H_k = \langle t_1,\dots, t_{2N-1}, r_{m'_{k+1}+1},\dots, r_{m'_{n}}, s_1,\dots, s_{m_k}\rangle,
\]
so $H_k<G_k$ and $H_k<G_{k+1}$, with $\lvert G_k:H_k\rvert = n_{e_k,v_k}$ and $\lvert G_{k+1}:H_k\rvert = n_{e_k,v_{k+1}}$. In particular,
\[
    H_0 = \langle t_1,\dots, t_{2N-1}, r_{1},\dots, r_{m'_{n}}\rangle,
\]
and
\[
    H_{n-1} = \langle t_1,\dots, t_{2N-1}, s_1,\dots, s_{m_{n-1}}\rangle.
\]
Define $G_0 = H_0\rtimes\langle x\rangle$ and $G_n = H_{n-1}\rtimes \langle y\rangle$, with $\ord(x) = \ord(y) = 3$, and the conjugation relations on the generators are given:
    \begin{center}
        \begin{tabular}{c||c|c|cc}
            $t$ & $t_{2i-1}$ & $t_{2i}$ & $r_i$ & ($i=1,\dots, m'_{n}$) \\
            \hline
            $t^x$ & $t_{2i}$ & $r_i$ & $t_{2i-1}$\\
            \hline
            $t$ & $t_{2N-1}$ & $t_{2i-1}$ & $t_{2i}$ &($i=m'_{n}+1,\dots, N-1$) \\
            \hline
            $t^x$ & $t_{2N-1}$ & $t_{2i}$ & $t_{2i-1}t_{2i}$
        \end{tabular}
    \end{center}
    \begin{center}
        \begin{tabular}{c||c|c|cc}
            $t$ & $t_1$ & $t_{2i}$ & $t_{2i+1}$ & ($i=1
            ,\dots, N-m_{n-1}-1$) \\
            \hline
            $t^y$ & $t_1$ & $t_{2i+1}$ & $t_{2i}t_{2i+1}$ \\
            \hline
            $t$ & $t_{2i}$ & $t_{2i+1}$ & $s_{N-i}$ & ($i=N-m_{n-1},\dots, N-1$) \\
            \hline
            $t^y$ & $t_{2i+1}$ & $s_{N-i}$ & $t_{2i}$
        \end{tabular}
    \end{center}
From the construction, one sees that $\Gamma$ satisfies the indices requirement in Proposition \ref{prop:non_reduced}; the covering tree has valency at most $4$, and $\Gamma$ acts on $X_4$. We will show that $\Gamma$ is two-generator with the desired geometric quantities, and its action on its covering tree is faithful.
\begin{proof}
    \textbf{Generating Pairs.} For $j=n$, let $(g,h) = (x,xyt_1)$. Since both $x$ and $yt_1$ are elliptic, and the graph of groups suggests
    \[
    d(T_x,T_{yt_1}) = n,
    \]
    by Lemma \ref{lem:geom_quant}, one has that $h$ is hyperbolic with $\ell(h) = 2n$. In addition, $g$ is elliptic of order $3$, and $d(T_g,A_h) = 0$. The original generators $y$, $t_i$, $r_i$ and $s_i$ are recovered as follows:
    \begin{center}
        \begin{tabular}{c||c|c|c|c|c|c}
            Generator & $t_1$ & $y$ & $t_{2i}$ & $t_{2i+1}$ & $r_i$ & $s_i$ \\
            \hline
            Recovery & $(g^{-1}h)^3$ & $(g^{-1}h)^4$ & $xt_{2i-1}x^{-1}$ & $yt_{2i}y^{-1}$ & $xt_{2i}x^{-1}$ & $yt_{2(N-i)+1}y^{-1}$
        \end{tabular}
    \end{center}

    For $j<n$, let $(g,h) = (yt_1,xs_{m_j})$. Similarly one has
    \[
    d(T_{xs_{m_j}},T_{yt_1}) = j,
    \]
    and $h$ is hyperbolic with $\ell(h) = 2j$ by Lemma \ref{lem:geom_quant}; $g$ is elliptic of order $6$, and $d(T_g,A_h) = n-j$. The original generators $x$, $y$, $t_i$, $r_i$ and $s_i$ are recovered as follows:
    \begin{center}
        \begin{tabular}{c||c|c|c|c|c|c|c}
            Generator & $t_1$ & $y$ & $t_{2i}$ & $t_{2i+1}$ & $s_i$ & $x$ & $r_i$ \\
            \hline
            Recovery & $g^3$ & $g^4$ & $ht_{2i-1}h^{-1}$ & $yt_{2i}y^{-1}$ & $yt_{2(N-i)-1}y^{-1}$ & $hs_{m_j}$ & $xt_{2i}x^{-1}$
        \end{tabular}
    \end{center}

    \textbf{Action Faithfulness.} It suffices to show that $H_0$ is core-free in $\Gamma$:
    \[
    \bigcap_{\gamma\in\Gamma}\gamma H_0\gamma^{-1} = \{1\}.
    \]
    Indeed, from the conjugation relations:
    \[
    H_0\cap yH_0y^{-1}\cap y^{-1}H_0y = \ker(r_1^*,\dots,r_{m'_{n}}^*,t_{2(N-m_{n-1})}^*,\dots,t_{2N-1}^*).
    \]
    As long as $\max(m'_n,m_{n-1})>0$, this kernel is non-trivial. Alternately applying the pulled-back functionals by $x$- and $y$-conjugations to the covectors defining this kernel gives all generators of $V_{H_0}^*$, namely
    \[
    V_{H_0}^* = \mathrm{span}(t_1^*,\dots,t_{2N-1}^*,r_1^*,\dots,r_{m'_{n}}^*).
    \]
    By Lemma \ref{lem:faithful}, $H_0$ is core-free, and the $\Gamma$-action on the Bass-Serre tree is faithful.
\end{proof}
\subsubsection{Circuits}
The infinite family of faithful $2$-HNN extensions in Theorem \ref{thm:HNN_rigidity} has loop quotient graphs. Modifying the construction yields a family of two-generator subgroups whose quotient graphs are circuits of any length.
\begin{prop}\label{prop:circuit}
    For any $n\geq 1$ and every assignment of indices $n_{e,v}$ on a circuit $Y$ of length $n$ satisfying the requirements of Proposition \ref{prop:non_reduced}, there exists an infinite family of two-generator subgroups $\Gamma = \langle g,h\rangle< \aut(X_4)$ containing members with arbitrarily large vertex stabilizers such that:
    \begin{itemize}
        \item The generator $g$ is elliptic of order $\ord(g)= 2$.
        \item The generator $h$ is hyperbolic, with $\ell(h) = n$ and $d(T_g,A_h) = 0$.
        \item The minimal quotient graph for $\Gamma$ is the circuit $Y$, and the associated graph of groups satisfies $[\Gamma_v : \Gamma_e] = n_{e,v}$ for each edge $e$ incident with a vertex $v$.
    \end{itemize}
\end{prop}
\textbf{The construction.} Let the vertices of $Y$ be $v_1$, $v_2$,\dots, $v_n$, and the edges be $e_1$, $e_2$,\dots, $e_n$. We aim to construct a family of vertex groups $G_i = \mathcal{G}(v_i)$ and edge groups $H_i = \mathcal{G}(e_i)$, such that the fundamental group $\Gamma$ of the graph of groups $(Y,\mathcal{G})$ satisfies the requirements in the proposition:
\begin{center}
    \begin{tikzpicture}[scale=2.5]
        \fill[] (0,0) circle(1pt) node[below]{$G_1$};
        \draw[thick] (0,0) -- node[below]{$H_1$} (1,0);
        \fill[] (1,0) circle(1pt) node[below]{$G_2$};
        \draw[thick,dashed] (1,0) -- (3,0);
        \fill[] (3,0) circle(1pt) node[below]{$G_{n-1}$};
        \draw[thick] (3,0) -- node[below]{$H_{n-1}$} (4,0);
        \fill[] (4,0) circle(1pt) node[below]{$G_n$};
        \draw[thick] (4,0) to [bend right=10] node[above]{$H_n$} (0,0);
    \end{tikzpicture}
\end{center}
Let us define the vertex and edge groups: without loss of generality, let $v_1$ be the vertex with the smallest vertex group order, as implied by the numbers $n_{e,v}$. Let
\[
m_k = \#\{1\leq i\leq k-1\mid n_{e_i,v_{i+1}} = 2\},\ m_k' = \#\{1\leq i\leq k-1\mid n_{e_i,v_i} = 2\},
\]
so
\[
m_{k+1} - m_k,\ m'_{k+1}-m'_k\in\{0,1\},\ m_{n+1} = m'_{n+1}\coloneqq m.
\]
For $N\geq 3$, we define
\[
G_k = \langle t_{m'_k+1},\dots, t_{mN+m_k+1}\mid t_i^2=1,\ (t_it_{i+mN})^2 = t_{i+m+1},\ (t_it_{i+j})^2 = 1,\forall j\neq mN\rangle,
\]
and
\[
H_k = \langle t_{m'_{k+1}+1},\dots, t_{mN+m_k+1}\rangle,
\]
with the same relators. Note that the choice of $v_1$ implies that $m_k\geq m'_k$, so every $G_k$ contains non-trivial commutator relations. On the other hand, $m_k - m'_k\leq m$, so each $t_{i+m+1}$ on the right-hand side of commutator relation does not appear in the left-hand side of another commutator $(t_jt_{j+mN})^2$. This guarantees the groups $G_k$ do not collapse.

Note that the subgroup relations $H_k<G_k$ and $H_k<G_{k+1}$ hold, with $\lvert G_k:H_k\rvert = n_{e_k,v_k}$ and $\lvert G_{k+1}:H_k\rvert = n_{e_k,v_{k+1}}$. In particular,
\[
G_1 = \langle t_1,\dots, t_{mN+1}\rangle,
\]
and
\[
H_n = \langle t_{m+1},\dots, t_{m(N+1)+\epsilon}\rangle,
\]
where $\epsilon\in \{0,1\}$ depending on whether $n_{e_n,v_1} = 1$ or $2$. Define the stable letter $x$ and the boundary injection $\varphi: H_n\to G_1$ by
\[
    \varphi(t_{i}) = xt_ix^{-1} = t_{i-m}.
\]
From the construction, one sees that $\Gamma$ satisfies the indices requirement in Proposition \ref{prop:non_reduced}, and $\Gamma$ acts on $X_4$. We will show that $\Gamma$ is two-generator with the desired geometric quantities, and its action is faithful.
\begin{proof}
    \textbf{Generating Pairs.} Let $(g,h) = (t_1,x)$; then $g$ is elliptic of order $2$, $h$ is hyperbolic with $\ell(h) = n$, and $d(T_g,A_h) = 0$. The original generators $t_i$ are recovered as follows:
    \begin{center}
        \begin{tabular}{c||c|c}
            Generator & $t_{i+1}$ & $t_{i+m}$ \\
            \hline
            Recovery & $\varphi((t_i\varphi^{-N}(t_i))^2)$ & $\varphi^{-1}(t_i)$
        \end{tabular}
    \end{center}

    \textbf{Action Faithfulness.} It is clear from the definition of $\varphi$ that
    \[
    \bigcap_{i\in\mathbb{Z}}\varphi^i(H_1) = \bigcap_{i\in\mathbb{Z}} x^i H_1 x^{-1} = \{1\}.
    \]
    Therefore, $H_1$ is core-free, and the $\Gamma$-action on the Bass-Serre tree is faithful.
\end{proof}
When $Y$ is a unicyclic graph with leaves, we are able to construct two-generator subgroups with minimal quotient $Y$ and abelian vertex stabilizers on the circuit. As the construction below demonstrates, this simplifies the general construction.
\begin{cor}\label{thm:tadpole}
    For any $n\geq 1$, $j\in\{0,1\}$, and every assignment of indices $n_{e,v}$ on an $(n,1)$-tadpole graph (a leaf vertex attached to an $n$-circuit) satisfying the requirements of Proposition \ref{prop:non_reduced}, there exists an infinite family of two-generator subgroups $\Gamma = \langle g,h\rangle< \aut(X_6)$ containing members with arbitrarily large vertex stabilizers such that:
    \begin{itemize}
        \item The generator $g$ is elliptic of order $\ord(g)= 6$.
        \item The generator $h$ is hyperbolic, with $\ell(h) = n+2j$ and $d(T_g,A_h) = 1-j$.
        \item The minimal quotient graph for $\Gamma$ is the tadpole graph $Y$, and the associated graph of groups satisfies $[\Gamma_v : \Gamma_e] = n_{e,v}$ for each edge $e$ incident with a vertex $v$.
    \end{itemize}
\end{cor}
\textbf{The construction.} Let the vertices of $Y$ be $v_0$, $v_1$,\dots, $v_n$ and the edges be $e_0$, $e_1$,\dots, $e_n$, corresponding to vertex groups $G_i$ and edge groups $H_i$. Here, $v_1$ to $v_n$ form a loop, and $v_0$ is a leaf attached to the loop at $v_1$:
\begin{center}
    \begin{tikzpicture}[scale=2.5]
        \fill[] (-1,0) circle(1pt) node[below]{$G_0$};
        \draw[thick] (-1,0) -- node[below]{$H_0$} (0,0);
        \fill[] (0,0) circle(1pt) node[below]{$G_1$};
        \draw[thick] (0,0) -- node[below]{$H_1$} (1,0);
        \fill[] (1,0) circle(1pt) node[below]{$G_2$};
        \draw[thick,dashed] (1,0) -- (3,0);
        \fill[] (3,0) circle(1pt) node[below]{$G_{n-1}$};
        \draw[thick] (3,0) -- node[below]{$H_{n-1}$} (4,0);
        \fill[] (4,0) circle(1pt) node[below]{$G_n$};
        \draw[thick] (4,0) to [bend right=10] node[above]{$H_n$} (0,0);
    \end{tikzpicture}
\end{center}
Again, let 
\[
m_k = \#\{1\leq i\leq k-1\mid n_{e_i,v_{i+1}} = 2\},\ m_k' = \#\{1\leq i\leq k-1\mid n_{e_i,v_i} = 2\},
\]
so $m_{n+1} = m_{n+1}' \coloneqq m$. For $i=1,\dots, n$, define the vertex and edge groups $G_i$ and $H_i$ with the same generators as in Proposition \ref{prop:circuit}, but without the non-trivial commutators:
\[
G_k = \langle t_{m'_k+1},\dots, t_{mN+m_k+1}\mid t_i^2=1, (t_it_{j})^2 = 1\rangle,
\]
and
\[
H_k = \langle t_{m'_{k+1}+1},\dots, t_{mN+m_k+1}\mid t_i^2=1, (t_it_{j})^2 = 1\rangle,
\]
where $N$ is any sufficiently large number; define the boundary injection $\varphi: H_n\to G_1$, $\varphi(t_{i}) = xt_ix^{-1} = t_{i-m}$ as before. In particular, the group for the junction vertex $v_1$ (where the tail meets the circuit) is
\[
G_1 = \langle t_1,\dots, t_{mN+1}\rangle.
\]
Depending on the number $n_{e_0, v_1}$, let $H_0 = G_1$, or an index-$2$ subgroup that does not contain any generator $t_i$ with $i>2m$. Define $G_0 = H_0\rtimes \langle y\rangle$ with $\ord(y) = 3$, and the non-trivial conjugation relations are given as follows:
\begin{center}
    \begin{tabular}{c||c|cc}
        $t$ & $t_{2i}$ & $t_{2i+1}$ & ($1\leq i\leq (m-1)/2$) \\
        \hline
        $t^y$ & $t_{2i+1}$ & $t_{2i}t_{2i+1}$\\
        \hline
        $t$ & $t_{m+2i-1}$ & $t_{m+2i}$ & ($1\leq i\leq m/2$) \\
        \hline
        $t^y$ & $t_{m+2i}$ & $t_{m+2i-1}t_{m+2i}$\\
    \end{tabular}
\end{center}
We will show that $\Gamma$ is two-generator with the desired geometric quantities, and its action is faithful.
\begin{proof}
    \textbf{Generating Pairs.} For $j=0$, let $(g,h) = (yt_1,x)$; then $g$ is elliptic of order $6$, $h$ is hyperbolic with $\ell(h) = n$ and $d(T_g,A_h) = 1$. The original generators $y$ and $t_i$ are recovered as follows:
    \begin{center}
        \begin{tabular}{c||c|c|c|c|c|c}
            Generator & $y$ & $t_1$ & $t_{2i}$ & $t_{2i+1}$ & $t_{i+m}$ \\
            \hline
            Recovery & $g^4$ & $g^3$  & $xyx^{-1}t_{2i-1}xy^{-1}x^{-1}$ & $yt_{2i}y^{-1}$ & $x^{-1}t_ix$
        \end{tabular}
    \end{center}

    For $j=1$, let $(g,h) = (yt_1,xyt_1)$; then $g$ is elliptic of order $6$, and by Lemma \ref{lem:geom_quant}, $h$ is hyperbolic with $\ell(h) = n+2$ and $d(T_g,A_h) = 0$. Note that the two generating pairs are equivalent.

    \textbf{Action Faithfulness.} As before, the subgroup $H_1$ is core-free in $\Gamma$, since
    \[
    \bigcap_{i\in\mathbb{Z}}\varphi^i(H_1) = \bigcap_{i\in\mathbb{Z}} x^i H_1 x^{-1} = \{1\}.
    \]
    
\end{proof}
\subsubsection{Star Graphs}
To illustrate the construction of two-generator subgroups whose minimal quotient graph has at least three leaves, we discuss the simplest case, that is, when it is a star graph.
\begin{prop}\label{thm:star}
    For any $n\geq 3$, each $j  = 1, \dots, n$, and every assignment of indices $n_{e,v}$ on a star graph $Y$ of $n$ edges satisfying the requirements of Proposition \ref{prop:non_reduced}, there exists an infinite family of two-generator subgroups $\Gamma = \langle g,h\rangle< \aut(X_{2n})$ containing members with arbitrarily large vertex stabilizers such that:
    \begin{itemize}
        \item The generator $g$ is elliptic of order $\ord(g)= 6$.
        \item The generator $h$ is hyperbolic, with $\ell(h) = 2j$, and $d(T_g,A_h) = 2(n-j)-1$ when $j<n$, or $d(T_g,A_h) = 0$ when $j=n$.
        \item The minimal quotient graph for $\Gamma$ is the star graph $Y$, and the associated graph of groups satisfies $[\Gamma_v : \Gamma_e] = n_{e,v}$ for each edge $e$ incident with a vertex $v$.
    \end{itemize}
\end{prop}
\textbf{The construction.} Let the vertices of $Y$ be $v_\star$, $v_0$,\dots, $v_{n-1}$, and the edges be $e_0$,\dots, $e_{n-1}$. Here $v_\star$ is the center vertex, and $e_i$ is incident with $v_\star$ and $v_i$. We aim to construct a family of vertex groups $G_\star$ and $G_i$, and edge groups $H_i$, $i=0,\dots, n-1$, such that the fundamental group $\Gamma$ of the graph of groups $(Y,\mathcal{G})$ satisfies the requirements in the proposition.

For sufficiently large $N$, define
\[
G_\star = \langle t_0,t_1,\dots, t_{3N-2}\rangle\cong (C_2)^{3N-1}.
\]
For $i=0,\dots, n-1$, define the edge subgroups: if $n_{e_i,v_\star} = 1$, set $H_i = G_\star$. Otherwise,
\[
H_0 = \langle t_i\mid i\neq 3N-2\rangle,\ H_{n-1} = \langle t_i\mid i\neq 0\rangle,
\]
and
\[
H_j = \langle t_i\mid i\neq 3j+2\rangle,\ j=1,\dots, n-2.
\]
Define $G_0 = H_0\rtimes \langle x\rangle$ and $G_{n-1} = H_{n-1}\rtimes \langle y\rangle$, with $\ord(x) = \ord(y) = 3$, and the conjugation relations on the generators as follows (trivial relations omitted):
\begin{center}
    \begin{tabular}{c||c|c|cc}
        $t$ & $t_{3i}$ & $t_{3i+1}$ & $t_{3i+2}$ & ($i=0,\dots,N-1$) \\
        \hline
        $t^x$ & $t_{3i+1}$ & $t_{3i+2}$ & $t_{3i}$ &
    \end{tabular}
\end{center}
\begin{center}
    \begin{tabular}{c||c|c|cc}
        $t$ & $t_{3i-1}$ & $t_{3i}$ & $t_{3i+1}$ & ($i=1,\dots,N$) \\
        \hline
        $t^y$ & $t_{3i}$ & $t_{3i+1}$ & $t_{3i-1}$
    \end{tabular}
\end{center}
For $j=1,\dots, n-2$, define $G_j = \langle z_j\rangle\rtimes H_j$, $\ord(z_j) = 3$, with the only non-trivial conjugation relation being
\[
t_{3j+1}z_jt_{3j+1} = z_j^2.
\]
The group inclusions $H_i<G_i$ and $H_i<G_\star$ define an $n$-star graph of groups, and the indices requirement in Proposition \ref{prop:non_reduced} is satisfied. We will show that its fundamental group $\Gamma$ is two-generator with the desired geometric quantities, and its action is faithful.
\begin{proof}
    \textbf{Generating Pairs.} For $j=n$, let $(g,h) = (yt_1,yz_1\dots z_{n-2}x)$. By Lemma \ref{lem:geom_quant}:
    \[
    \ell(h) = d(v_0,v_1) + \dots + d(v_{n-2},v_{n-1}) + d(v_{n-1}, v_0) = 2n,
    \]
    and $d(T_g,A_h) = 0$. Denote $x_i\coloneqq z_i\dots z_{n-2}x$; then the original generators are recovered as follows:
    \begin{center}
        \begin{tabular}{c||c|c|c|c|c|c|c|c}
            Generator & $y$ & $t_1$ & $x_1$ & $t_{3i-1}$ & $t_{3i}$ & $t_{3i+1}$ & $z_i$ & $x_{i+1}$ \\
            \hline
            Recovery & $g^4$ & $g^3$ & $y^{-1}h$ & $x_it_{3i-2}x_i^{-1}$ & $yt_{3i-1}y^{-1}$ & $y^{-1}t_{3i-1}y$ & $t_{3i+1}x_it_{3i}x_i^{-1}$ & $z_i^{-1}x_i$
        \end{tabular}
    \end{center}
    Specifically, when $i=n-2$, we obtain $x_{n-1} = x$. The other generators $t_j$ are recovered from $x$- and $y$-conjugations.

    For $j<n$, let $(g,h) = (yt_1, z_1\dots z_{n-2}xz_{n-1-j}^{-1}\dots z_1^{-1})$. Lemma \ref{lem:geom_quant} implies that $\ell(h) = 2j$, and
    \[
        \begin{split}
            & d(T_g,A_h) = d((z_1\dots z_{n-1-j}).T_{yt_1},A_{z_{n-j\dots z_{n-2}x}}) \\
            & = d(v_{n-1},v_1) + d(v_1,v_2) + \dots + d(v_{n-2-j}, v_{n-1-j}) + d(v_{n-1-j}, v_\star) = 2(n-j) - 1.
        \end{split}
    \]
    For $i\leq n-1-j$, let $x_i\coloneqq z_i\dots z_{n-2}xz_{n-1-j}^{-1}\dots z_i^{-1}$. Since $z_i$ commutes with $t_{3i}$, a similar relation $z_i = t_{3i+1}x_i t_{3i}x_i^{-1}$ still holds. Therefore, the original generators are recovered as follows:
    \begin{center}
        \begin{tabular}{c||c|c|c|c|c|c|c|c}
            Generator & $y$ & $t_1$ & $x_1$ & $t_{3i-1}$ & $t_{3i}$ & $t_{3i+1}$ & $z_i$ & $x_{i+1}$ \\
            \hline
            Recovery & $g^4$ & $g^3$ & $h$ & $x_it_{3i-2}x_i^{-1}$ & $yt_{3i-1}y^{-1}$ & $y^{-1}t_{3i-1}y$ & $t_{3i+1}x_it_{3i}x_i^{-1}$ & $z_i^{-1}x_iz_i$
        \end{tabular}
    \end{center}
    For $i>n-1-j$, the recovery of $t_{3i}$, $t_{3i-1}$, $t_{3i+1}$, $z_i$ and $x_{i+1}$ are the same as in the case $j=n$. 

    \textbf{Action Faithfulness.} To show that $H_0$ is core-free in $\Gamma$, we begin with the intersection:
    \[
    H_0\cap \left(\bigcap_i z_iH_0z_i^{-1}\right) = \ker(t_{4}^*,t_{7}^*,\dots, t_{3n-5}^*).
    \]
    The $\varphi_x^T$ and $\varphi_y^T$ actions on those $t_{3i+1}^*\in V_{H_0}^*$ result in the other generators of this dual space. By Lemma \ref{lem:faithful}, $H_0$ is core-free, and the $\Gamma$-action on the Bass-Serre tree is faithful.
\end{proof}
A modification of the construction allows the quotient graph to be a star graph with a loop:
\begin{cor}\label{thm:star_loop}
    For any $n\geq 3$, each $j  = 1, \dots, n$, and every assignment of indices $n_{e,v}$ on a star graph $Y$ of $n-1$ edges and an additional loop at the center, satisfying the requirements of Proposition \ref{prop:non_reduced}, there exists an infinite family of two-generator subgroups $\Gamma = \langle g,h\rangle< \aut(X_{2n+2})$ containing members with arbitrarily large vertex stabilizers such that:
    \begin{itemize}
        \item The generator $g$ is elliptic of order $\ord(g)= 6$.
        \item The generator $h$ is hyperbolic, with $\ell(h) = 2j-1$, and $d(T_g,A_h) = 2(n-j)-1$ when $j<n$, or $d(T_g,A_h) = 0$ when $j=n$.
        \item The minimal quotient graph for $\Gamma$ is the looped star graph $Y$, and the associated graph of groups satisfies $[\Gamma_v : \Gamma_e] = n_{e,v}$ for each edge $e$ incident with a vertex $v$.
    \end{itemize}
\end{cor}
\textbf{The construction.} Let the vertices be $v_\star$, $v_1$,\dots, $v_{n-1}$, and the edges be $e_\star$, $e_1$,\dots, $e_{n-1}$. As before, $v_\star$ is the center vertex, while $e_\star$ is a loop at $v_{\star}$. We will define the corresponding vertex groups $G_\star$ and $G_i$, and edge groups $H_\star$ and $H_i$:
\begin{center}
    \begin{tikzpicture}[scale=2.5]
        \fill[] (0,0) circle(1pt) node[below]{$G_\star$};
        \draw[thick] (0,0) to [out = 60, in = -240] (0.75,0.433) node[above right]{$H_\star$} to [out = 300, in = 0] (0,0);
        \fill[] (-0.5,0.866) circle(1pt) node[right]{$G_1$};
        \draw[thick] (0,0) -- node[above right]{$H_1$} (-0.5,0.866);
        \fill[] (-1,0) circle(1pt) node[left]{$G_2$};
        \draw[thick] (0,0) -- node[above]{$H_2$} (-1,0);
        \fill[] (0.5,-0.866) circle(1pt) node[right]{$G_{n-1}$};
        \draw[thick] (0,0) -- node[above right]{$H_{n-1}$} (0.5,-0.866);
        \draw[dashed,thick] (-0.866,-0.5) arc (210:270:1);
    \end{tikzpicture}
\end{center}
As before, for sufficiently large $N$, define 
\[
G_\star = \langle t_0,t_1,\dots, t_{3N-2}\rangle\cong (C_2)^{3N-1}.
\]
For $i=1,\dots, n-2$, define the edge subgroups: if $n_{e_i,v_\star} = 1$, set $H_i = G_\star$. Otherwise,
\[
H_{n-1} = \langle t_i\mid i\neq 0\rangle,\text{ and }H_j = \langle t_i\mid i\neq 3j+2\rangle,\ j=1,\dots, n-2.
\]
Define $G_{n-1} = H_{n-1}\rtimes \langle y\rangle$, with $\ord(y) = 3$, and the non-trivial conjugation relations on the generators are given as before:
\begin{center}
    \begin{tabular}{c||c|c|cc}
        $t$ & $t_{3i-1}$ & $t_{3i}$ & $t_{3i+1}$ & ($i=1,\dots,N$) \\
        \hline
        $t^y$ & $t_{3i}$ & $t_{3i+1}$ & $t_{3i-1}$
    \end{tabular}
\end{center}
For $j=1,\dots, n-2$, define $G_j = \langle z_j\rangle\rtimes H_j$, $\ord(z_j) = 3$, with the only non-trivial conjugation relation being
\[
t_{3j+1}z_jt_{3j+1} = z_j^2.
\]
Finally, define
\[
H_\star = \langle t_0,t_1,\dots, t_{3N-3}\rangle,
\]
and the boundary injection
\[
\varphi: H_\star\to G_\star,\ \varphi(t_i) = xt_ix^{-1} = t_{i+1}.
\]
We will show that its fundamental group $\Gamma$ is two-generator with the desired geometric quantities, and its action is faithful.
\begin{proof}
    \textbf{Generating Pairs.} For $j=n$, let $(g,h) = (yt_1,yz_1\dots z_{n-2}x)$. By Lemma \ref{lem:geom_quant}, we have $\ell(h) = 2n-1$, and $d(T_g,A_h) = 0$. 

    For $j<n$, let $(g,h) = (yt_1, z_1\dots z_{n-2}xz_{n-1-j}^{-1}\dots z_1^{-1})$. Lemma \ref{lem:geom_quant} implies that $\ell(h) = 2j-1$, and $d(T_g,A_h) = 2(n-j)-1$.
    
    The conjugation action of the stable letter $x$ on the generators $t_i$ is very similar to the elliptic automorphism $x$ in Proposition \ref{thm:star}. As a result, the recovery follows the same formulas as before. 

    \textbf{Action Faithfulness.} As in the circuit cases, the subgroup $H_\star$ is core-free, since
    \[
    \bigcap_{i\in\mathbb{Z}}\varphi^i(H_\star) = \bigcap_{i\in\mathbb{Z}} x^i H_\star x^{-1} = \{1\}.
    \]
\end{proof}

\subsection{Two-generator Subgroups with General Quotient Graphs}
Let us first introduce a shorthand notation: fix an assignment of numbers $n_{e,v}$ and for $v,v'\in\mathcal{V}(Y)$, denote
\[
m_{v,v'} = \#\{k\mid n_{e_i,v_{i}} > 1\},
\]
where $v_0$,\dots, $v_n$ are consecutive vertices on a path with $v = v_0$ and $v' = v_n$, and $e_i$ is the edge connecting $v_i$ and $v_{i+1}$.
\subsubsection{General Trees}
Having discussed the case where $Y$ is a path, we now assume $Y$ is a tree with at least $3$ leaves, denoted by $v_0$, \dots, $v_{n-1}$. Here, $Y$ is obtained by folding the edges of an $(l,d)$-tadpole graph; we assume that its junction vertex corresponds to $w\in\mathcal{V}(Y)$.

When taking the quotient from the $(l,d)$-tadpole graph to $Y$, both the $l$-circuit and the $d$-path become walks decomposed into maximal paths. Assume the endpoints of these paths are taken among $v_0$,\dots, $v_{n-1}$ and $w$, since any more general case reduces to such a particular case with smaller $l$ and $d$. 

More specifically, we may assume $v_{j}$,\dots, $v_{n-1}$ are leaves from the circuit, with
\[
\sum_{i=j+1}^{n-1}d(v_{i},v_{i-1}) + d(v_{j},w) + d(v_{n-1},w) = l,
\]
while $v_0$,\dots, $v_{j-1}$ are leaves from the path, with
\[
\sum_{i=1}^{j-1}d(v_i,v_{i-1}) + d(v_{j-1},w) = d.
\]
We will begin the construction with the path connecting $v_0$ and $v_{n-1}$; for $i=1$,\dots, $n-2$, let $w_i$ be the projection of $v_i$ to the tree generated by $v_0$,\dots, $v_{i-1}$, and $v_{n-1}$, and we construct the vertex and edge groups connecting $v_i$ and $w_i$. In particular, we set $m_i = m_{v_i,w_i}$, and $m_\star = \max_{1\leq i\leq n-2}(m_{v_0,v_i}, m_{v_{n-1},v_i})$.

\textbf{The construction.} We start with an Abelian group generated by involutions,
\[
K = \langle t_0,t_1,\dots,t_{3n-6}=p_0,p_1,\dots,p_{2N},q_1,\dots,q_M\rangle,
\]
where 
\[
M = \sum_{i=1}^{n-2} (m_i-1),\text{ and }N \geq \frac{M+m_\star}{2}.
\]
Define the vertex and edge groups on the path $[v_0,v_{n-1}]$ from $K$, analogously to the construction in Proposition \ref{thm:path}: for each vertex $v\in\mathcal{V}([v_0,v_{n-1}])$, define
\[
G_v = \langle K, r_{m_{v_{n-1},v_0}-m_{v,v_0}},\dots, r_{m_{v_{n-1},v_0}-1}, s_1,\dots, s_{m_{v,v_{n-1}}}\rangle,
\]
and edge groups $H_e$ similarly, so the indices agree with the numbers $n_{e,v}$. In particular, the edge groups at the leaves $v_0$ and $v_{n-1}$ are
\[
H_0 = \langle K, s_1,\dots, s_{m_{v_0,v_{n-1}}-1}\rangle,
\]
and
\[
H_{n-1} = \langle K, r_1,\dots, r_{m_{v_{n-1},v_0}-1} \rangle.
\]
Define $G_0 = H_0\rtimes\langle y\rangle$ and $G_{n-1} = H_{n-1}\rtimes \langle x\rangle$, with $\ord(x) = \ord(y) = 3$, and the non-trivial conjugation relations are:
\begin{center}
    \begin{tabular}{c||c|c|cc}
        $t$ & $t_{3i}$ & $t_{3i+1}$ & $t_{3i+2}$ & ($i=0,\dots,n-3$) \\
        \hline
        $t^x$ & $t_{3i+1}$ & $t_{3i+2}$ & $t_{3i}$ \\
        \hline
        $t$ & $p_{2i-2}$ & $p_{2i-1}$ & $r_i$ & ($i=1,\dots,m_{v_{n-1},v_0}-1$) \\
        \hline
        $t^x$ & $p_{2i-1}$ & $r_i$ & $p_{2i-2}$ \\
        \hline
        $t$ & $p_{2i}$ & $p_{2i+1}$ & & ($i= m_{v_{n-1},v_0}-1,...,N-1$) \\
        \hline
        $t^x$ & $p_{2i+1}$ & $p_{2i}p_{2i+1}$
    \end{tabular}
\end{center}
\begin{center}
    \begin{tabular}{c||c|c|cc}
        $t$ & $t_{3i+1}$ & $t_{3i+2}$ & $t_{3i+3}$ & ($i=0,\dots,n-3$) \\
        \hline
        $t^y$ & $t_{3i+2}$ & $t_{3i+3}$ & $t_{3i+1}$ \\
        \hline
        $t$ & $p_{2i-1}$ & $p_{2i}$ & $s_i$ & ($i=1,\dots,m_{v_0,v_{n-1}}-1$) \\
        \hline
        $t^y$ & $p_{2i}$ & $s_i$ & $p_{2i-1}$ \\
        \hline
        $t$ & $p_{2i+1}$ & $p_{2i+2}$ & & ($i= m_{v_0,v_{n-1}}-1,...,N-1$) \\
        \hline
        $t^y$ & $p_{2i+2}$ & $p_{2i+1}p_{2i+2}$
    \end{tabular}
\end{center}
Next, we construct the vertex and edge groups on the path $[v_i,w_i]$, inductively from $i=1$ to $i=n-2$. Note that the group $G_{w_i}$ is already constructed in the induction step, and all these groups will turn out to be abelian groups generated by involutions. Let $K_i$ be the subgroup of $G_{w_i}$ generated by letters except $t_{3i-1}$ and $p_1$,\dots, $p_{2N}$, and suppose
\[
G_{w_i} = \langle K_i, p_1,\dots, p_{n_i},t_{3i-1}\rangle.
\]
In addition, let $n_i' = \sum_{j=1}^{i-1}(m_j-1)$. For each vertex $v$ on $[v_i,w_i]$, define
\[
G_v = \langle K_i, p_1,\dots, p_{n_i - m_{w_i,v}},o_{n_i'+1},\dots, o_{n_i'+m_{v,w_i}}\rangle\rtimes \langle t_{3i-1}\rangle,
\]
and for each $e = [v,v']$ where $d(w_i,v') - d(w_i,v) = 1$, define
\[
H_e = \langle K_i, p_1,\dots, p_{n_i - m_{w_i,v'}},o_{n_i'+1},\dots, o_{n_i'+m_{v,w_i}}\rangle\rtimes \langle t_{3i-1}\rangle,
\]
where the new generators $o_j$ are involutions, the first factor of the semidirect product is abelian, and the only non-trivial conjugation relation is
\[
o_j^{t_{3i-1}} = o_jq_j,\ n_i'+1\leq j\leq n_{i+1}'.
\]
Denote by $e_{i}$ the leaf edge, and we define
\[
G_{v_i} = \langle H_{e_i},z_i\rangle = \langle K_i, p_1,\dots, p_{n_i - m_{w_i,v_i}},o_{n_i'+1},\dots, o_{n_{i+1}'}\rangle\rtimes \langle t_{3i-1},z_i\rangle,
\]
where $\ord(z_i) = 3$, $\langle t_{3i-1},z_i\rangle\cong D_6$, and the non-trivial conjugation relations are given as
\begin{center}
    \begin{tabular}{c||c|c|cc}
        $t$ & $p_j$ & $o_j$ & $q_j$ & $(j=n_i'+1,\dots,n_{i+1}')$ \\
        \hline
        $t^{z_i}$ & $o_j$ & $o_jq_j$ & $o_jq_jp_j$ \\
        \hline
        $t^{t_{3i-1}}$ & $p_j$ & $o_jq_j$ & $q_j$
    \end{tabular}
\end{center}
For clarity, we include the conjugation relations for $t_{3i-1}$, and it is clear that these are compatible with the fact that $\langle t_{3i-1},z_i\rangle\cong D_6$. Moreover, the construction satisfies the required vertex-edge index assignments. We will show that its fundamental group $\Gamma$ is two-generator with the desired geometric quantities, and its action is faithful.
\begin{proof}[Proof of Theorem \ref{thm:ellip_gog}, $Y$ is a tree]
    \textbf{Generating Pairs.} If $w$ lies on the tree induced by $v_j,\dots, v_{n-1}$, we may assume it is on the path $[v_j,v_{n-1}]$ by relabeling the leaves in the construction above. The distance sum relation in Lemma \ref{lem:geom_quant} can be further simplified as
    \[
    \sum_{i=j+1}^{n-1}d(v_{i},v_{i-1}) + d(v_{j},v_{n-1}) = l.
    \]
    In this scenario, we let $(g,h) = (yt_0,z_1\dots z_{n-2}xz_{j-1}^{-1}\dots z_1^{-1})$. By Lemma \ref{lem:geom_quant}, the desired geometric quantities hold: $\ell(h) = l$ and $d(T_g,A_h) = d$. Now it holds that $\Gamma = \langle g,h\rangle$: for the original generators $x$, $y$, $t_0$,\dots, $t_{3n-6}=p_0$, and $z_1$,\dots, $z_{n-2}$, an argument similar to that in the proof of Proposition \ref{thm:star} shows that they are obtained from $g$ and $h$. The remainder of the generators, namely $p_i$, $q_i$, $r_i$, $s_i$, and $o_i$, are obtained as follows:
    \begin{center}
        \begin{tabular}{c||c|c|c|c|c|c}
            Generator & $p_{2i-1}$ & $p_{2i}$ & $r_i$ & $s_i$ & $o_j$ & $q_j$ \\
            \hline
            Recovery & $xp_{2i-2}x$ & $yp_{2i-1}y$ & $xp_{2i-1}x$ & $yp_{2i}y$ & $z_ip_jz_i^{-1}$ & $o_jz_io_jz_i^{-1}$
        \end{tabular}
    \end{center}
    For $w$ not on this induced subtree, requirement (5) in Proposition \ref{prop:non_reduced} implies $H_e$ is properly contained in $G_w$, where $e$ is the edge incident with $w$ in $[v_j,w]$ and $[v_{n-1},w]$. From our construction, one sees that the generator in $G_w-H_e$ is either $s_i$ or $o_i$. Let $(g,h) = (yt_0,z_1\dots z_{n-2}s_ixz_{j-1}^{-1}\dots z_1^{-1})$ (or $o_i$ in place of $s_i)$, so the geometric quantities are satisfied. Note that the generator recovery is similar: $s_i$ commutes with $t_0$,\dots, $t_{3n-6}$ and $p_1$,\dots, $p_{2N}$; while $o_i$ does not commute with $t_{3i-1}$, a closer examination of the recovery process in Proposition \ref{thm:star} indicates that it will not affect the recovery of the original generators when $x$ is replaced with $o_ix$. The recovery of $p_i$, $q_i$, $r_i$, $s_i$, and $o_i$ is also similar, where $x$ would be recovered as long as $s_i$ or $o_i$ is recovered.

    \textbf{Action Faithfulness.} To show that $H_{n-1}$ is core-free, we note that the relations
    \[
    t_{3i-1}^{z_i} = t_{3i-1}z_i,\ p_j^{z_i} = o_j,\ q_j^{z_i} = o_jp_jq_j,
    \]
    imply that
    \[
    H_{n-1}\cap z_iH_{n-1}z_i^{-1}\cap z_i^{-1}H_{n-1}z_i = \ker(t_{3i-1}^*,(p_{j_1}^*+q_{j_1}^*),\dots,(p_{j_2}^*+q_{j_2}^*)),
    \]
    where $j_1 = n_i'+1$ and $j_2 = n_{i+1}'$. In particular, each $p_j^*+q_j^*$, $j=1,\dots, M$, serves as a null covector for a certain intersection above with $1\leq i\leq n-2$. In addition, certain compositions of conjugations by $x$ and $y$ take $t_2^*$,\dots, $t_{3n-7}^*$ to each $t_i^*$, $p_i^*$ and $r_i^*$. Since the covectors $t_i^*$, $p_i^*$, $(p_i^*+q_i^*)$ and $r_i^*$ span the entire $V_{H_{n-1}}^*$, $H_{n-1}$ is core-free in $\Gamma$.
\end{proof}
\subsubsection{General Unicyclic Graphs}
Having discussed the case where $Y$ is a circuit, we now assume $Y$ has at least one leaf vertex. As before, we denote these vertices by $v_0$,\dots, $v_{n-2}$, and the junction vertex by $w$. We assume $v_j$,\dots, $v_{n-2}$ are leaves from the circuit, $v_0$,\dots, $v_{j-1}$ from the path, with vertex distance relations similar to the previous case.

We will begin the construction with the tadpole graph containing the circuit $C$ in $Y$ and the path connecting $v_0$ to it. For $i=1$,\dots, $n-2$, let $w_i$ be the projection of $v_i$ to the tree generated by $v_0$,\dots, $v_{i-1}$, and $C$. We construct the vertex and edge groups connecting $v_i$ and $w_i$. In particular, we set $m_i = m_{v_i,w_i}$, and $m_\star = \max_{1\leq i\leq n-2}(m_{v_0,v_i})$.

\textbf{The construction.} We start with an Abelian group generated by involutions. Let $v_\star$ be the vertex on the circuit of $Y$ that is closest to $v_0$. As in the construction in Proposition \ref{prop:circuit} and Corollary \ref{thm:tadpole}, we let $m$ be the number of edge groups on the circuit properly contained in the initial vertex group following either direction. Fix an Abelian group generated by involutions,
\[
K = \langle t_m,\dotsm t_{mN},q_1,\dots, q_M\rangle,
\]
where
\[
M = \sum_{i=1}^{n-2}(m_i-1),\text{ and }N\geq 3n-4+m_{v_\star,v_0}+m_\star+M.
\]
Define the vertex and edge groups on the circuit $C\subset Y$ from $K$, analogously to the construction in Proposition \ref{prop:circuit}: fix a direction from $v_\star$ to each $v\in\mathcal{V}(C)$, and for each vertex $v\in\mathcal{V}(C)$, define
\[
    G_v = \langle K, t_{m_{v_\star,v}},\dots, t_{m-1}, t_{mN+1},\dots, t_{mN+m_{v,v_\star}}\rangle.
\]
In particular:
\[
    G_{v_\star} = \langle t_0,\dotsm t_{mN},q_1,\dots, q_M\rangle.
\]
Define the edge groups $H_e$ similarly; in particular, the last edge subgroup $H_e$ before returning to $G_{v_\star}$ is
\[
    H_e = \langle t_{m},t_{m+1},\dots, t_{m(N+1)-1},q_1,\dots, q_M\ \rangle,
\]
or with an additional $t_{m(N+1)}$, depending on $n_{e,v_\star}$. In both cases, define the boundary monomorphism and HNN stable letter:
\[
\varphi: H_e\to G_{v_\star},\ \varphi(t_{i}) = xt_ix^{-1} = t_{i-m},\ \varphi(q_j) = q_j.
\]
Next, we define the vertex and edge groups on the path $[v_0,v_\star]$. Now we fix
\[
K_0 = \langle t_0,\dots, t_{(N-m_{v_\star,v_0})m}, t_1,\dots, t_{m-1},\dots, t_{(N-1)m+1},\dots, t_{Nm-1}, q_1,\dots, q_M\rangle.
\]
For each vertex $v$ on the path, define
\[
G_v = \langle K_0, t_{(N-m_{v_\star,v_0}+1)m},\dots, t_{(N-m_{v_\star,v})m}, s_1,\dots, s_{m_{v,v_\star}}\rangle,
\]
and edge groups $H_e$ similarly, so the indices agree with the numbers $n_{e,v}$. In particular, the edge group at the leaf $v_0$ is
\[
H_0 = \langle K_0, s_1,\dots, s_{m_{v_0,v_\star}-1}\rangle.
\]
Define $G_0= H_0\rtimes \langle y\rangle$, with $\ord(y) = 3$; the non-trivial conjugation relations follow the pattern of Proposition \ref{thm:star_loop}, Corollary \ref{thm:tadpole}, and Proposition \ref{thm:path}: 
\begin{center}
    \begin{tabular}{c||c|c|cc}
        $t$ & $t_{(3i+1)m}$ & $t_{(3i+2)m}$ & $t_{(3i+3)m}$ & ($i=0,\dots,n-3$) \\
        \hline
        $t^y$ & $t_{(3i+2)m}$ & $t_{(3i+3)m}$ & $t_{(3i+1)m}$ \\
        \hline
        $t$ & $t_{(3n-5)m+2i}$ & $t_{(3n-5)m+2i+1}$ & & ($1\leq i\leq (m-1)/2$) \\
        \hline
        $t^y$ & $t_{(3n-5)m+2i+1}$ & $t_{(3n-5)m+2i}t_{(3n-5)m+2i+1}$ \\
        \hline
        $t$ & $t_{(3n-4)m+2i-1}$ & $t_{(3n-4)m+2i}$ & & ($1\leq i\leq m/2$) \\
        \hline
        $t^y$ & $t_{(3n-4)m+2i}$ & $t_{(3n-4)m+2i-1}t_{(3n-4)m+2i}$ \\
        \hline
        $t$ & $t_{(3n-4+i)m}$ & $s_i$ & & ($i=1,\dots, m_{v_0,v_\star}-1$) \\
        \hline
        $t^y$ & $s_i$ & $t_{(3n-4+i)m}$
        \end{tabular}
\end{center}
Finally, we construct the vertex and edge groups on the path $[v_i,w_i]$. This is very similar to the corresponding construction in the tree case: let $K_i$ be the subgroup of $G_{w_i}$ generated by all letters except $t_{(3i-1)m}$ and $t_{jm}$ for $j\geq 3n-5$, and suppose
\[
G_{w_i} = \langle K_i, t_{(3n-5)m},\dots, t_{n_im},t_{(3i-1)m}\rangle.
\]
Set $n_i'$ in the same way, and define similarly:
\[
G_v = \langle K_i, t_{(3n-5)m},\dots, t_{(n_i-m_{w_i,v})m},o_{n_i'+1},\dots, o_{n_i'+m_{v,w_i}}\rangle\rtimes \langle t_{3i-1}\rangle,
\]
\[
H_e = \langle K_i, t_{(3n-5)m},\dots, t_{(n_i-m_{w_i,v'})m},o_{n_i'+1},\dots, o_{n_i'+m_{v,w_i}}\rangle\rtimes \langle t_{3i-1}\rangle,
\]
and
\[
G_{v_i} = \langle H_{e_i},z_i\rangle = \langle K_i, t_{(3n-5)m},\dots, t_{(n_i-m_{w_i,v_i})m},o_{n_i'+1},\dots, o_{n_{i+1}'}\rangle\rtimes \langle t_{3i-1}\rangle,
\]
with conjugation relations
\begin{center}
    \begin{tabular}{c||c|c|cc}
        $t$ & $p_j$ & $o_j$ & $q_j$ & $(j=n_i'+1,\dots,n_{i+1}')$ \\
        \hline
        $t^{z_i}$ & $o_j$ & $o_jq_j$ & $o_jq_jp_j$ \\
        \hline
        $t^{t_{3i-1}}$ & $p_j$ & $o_jq_j$ & $q_j$
    \end{tabular}
\end{center}
Here, we use the shorthand
\[
p_j\coloneqq t_{(3n-4+m_{v_\star,v_0}+j)m}.
\]
This gives a graph of groups on $Y$ with the desired indices $n_{e,v}$. We will show that its fundamental group $\Gamma$ is two-generator, and has a faithful action.
\begin{proof}[Proof of Theorem \ref{thm:ellip_gog}, $Y$ is unicyclic]
    \textbf{Generating Pairs.} If $w$ lies on the graph induced by $v_j,\dots, v_{n-2}$ and the circuit, we may assume it is on the path $[v_j,v_{n-1}]$ by relabeling the vertices in the construction above. Let $(g,h) = (yt_0,z_1\dots z_{n-2}xz_{j-1}^{-1}\dots z_1^{-1})$. By Lemma \ref{lem:geom_quant}, the desired geometric quantities hold: $\ell(h) = l$ and $d(T_g,A_h) = d$. Similar to the proof for the tree case, the original generators $x$, $y$, $t_0$, $t_{m}$,\dots, $t_{(3n-6)m}$, and $z_1$,\dots, $z_{n-2}$ are obtained from $g$ and $h$. Following the proof for Corollary \ref{thm:tadpole}, we obtain the other generators $t_i$ from $t_{(3n-6)m}$, $x$, and $y$. The recovery of $s_i$, $p_j$, $q_j$, and $o_j$, are again similar to the proof for the tree case.

    If $w$ does not lie on this induced subgraph, requirement (5) in Proposition \ref{prop:non_reduced} again implies there is a generator $s_i$ or $o_i$ in $G_w - H_e$. Let $(g,h) = (yt_0,z_1\dots z_{n-2}s_ixz_{j-1}^{-1}\dots z_1^{-1})$; the geometric quantities are satisfied, and the original generators are recovered following the same argument as in the tree case.

    \textbf{Action Faithfulness.} We aim to show that the group $G_{v_\star}$ is core-free. Note that the HNN stable letter shifts all the generators $t_i$ away, meaning
    \[
    \bigcap_i x^i G_{v_\star} x^{-i} = \ker(t_0^*,\dots, t_{mN}^*).
    \]
    As in the proof for the tree case, we also have
    \[
    G_{v_\star}\cap z_iG_{v_\star} z_i^{-1}\cap z_i^{-1}G_{v_\star} z_i = \ker(t_{(3i-1)m}^*, (p_{j_1}^*+q_{j_1}^*),\dots, (p_{j_2}^*+q_{j_2}^*)),
    \]
    where $j_1 = n_i'+1$ and $j_2 = n_{i+1}'$; each $j=1,\dots, M$ corresponds to a certain $1\leq i\leq n-2$ in the relation above.

    Finally, $y$-conjugations take $t_{(3n-4+j)m}^*$ to $s_j^*$ for $j=1,\dots, m_{v_0,v_\star}-1$.
    
    These covectors together span the entire $V^*_{G_{v_\star}}$, so $G_{v_\star}$ is core-free in $\Gamma$.
\end{proof}

\section{Poincar\'e's Algorithm for Tree Automorphism Subgroups}\label{sec:Poincare}
Riley's algorithm \cite{riley1983applications} decides whether a finite set of elements $g_1,\dots, g_k\in \psl(2,\mathbb{R})$ or $\psl(2,\mathbb{C})$ generates a discrete and geometrically finite subgroup. The algorithm runs on a BSS machine, which works well with real numbers. By contrast, the automorphisms in $\aut(X)$ do not have a natural parametrization by real numbers. For this reason, let us introduce the following \emph{BSS machine over a tree} before describing our generalized algorithm:
\begin{defn}
    Let $X$ be an infinite, locally finite tree. A \emph{BSS machine over $X$} is a computational model with finitely many variables, each of which may store a vertex of $X$, an edge of $X$, or an automorphism in $\aut(X)$.

    The machine may copy the value of one variable into another of the same type. It may update an automorphism variable by composing two stored automorphisms or by taking the inverse of a stored automorphism. It may update a vertex (edge) variable either to the image of a stored vertex (edge) under a stored automorphism or to a specified vertex (edge) in the finite tree derived from the stored vertices and edges. It may also branch according to equality tests among stored vertices, edges, or automorphisms.

    The machine starts from a finite input and halts when a designated halting condition is met.
\end{defn}
\begin{thm}\label{thm:algorithm}
    For any locally finite tree $X$, there is a semi-decidable algorithm for the discreteness of subgroups of $\aut(X)$, performed on the BSS machine over $X$: for any finite set of elements $g_1,\dots, g_k\in \aut(X)$, the algorithm decides whether $\Gamma = \langle g_1,\dots, g_k\rangle$ is a discrete subgroup of $\aut(X)$.
\end{thm}

The algorithm is based on the following construction of \emph{Dirichlet domains}.
\begin{defn}
    Let $v\in \mathcal{V}(X)$, and let $\Gamma_0$ be a finite subset of $\aut(X)$. Define the \emph{Dirichlet domain} $D(v,\Gamma_0)$ centered at $v$ for $\Gamma_0$ as the following graph of $\aut(X)$-subsets:
    \begin{enumerate}
        \item The set $\Gamma_0(v)$ associated with $v$ is the stabilizer subset:
        \[
        \Gamma_0(v) = \{g\in\Gamma_0\mid g.v = v\}.
        \]
        \item Consider the subset of vertices
        \[
        \mathcal{V}(D_0') = \left\{w\in\mathcal{V}(X)\,\left|\,\begin{aligned}
            & d(v,w)\leq d(g.v,w),\forall g\in\Gamma_0,\\
            & d(v,w)\leq \frac{d(v,g_0.v)}{2},\ d(g_0.v,w)\leq \frac{d(v,g_0.v)+1}{2},\exists g_0\in \Gamma_0
        \end{aligned}\right.\right\}.
        \]
        Note that $\mathcal{V}(D_0')$ is finite, consisting of vertices at distance $\leq \frac{1}{2}$ from the midpoints of the paths $[v,g.v]$. Let $D'$ be the subtree induced by this subset.
        \item The relation $w\cong w'\iff \exists g\in \Gamma_0(v),\ g.w = w'$ induces an equivalence relation $w\sim w'$ on the vertices. Let $D''$ be any lift of $D'/\sim$ to $D'$, and let $\mathcal{V}(D_0'')$ be the set of vertices of $D''$ that come from $\mathcal{V}(D_0')$.
        \item Define another relation $w\cong' w'\iff \exists g\in \Gamma_0,\ g.w = w'$ and $d(v,w) = d(v,w')$. This induces an equivalence relation $w\sim' w'$ on $\mathcal{V}(D_0'')$. For $d(v,w)$ increasing from small to large, keep only one representative in each equivalence class. If $w$ is not selected as a representative, let $w_1$ be the vertex on $[v,w]$ such that $d(w,w_1) = 1$, and remove from $D''$ the component of $D'' - [w_1,w]$ containing $w$. Let $D$ be the resulting subtree.
        \item Let $\mathcal{V}(D_0)$ be the subset of vertices in $D$ coming from the vertices $w$ in $\mathcal{V}(D_0')$ such that $d(w,v) = \frac{d(v,g_0.v)-1}{2}$ and $d(w,g_0.v) = \frac{d(v,g_0.v)+1}{2}$ for some $g_0\in\Gamma_0$. Let $\mathcal{E}(D_0)$ be the edges $e$ in such $[w,g_0.v]$ adjacent to $w\in \mathcal{V}(D_0)$ such that $g_0^{-1}.e$ is not already in $\mathcal{E}(D_0)$. Let $\iota_e = \mathrm{ad}_{g_0}$ be the boundary map corresponding to $e$, and let $g_0^{-1}.w$ be the target vertex for $e$.
        \item Let $\mathcal{V}(D_1)$ be the set of vertices $w_1$ in Step (4), and let $\mathcal{E}(D_1)$ be the edge $e = [w_1,w]$ in that step. If $w'$ is the selected representative for this $w$, and if $g_1\in\Gamma_0$ satisfies $g_1.w' = w$, let $\iota_e = \mathrm{ad}_{g_1}$ be the boundary map corresponding to $e$, and let $w'$ be the target vertex for $e$.
        \item Let $D(v,\Gamma_0)$ be the subset $D\cup \mathcal{E}(D_0)\cup \mathcal{E}(D_1)$. Each vertex or edge in $D(v,\Gamma_0)$ is associated with its stabilizer subset in $\Gamma_0$, and each open edge $e\in \mathcal{E}(D_0)\cup \mathcal{E}(D_1)$ is associated with a boundary map $\iota_e$.
    \end{enumerate}
\end{defn}
\begin{prop}
    If $\Gamma$ is a finitely generated discrete subgroup of $\aut(X)$, and $v$ is a vertex on the minimal $\Gamma$-subtree, then there exists a subset $\Gamma_0\subset \Gamma$ such that the Dirichlet domain $D(v,\Gamma_0)$ is isomorphic to the quotient graph of groups of $\Gamma$. 
\end{prop}
\begin{proof}
    Since $\Gamma$ is a finitely generated discrete subgroup of $\aut(X)$, the quotient of the minimal $\Gamma$-subtree $T$ by $\Gamma$ is a finite graph. Still denote by $v$ its equivalence class representing a vertex on this finite quotient graph, and let $d = \max_{w\in\mathcal{V}(T/\Gamma)}d(v,w)$. Let $\Gamma_0$ contain all vertex stabilizer groups $\Gamma_w$ for $d(v,w)\leq d$, and all elements $g$ such that $d(v,g.v)\leq 2d$. Since $\Gamma$ is discrete, this includes only finitely many elements. We will show that $D(v,\Gamma_0)$, following this construction, is a fundamental graph of groups of $\Gamma$. Indeed, we have the following lemmas:
    \begin{lem}
        Each equivalence class $[w]\in T/\Gamma$ has a representative $w'\in D(v,\Gamma_0)$.
    \end{lem}
    \begin{proof}
        Consider the representatives of $[w]$ with the smallest distance to $v$; there are only finitely many such vertices. We will show that at least one of them is contained in $D(v,\Gamma_0)$.

        First, $[w]\in T/\Gamma$ corresponds to a hyperbolic element $g\in \Gamma$, which is either the product of elliptic elements on two leaves, or the monodromy of a circuit containing $[w]$. This hyperbolic element satisfies $d(v,g.v)\leq 2d$, so $g\in \Gamma_0$. By choosing the representatives closest to $v$, these vertices $w$ lie on paths $[v,g.v]$ for automorphisms $g$ satisfying the aforementioned description, and $d(v,w)\leq d(g.v,w)$.

        In Step (2) of the construction, there are no other elements $g'$ such that the midpoint of $[v,g'.v]$ lies between $v$ and $w$. Otherwise, $d(v,g'^{-1}.w) = d(g.v,w)<d(v,w)$, a contradiction to our choice.

        From the description of Steps (3) and (4), we see that at least one of the vertices $w$ survives and becomes a vertex in $D(v,\Gamma_0)$.
    \end{proof}
    \begin{lem}
        Vertices in $D(v,\Gamma_0)$ are in different equivalence classes.
    \end{lem}
    \begin{proof}
        For each vertex $w$ in $D(v,\Gamma_0)$, the projection of the path $[v,w]$ to $T/\Gamma$ does not have repeated vertices: otherwise, suppose $u$ and $g.u$ are two vertices in $[v,w]$. Then $d(v,g.v)<2d$ and $g\in \Gamma_0$. Moreover, it follows that $d(g.v,w)<d(v,w)$, contradicting Step (2) in the construction of $D(v,\Gamma_0)$.

        Now, suppose $w\sim w' = g.w$ are different vertices in $D(v,\Gamma_0)$. We consider the quotient paths of $[v,w]$ and $[v,g.w]$ in $T/\Gamma$. Step (3) in the Dirichlet domain construction guarantees that the quotient paths are different. If $d(v,w)\neq d(v,g.w)$, suppose $d(v,g.w)>d(v,w)$. Then $d(v,w')>d(g.v,g.w) = d(g.v,w')$, again contradicting Step (2) in the construction. Therefore $d(v,w) = d(v,g.w)$, but according to Step (4), at most one of them is in $D(v,\Gamma_0)$.
    \end{proof}
    It is now clear that the quotient of the subtree $D$ by $\Gamma$ is an isomorphism, and in particular, $D/\Gamma$ is a maximal subtree of $T/\Gamma$. For edges of $T/\Gamma$ not in $D/\Gamma$, if the distance between the two endpoints in $D/\Gamma$ is even, then Step (5) corresponds to these edges in $\mathcal{E}(D_0)$; if the distance is odd, then Step (6) corresponds to these edges in $\mathcal{E}(D_1)$. This shows that $D(v,\Gamma_0)$ and $T/\Gamma$ are isomorphic as graphs.

    For each vertex $w$ in $D$, since the distance $d(v,w)$ is less than the girth of $T/\Gamma$, the vertex stabilizer $\Gamma_w$ is contained in $\Gamma_0$. Therefore, the resulting graph of groups $D(v,\Gamma_0)$ is isomorphic to the quotient graph of groups of $\Gamma$.
\end{proof}
If $D(v,\Gamma_0)$ is a graph of groups, then $\pi_1(D(v,\Gamma_0))$ is a subgroup of $\langle g_1,\dots,g_k\rangle$.If each $g_i$ is contained in this fundamental group, then we will have $\Gamma = \pi_1(D(v,\Gamma_0))$. This leads to the algorithm claimed by Theorem \ref{thm:algorithm}.
\begin{alg}
    \underline{Input}: A finite generating set of tree automorphisms $g_1,\dots,g_k\in \aut(X)$, and a center vertex $v\in \mathcal{V}(X)$.
    
    \underline{Output}: \texttt{true} if the subgroup $\Gamma = \langle g_1,\dots,g_k\rangle$ is discrete, with a finite graph of finite groups $(Y,\mathcal{G})$ such that $\Gamma = \pi(Y,\mathcal{G})$.

    \begin{enumerate}
        \item Set $l=1$ and compute the subset $\Gamma_l$ containing all words in $g_1,\dots,g_k$ of length $\leq l$.
        \item Compute the Dirichlet domain $D(v,\Gamma_l)$ and verify whether it satisfies the following:
        \begin{enumerate}
            \item The subsets assigned to each vertex and edge in $D(v,\Gamma_l)$ form subgroups in $\aut(X)$.
            \item For each open edge $e$, the boundary map $\iota_e$ is a monomorphism from $\Gamma_l(e)$ to $\Gamma_l(w)$, where $w$ is the target vertex of $e$.
        \end{enumerate}
        \item If the two conditions are satisfied, then $D(v,\Gamma_l)$ is a graph of groups. Check the additional conditions, analogous to the last part of Section 1 in \cite{riley1983applications}:
        \begin{enumerate}
            \item For each original generator $g_i$, check whether $g_i.v$ lies on the $\pi_1(D(v,\Gamma_l))$-minimal subtree.
            \item If so, project the path $[v,g_i.v]$ to $D(v,\Gamma_l)$, and check whether it is a closed walk.
            \item If the walk is closed, let $g'$ be the corresponding monodromy taking $v$ to $g_i.v$. Check whether $g_i^{-1}g'\in \Gamma_l(v)$.
        \end{enumerate}
        \item If any condition is not satisfied, increment $l$ by $1$ and repeat.
        \item If all conditions hold, then $\Gamma = \langle g_1,\dots,g_k\rangle$ is a discrete subgroup of $\aut(X)$ and is isomorphic to the fundamental group of the graph of groups $D(v,\Gamma_l)$.
    \end{enumerate}
\end{alg}
\appendix
\section{Computing Elliptic Generating Pairs of Amalgamated Products}\label{sec:program}
As in Subsection \ref{subsec:3_3_amalgam}, we asked whether each of the faithful $(3,2)$- or $(3,3)$-amalgamated products $\Gamma = G_1*_HG_2$ admits a generating pair with particular geometric quantities:
\begin{itemize}
    \item A two-elliptic generating pair $(g,h)$, such that $d(T_g,T_h) = 1$, and $\ord(g) = n$, $\ord(h) = m$, for given integers $n,m$.
    \item An elliptic-hyperbolic generating pair $(g,h)$, such that $\ell(h) = 2$, and $\ell(T_g\cap A_h) = l$, for given integer $l$.
\end{itemize}
We may further assume, up to conjugating $(g,h)$: for the first scenario, we have $g\in G_1$ and $h\in G_2$, or $g\in G_2$ and $h\in G_1$. For the second scenario, we have $h = g_1g_2$, and up to an additional $h$-conjugation, we can assume $g\in G_1$ or $G_2$. In particular, we can assume $g\in H$ if $l\geq 1$.

To test if a chosen pair of elements generates the entire group $\Gamma = G_1*_H G_2$, we apply the following saturation procedure:
\begin{prop}
    Suppose that $G_1*_H G_2$ is an amalgamated product of finite groups.
    
    (1) Suppose $g\in G_1$, $h\in G_2$, and let $K_1^{(0)} = \langle g\rangle$, $K_2^{(0)} = \langle h\rangle$. Iteratively define
    \[
    K_1^{(i+1)} = K_1^{(i)}(K_2^{(i)}\cap H),\ K_2^{(i+1)} = K_2^{(i)}(K_1^{(i)}\cap H).
    \]
    Then $\langle g,h\rangle = G_1*_H G_2$ if and only if there is $N<\infty$ such that $K_1^{(N)} = G_1$ and $K_2^{(N)} = G_2$.

    (2) Suppose $g_1\in G_1 - G_2$, $g_2\in G_2 - G_1$, $h\in H$, and let $K_1^{(0)} = K_2^{(0)} = \langle h\rangle$. Iteratively define
    \[
    K_1^{(i+1)} = K_1^{(i)}(K_2^{(i)}\cap H)\left(g_1(g_2K_2^{(i)}\cap H)\right)\left(g_1(g_2K_2^{(i)}g_2^{-1}\cap H)g_1^{-1}\right),
    \]
    and
    \[
    K_2^{(i+1)} = K_2^{(i)}(K_1^{(i)}\cap H)\left((K_1^{(i)}g_1\cap H)g_2\right)\left(g_2^{-1}(g_1^{-1}K_1^{(i)}g_1\cap H)g_2\right).
    \]
    Then $\langle g_1g_2,h\rangle = G_1*_H G_2$ if and only if there is $N<\infty$ such that $K_1^{(N)} = G_1$ and $K_2^{(N)} = G_2$.
\end{prop}
\begin{proof}
    \textbf{Claim (1).} The ``if'' part is clear. For the ``only if'' part, assuming $\langle g,h\rangle = G_1*_H G_2$, we need to show that $K_1^{(N)} = G_1$ (with $G_2$ for the same argument) for some $N$. Since $G_1$ is finite, it suffices to show that $\gamma\in K_1^{(N)}$ for certain $N$ for each $\gamma\in G_1$.

    Since $\langle g,h\rangle = G_1*_H G_2$, $\gamma$ can be expressed as a word in $g$ and $h$. By the Fundamental Theorem of Graph of Groups, we can consider $G_1*_H G_2$ as a quotient of the free product $G_1 * G_2$, and the word
    \[
    \gamma = g^{i_1}_{G_1} \cdot h^{i_2}_{G_2}\cdot g^{i_3}_{G_1}\cdot\dots
    \]
    has a reduced form by: (1) replacing consecutive letters $({\gamma_1}_{G_i}{\gamma_2}_{G_i})$ with their product $(\gamma_1\gamma_2)_{G_i}$ for $i=1,2$, or (2) replacing $(\gamma_1)_{G_1}$ with $(\gamma_1)_{G_2}$ if $\gamma_1\in H$. 
    
    In the initial expression of $\gamma$, each letter is in $K_1^{(0)}$ or $K_2^{(0)}$. By replacing a letter $(\gamma_1)_{G_1}\in K_1^{(i)}\cap H$ with $(\gamma_1)_{G_2}$ and combining it with another $(\gamma_2)_{G_2}$, the resulting combination is in $K_2^{(i+1)} = K_2^{(i)}(K_1^{(i)}\cap H)$. The reduction procedure results in a reduced form in finitely many steps, which is exactly the single letter $(\gamma)_{G_1}$, as $\gamma\in G_1$. Consequently, $\gamma$ is contained in the saturation group $K_1^{(N)}$ for certain $N<\infty$.

    \textbf{Claim (2).} The proof of the second claim is similar; the major difference is that the generators $g_1$ and $g_2$ are not in our initial saturation subgroups. Instead, if two letters $\gamma_1$ and $\gamma_2$ (possibly identity elements) inside the saturation subgroups merge when crossing a $(g_1g_2)$ generator, it would follow a rule, equivalent to either
    \[
    {\gamma_1}_{G_1}({g_1}_{G_1}{g_2}_{G_2}){\gamma_2}_{G_2} = {\gamma_1}_{G_1}{g_1}_{G_1}{(g_2\gamma_2)}_{G_2} = {\gamma_1}_{G_1}{g_1}_{G_1}{(g_2\gamma_2)}_{G_1} = {(\gamma_1g_1g_2\gamma_2)}_{G_1},
    \]
    or
    \[
    \begin{split}
        & {\gamma_1}_{G_1}({g_1}_{G_1}{g_2}_{G_2}){\gamma_2}_{G_2}({g_2^{-1}}_{G_2}{g_1^{-1}}_{G_1}) = {\gamma_1}_{G_1}{g_1}_{G_1}({g_2\gamma_2g_2^{-1})}_{G_2}{g_1^{-1}}_{G_1}\\
        & = {\gamma_1}_{G_1}{g_1}_{G_1}({g_2\gamma_2g_2^{-1})}_{G_1}{g_1^{-1}}_{G_1} = {(\gamma_1g_1g_2\gamma_2g_2^{-1}g_1^{-1})}_{G_1}.
    \end{split}
    \]
    These correspond to the left- or right-multiplications and conjugations in the construction of the saturation subgroups. Consequently, if $\langle h,g_1g_2\rangle = G_1*_H G_2$, then for each $\gamma \in G_1$, $\gamma$ is contained in certain $K_1^{(N)}$ under this strengthened definition.
\end{proof}

We implemented this algorithm with \textit{GAP}. Since each step in the algorithm is computed within a finite group, either $G_1$ or $G_2$, we only need to define $G_1$, $G_2$ and $H$ without explicitly considering the amalgamated product in \textit{GAP}. In this setting, it is better to consider $H_1<G_1$ and $H_1<G_2$ as separate groups, with a canonical isomorphism $\varphi: H_1\to H_2$ between them. We refer to \cite{conder2025edge} for the explicit group presentations of the faithful amalgamated products. For example, the amalgamated product $C_{12} *_{V} A_4$ of type $\mathrm{G}_2$ is expressed as follows:
\begin{verbatim}
F1 := FreeGroup("c","d","x");;
c1 := F1.1;; d1 := F1.2;; x1 := F1.3;;
G1 := F1 / [ c1^2, d1^2, (c1*d1)^2, x1^3, (c1*x1)^2, d1*x1*d1*x1^-1 ];;
Ord1 := Order(G1);;

F2 := FreeGroup("c","d","y");;
c2 := F2.1;; d2 := F2.2;; y2 := F2.3;;
G2 := F2 / [ c2^2, d2^2, (c2*d2)^2, y2^3, (c2*y2)^3, d2*y2^-1*c2*y2 ];;
Ord2 := Order(G2);;

H1 := Subgroup( G1, [ G1.1, G1.2 ] );;
H2 := Subgroup( G2, [ G2.1, G2.2 ] );;
phi := GroupHomomorphismByImages( H1, H2, [ G1.1, G1.2 ],[ G2.1, G2.2 ]);;
psi := InverseGeneralMapping(phi);
\end{verbatim}
For the first task, we first obtain the possible orders of elements in $G_1$ and $G_2$, and implement the saturation algorithm to decide if each order pair admits a generating pair.
\begin{verbatim}
orders1 := Filtered( Set( List( AsList(G1), Order ) ), n -> n <> 1 );
orders2 := Filtered( Set( List( AsList(G2), Order ) ), n -> n <> 1 );

for n1 in orders1 do
    elts1 := Filtered( AsList(G1), g -> Order(g) = n1 );;
    SortBy( elts1, g -> Length( UnderlyingElement( g ) ) ); 
    for n2 in orders2 do
        elts2 := Filtered( AsList(G2), g -> Order(g) = n2 );;
        SortBy( elts2, g -> Length( UnderlyingElement( g ) ) ); 
        Print("possible generating pair of order (", n1, ", ", n2, "):\n");
        GeneratorsFound := false;
        for g1 in elts1 do
            for g2 in elts2 do
                if not GeneratorsFound then
                    K1 := Subgroup( G1, [ g1 ] );;
                    K2 := Subgroup( G2, [ g2 ] );;
                    m1 := n1;;
                    m2 := n2;;
                    SearchDone := false;
                    while not SearchDone do
                        K2 := ClosureGroup(K2, Image(phi,
                            Intersection(K1, H1)));
                        K1 := ClosureGroup(K1, Image(psi, 
                            Intersection(K2, H2)));
                        m1new := Order(K1);
                        m2new := Order(K2);
                        if m1new=Ord1 and m2new=Ord2 then
                            m1 := m1new;
                            m2 := m2new;
                            gen1 := g1;
                            gen2 := g2;
                            GeneratorsFound := true;
                            SearchDone := true;
                        elif m1new=m1 and m2new=m2 then
                            SearchDone := true;
                        else
                            m1 := m1new;
                            m2 := m2new;
                        fi;
                    od;
                fi;
            od;
        od;
        if GeneratorsFound then
            Print("(", gen1, ", ", gen2, ")\n");
        else
            Print("No generating pairs found\n");
        fi;
    od;
od;
\end{verbatim}
By running the program for amalgamated products of Djokovi\'c-Miller or Goldschmidt types, we derive a table of generating pairs of every eligible order. Here, cells colored in gray indicate that no pairs of elements have such orders, while crosses imply that none of the pairs in these orders generate the groups.
\begin{sidewaystable}
\centering
\vspace{16.75cm}
\resizebox{\textheight}{!}{%
    \begin{tabular}{|c||c|c|c|c|c|c|c|c|c|}
        \hline
        \diagbox{\tiny{Orders}}{\tiny{Type}} & $\mathrm{DjM}_2^2$ & $\mathrm{DjM}_3$ & $\mathrm{DjM}_4^1$ & $\mathrm{DjM}_4^2$ & $\mathrm{DjM}_5$ & $\mathrm{G}_1^2$ & $\mathrm{G}_1^3$ & $\mathrm{G}_2$ & $\mathrm{G}_2^3$ \\
        \hline\hline
        $(2,4)$ & $(px,y)$ & $(qx,py)$ & {\color{red}\texttimes} & $(rx,y)$ & $(sx,py)$ & \cellcolor{gray!50} & \cellcolor{gray!50} & \cellcolor{gray!50} & \cellcolor{gray!50} \\
        \hline
        $(2,6)$ & \cellcolor{gray!50} & \cellcolor{gray!50} & \cellcolor{gray!50} & \cellcolor{gray!50} & \cellcolor{gray!50} & $(cx,cy)$ & $(cdx,cy)$ & {\color{red}\texttimes} & {\color{red}\texttimes} \\
        \hline
        $(2,8)$ & \cellcolor{gray!50} & \cellcolor{gray!50} & $(rx,qy)$ & $(rx,qy)$ & $(sx,ry)$ & \cellcolor{gray!50} & \cellcolor{gray!50} & \cellcolor{gray!50} & \cellcolor{gray!50} \\
        \hline
        $(3,4)$ & $(x,y)$ & {\color{red}\texttimes} & {\color{red}\texttimes} & $(x,y)$ & {\color{red}\texttimes} & \cellcolor{gray!50} & \cellcolor{gray!50} & \cellcolor{gray!50} & \cellcolor{gray!50} \\
        \hline
        $(3,6)$ & \cellcolor{gray!50} & \cellcolor{gray!50} & \cellcolor{gray!50} & \cellcolor{gray!50} & \cellcolor{gray!50} & $(x,cy)$ & {\color{red}\texttimes} & $(y,dx)$ & $(y,dx)$ \\
        \hline
        $(3,8)$ & \cellcolor{gray!50} & \cellcolor{gray!50} & $(x,qy)$ & $(x,qy)$ & {\color{red}\texttimes} & \cellcolor{gray!50} & \cellcolor{gray!50} & \cellcolor{gray!50} & \cellcolor{gray!50} \\
        \hline
        $(4,2)$ & \cellcolor{gray!50} & \cellcolor{gray!50} & $(prx,y)$ & {\color{red}\texttimes} & $(pqsx,y)$ & \cellcolor{gray!50} & \cellcolor{gray!50} & \cellcolor{gray!50} & \cellcolor{gray!50} \\
        \hline
        $(4,4)$ & \cellcolor{gray!50} & \cellcolor{gray!50} & {\color{red}\texttimes} & $(prx,y)$ & $(pqsx,py)$ & \cellcolor{gray!50} & \cellcolor{gray!50} & \cellcolor{gray!50} & \cellcolor{gray!50} \\
        \hline
        $(4,8)$ & \cellcolor{gray!50} & \cellcolor{gray!50} & $(prx,qy)$ & $(prx,qy)$ & $(pqsx,ry)$ & \cellcolor{gray!50} & \cellcolor{gray!50} & \cellcolor{gray!50} & \cellcolor{gray!50} \\
        \hline
        $(6,2)$ & \cellcolor{gray!50} & $(px,y)$ & \cellcolor{gray!50} & \cellcolor{gray!50} & $(px,y)$ & \cellcolor{gray!50} & $(dx,cdy)$ & \cellcolor{gray!50} & $(cy,cdx)$ \\
        \hline
        $(6,3)$ & \cellcolor{gray!50} & \cellcolor{gray!50} & \cellcolor{gray!50} & \cellcolor{gray!50} & \cellcolor{gray!50} & \cellcolor{gray!50} & {\color{red}\texttimes} & \cellcolor{gray!50} & {\color{red}\texttimes}\\
        \hline
        $(6,4)$ & \cellcolor{gray!50} & $(px,py)$ & \cellcolor{gray!50} & \cellcolor{gray!50} & $(px,py)$ & \cellcolor{gray!50} & \cellcolor{gray!50} & \cellcolor{gray!50} & \cellcolor{gray!50} \\
        \hline
        $(6,6)$ & \cellcolor{gray!50} & \cellcolor{gray!50} & \cellcolor{gray!50} & \cellcolor{gray!50} & \cellcolor{gray!50} & \cellcolor{gray!50} & $(dx,cy)$ & \cellcolor{gray!50} & $(cy,dx)$ \\
        \hline
        $(6,8)$ & \cellcolor{gray!50} & \cellcolor{gray!50} & \cellcolor{gray!50} & \cellcolor{gray!50} & $(px,ry)$ & \cellcolor{gray!50} & \cellcolor{gray!50} & \cellcolor{gray!50} & \cellcolor{gray!50} \\
        \hline
        \diagbox{\tiny{Orders}}{\tiny{Type}} & $\mathrm{G}_2^1$ & $\mathrm{G}_2^2$ & $\mathrm{G}_2^4$ & $\mathrm{G}_3$ & $\mathrm{G}_3^1$ & $\mathrm{G}_4$ & $\mathrm{G}_4^1$ & $\mathrm{G}_5$ & $\mathrm{G}_5^1$ \\
        \hline\hline
        $(2,4)$ & {\color{red}\texttimes} & $(dy^{-1},dx)$ & $(dy^{-1},dex)$ & {\color{red}\texttimes} & {\color{red}\texttimes} & {\color{red}\texttimes} & {\color{red}\texttimes} & {\color{red}\texttimes} & $(xs,yv)$ \\
        \hline
        $(2,6)$ & {\color{red}\texttimes} & {\color{red}\texttimes} & {\color{red}\texttimes} & \cellcolor{gray!50} & $(dx^{-1},ey)$ & $(sx,a^{-1}by)$ & $(xts,a^{-1}by)$ & $(xs,ty)$ & {\color{red}\texttimes} \\
        \hline
        $(2,8)$ & \cellcolor{gray!50} & \cellcolor{gray!50} & \cellcolor{gray!50} & \cellcolor{gray!50} & \cellcolor{gray!50} & $(sx,a^{-1}y)$ & $(xts,a^{-1}y)$ & $(xs,b^{-1}y)$ & {\color{red}\texttimes} \\
        \hline
        $(2,12)$ & $(dy^{-1},dx)$ & \cellcolor{gray!50} & $(dey^{-1},dx)$ & \cellcolor{gray!50} & \cellcolor{gray!50} & $(sx,sy)$ & $(xs,yt)$ & \cellcolor{gray!50} & $(xs,avy)$ \\
        \hline
        $(3,4)$ & {\color{red}\texttimes} & $(y,dx)$ & {\color{red}\texttimes} & $(x,dy)$ & $(x,dey^{-1})$ & $(x,b^{-1}y)$ & {\color{red}\texttimes} & $(x,ay)$ & $(x,yv)$ \\
        \hline
        $(3,6)$ & {\color{red}\texttimes} & $(y,cx)$ & {\color{red}\texttimes} & \cellcolor{gray!50} & $(x,ey)$ & $(x,a^{-1}by)$ & {\color{red}\texttimes} & $(x,ty)$ & {\color{red}\texttimes} \\
        \hline
        $(3,8)$ & \cellcolor{gray!50} & \cellcolor{gray!50} & \cellcolor{gray!50} & \cellcolor{gray!50} & \cellcolor{gray!50} & $(x,a^{-1}y)$ & $(x,tyb)$ & $(x,b^{-1}y)$ & {\color{red}\texttimes} \\
        \hline
        $(3,12)$ & $(y,dx)$ & \cellcolor{gray!50} & {\color{red}\texttimes} & \cellcolor{gray!50} & \cellcolor{gray!50} & $(x,sy)$ & $(x,yt)$ & \cellcolor{gray!50} & $(x,avy)$ \\
        \hline
        $(4,2)$ & $(cy,cx)$ & {\color{red}\texttimes} & $(cey,cx)$ & {\color{red}\texttimes} & {\color{red}\texttimes} & {\color{red}\texttimes} & $(a^2stx,ytb)$ & {\color{red}\texttimes} & $(xsv,a^{-1}y)$ \\
        \hline
        $(4,3)$ & {\color{red}\texttimes} & {\color{red}\texttimes} & {\color{red}\texttimes} & $(cx,cy)$ & $(cex,y)$ & $(asx^{-1},y)$ & $(a^2stx,y)$ & $(b^{-1}sx,y)$ & $(xsv,y)$ \\
        \hline
        $(4,4)$ & {\color{red}\texttimes} & $(cy,dx)$ & $(cy,dex)$ & $(cx,dy)$ & $(cx,dey^{-1})$ & $(asx^{-1},b^{-1}y)$ & $(a^2stx,b^{-1}y)$ & {\color{red}\texttimes} & $(xsv,ay)$ \\
        \hline
        $(4,6)$ & {\color{red}\texttimes} & $(cy,cx)$ & $(cy,cex)$ & \cellcolor{gray!50} & $(cx,cey^{-1})$ & $(asx^{-1},a^{-1}by)$ & $(a^2stx,a^{-1}by)$ & $(b^{-1}sx,sy)$ & $(xsv,ty)$ \\
        \hline
        $(4,8)$ & \cellcolor{gray!50} & \cellcolor{gray!50} & \cellcolor{gray!50} & \cellcolor{gray!50} & \cellcolor{gray!50} & $(asx^{-1},a^{-1}y)$ & $(a^2stx,a^{-1}y)$ & $(b^{-1}sx,b^{-1}y)$ & $(xsv,b^{-1}y)$ \\
        \hline
        $(4,12)$ & $(cy,dx)$ & \cellcolor{gray!50} & $(cey,dx)$ & \cellcolor{gray!50} & \cellcolor{gray!50} & $(asx^{-1},sy)$ & $(b^{-1}sx,yt)$ & \cellcolor{gray!50} & $(xsv,avy)$ \\
        \hline
        $(6,2)$ & \cellcolor{gray!50} & \cellcolor{gray!50} & $(ey,cx)$ & \cellcolor{gray!50} & $(ex,dy^{-1})$ & \cellcolor{gray!50} & $(xt,b^{-1}ty)$ & $(xt,a^{-1}y)$ & $(xt,a^2vy^{-1})$ \\
        \hline
        $(6,3)$ & \cellcolor{gray!50} & \cellcolor{gray!50} & {\color{red}\texttimes} & \cellcolor{gray!50} & $(ex, y)$ & \cellcolor{gray!50} & $(xt,y)$ & $(xt,y)$ & $(vx,y)$ \\
        \hline
        $(6,4)$ & \cellcolor{gray!50} & \cellcolor{gray!50} & $(ey,dex)$ & \cellcolor{gray!50} & $(ex,yd)$ & \cellcolor{gray!50} & $(xt,b^{-1}y)$ & $(xt,ay)$ & $(xt,yv)$ \\
        \hline
        $(6,6)$ & \cellcolor{gray!50} & \cellcolor{gray!50} & $(ey,cex)$ & \cellcolor{gray!50} & $(ex,ey)$ & \cellcolor{gray!50} & $(xt,a^{-1}by)$ & $(xt,ty)$ & $(vx,ty)$ \\
        \hline
        $(6,8)$ & \cellcolor{gray!50} & \cellcolor{gray!50} & \cellcolor{gray!50} & \cellcolor{gray!50} & \cellcolor{gray!50} & \cellcolor{gray!50} & $(xt,a^{-1}y)$ & $(xt,b^{-1}y)$ & $(vx,b^{-1}y)$ \\
        \hline
        $(6,12)$ & \cellcolor{gray!50} & \cellcolor{gray!50} & $(ey,dx)$ & \cellcolor{gray!50} & \cellcolor{gray!50} & \cellcolor{gray!50} & $(xt,yt)$ & \cellcolor{gray!50} & $(xt,avy)$ \\
        \hline
        $(8,2)$ & \cellcolor{gray!50} & \cellcolor{gray!50} & \cellcolor{gray!50} & \cellcolor{gray!50} & \cellcolor{gray!50} & {\color{red}\texttimes} & $(xsb,ytb)$ & {\color{red}\texttimes} & $(xsb,a^2vy^{-1})$ \\
        \hline
        $(8,3)$ & \cellcolor{gray!50} & \cellcolor{gray!50} & \cellcolor{gray!50} & \cellcolor{gray!50} & \cellcolor{gray!50} & $(asx,y)$ & $(xtsb,y)$ & $(xsb,ab^{-1}y)$ & {\color{red}\texttimes} \\
        \hline
        $(8,4)$ & \cellcolor{gray!50} & \cellcolor{gray!50} & \cellcolor{gray!50} & \cellcolor{gray!50} & \cellcolor{gray!50} & $(asx,b^{-1}y)$ & $(xtsb,b^{-1}y)$ & {\color{red}\texttimes} & $(xsb,yv)$ \\
        \hline
        $(8,6)$ & \cellcolor{gray!50} & \cellcolor{gray!50} & \cellcolor{gray!50} & \cellcolor{gray!50} & \cellcolor{gray!50} & $(asx,a^{-1}by)$ & $(xtsb,a^{-1}by)$ & $(xsb,ty)$ & {\color{red}\texttimes} \\
        \hline
        $(8,8)$ & \cellcolor{gray!50} & \cellcolor{gray!50} & \cellcolor{gray!50} & \cellcolor{gray!50} & \cellcolor{gray!50} & $(asx,a^{-1}y)$ & $(xsb,tyb)$ & $(xsb,b^{-1}y)$ & {\color{red}\texttimes} \\
        \hline
        $(8,12)$ & \cellcolor{gray!50} & \cellcolor{gray!50} & \cellcolor{gray!50} & \cellcolor{gray!50} & \cellcolor{gray!50} & $(asx,sy)$ & $(xsb,yt)$ & \cellcolor{gray!50} & $(xsb,avy)$ \\
        \hline
    \end{tabular}
}
\end{sidewaystable}

The second task also requires a utility function for computing $\ell(A_{g_1g_2}\cap T_h)$, which reduces to checking if $h$ is contained in each of the vertex stabilizers on the axis:
\[
...,g_1G_2g_1^{-1}, G_1,\ G_2,\ g_2^{-1}G_1g_2, g_2^{-1}g_1^{-1}G_2g_1g_2,\dots
\]
Or, equivalently, checking if each element is inside $G_1$ or $G_2$:
\[
h\overset{?}{\in} G_1,\ g_1^{-1}hg_1\overset{?}{\in} G_2,\ g_2^{-1}(g_1^{-1}hg_1)g_2\overset{?}{\in} G_1,...
\]
and
\[
h\overset{?}{\in} G_2,\ g_2hg_2^{-1}\overset{?}{\in} G_1,\ g_1(g_2hg_2^{-1})g_1^{-1}\overset{?}{\in} G_2,...
\]
This gives the following utility function:
\begin{verbatim}
AxisTreeIntersLen := function(h, g1, g2, H1, H2, phi, psi, maxSteps)
    local my_len, h_temp, i;

    my_len := 0;
    h_temp := h;

    for i in [1..maxSteps] do
        if not (h_temp in H1) then
            break;
        fi;
        h_temp := Image(phi, h_temp);
        my_len := my_len + 1;
        h_temp := g2 * h_temp * g2^-1;

        if not (h_temp in H2) then
            break;
        fi;
        h_temp := Image(psi, h_temp);
        my_len := my_len + 1;
        h_temp := g1 * h_temp * g1^-1;
    od;

    h_temp := g1^-1 * h * g1;

    for i in [1..maxSteps] do
        if not (h_temp in H1) then
            break;
        fi;
        h_temp := Image(phi, h_temp);
        my_len := my_len + 1;
        h_temp := g2^-1 * h_temp * g2;

        if not (h_temp in H2) then
            break;
        fi;
        h_temp := Image(psi, h_temp);
        my_len := my_len + 1;
        h_temp := g1^-1 * h_temp * g1;
    od;

    return my_len;
end;
\end{verbatim}
Below is the \textit{GAP} implementation, where a few optimizations are applied to reduce the time complexity as in the code comment.

\begin{verbatim}
coset_list := function(K2, H2, psi, g1, g2, h1, h2)
    local my_list;
    my_list := List(List(
                Filtered(List(AsList(K2), x -> g2 * x * h2), x -> x in H2 ),
                x -> Image(psi, x)), x -> g1 * x * h1 );
    return my_list;
end;

# Optimization: consider only the $H$-conjugation equivalence
# representatives in $g_1$
elts1 := Filtered( List( OrbitsDomain( H1, AsList(G1), OnPoints )
  , o -> o[1] ), x -> not x in H1 );
elts2 := Difference( AsList(G2), AsList(H2) );;
SortBy( elts1, g -> Length( UnderlyingElement( g ) ) ); 
SortBy( elts2, g -> Length( UnderlyingElement( g ) ) );

elts_h := List( AsList(H1) );
SortBy( elts_h, g -> Length( UnderlyingElement( g ) ) ); 

largest_length := 0;
genh := Identity(G1);
gen1 := Identity(G1);
gen2 := Identity(G2);

for g1 in elts1 do
    for g2 in elts2 do
        for h in elts_h do
            my_length := AxisTreeIntersLen(h, g1, g2, H1, H2, phi, psi, 100);
            # Optimization: Compute the quantity $l$ in advance, and only
            # proceed to the generating set check if $l$ is greater than 
            # a certain threshold
            # Here only intersection lengths at least $6$ are considered
            if my_length > 5 then
                GeneratorsFound := false;
                K1 := Subgroup( G1, [ h ] );;
                K2 := Image(phi, Intersection(K1, H1));;
                m1 := Order(K1);;
                m2 := Order(K2);;
                SearchDone := false;
                # Optimization: check whether the inclusions $g_1\in K_1$
                # and $g_2\in K_2$ are both satisfied
                # Once both inclusions hold, we only need to perform 
                # the same iteration as in case (1)
                Recovered_g1g2 := false;
                while not SearchDone do
                    if not Recovered_g1g2 then
                        K1 := Subgroup( G1, Concatenation( List(K1),
                          List(Image(psi, Intersection(K2, H2))),
                          coset_list(K2, H2, psi, 
                          g1, g2, Identity(G1), Identity(G2)), 
                          coset_list(K2, H2, psi, g1, g2, g1^-1, g2^-1)));
                        K2 := Subgroup( G2, Concatenation( List(K2), 
                          List(Image(phi, Intersection(K1, H1))),
                          coset_list(K1, H1, phi, 
                          Identity(G2), Identity(G1), g2, g1),
                          coset_list(K1, H1, phi, g2^-1, g1^-1, g2, g1)));
                        Recovered_g1g2 := (g1 in K1) and (g2 in K2);
                    else
                        K2 := ClosureGroup(K2, Image(phi, Intersection(K1, H1)));
                        K1 := ClosureGroup(K1, Image(psi, Intersection(K2, H2)));
                    fi;
                    m1new := Order(K1);
                    m2new := Order(K2);
                    if m1new=Ord1 and m2new=Ord2 then
                        m1 := m1new;
                        m2 := m2new;
                        GeneratorsFound := true;
                        SearchDone := true;
                    elif m1new=m1 and m2new=m2 then
                        SearchDone := true;
                    else
                        m1 := m1new;
                        m2 := m2new;
                    fi;
                od;
                if GeneratorsFound and my_length > largest_length then
                    largest_length := my_length;
                    genh := h;
                    gen1 := g1;
                    gen2 := g2;
                fi;
            fi;
        od;
    od;
od;
Print("Largest length:", largest_length, ", at generators
(", genh, ", ", gen1, " times ", gen2, ")\n");
\end{verbatim}

The largest possible intersection length is summarized in the table below.
\begin{center}
    \begin{tabular}{c|c|c}
        Amalgam Type & Maximal $\ell(T_g,A_h)$ & Generators $(g,h)$ \\
        \hline
        $\mathrm{G}_1^2$ & $2$ & $(c,xy)$ \\
        $\mathrm{G}_2$ & $4$ & $(d,xy)$ \\
        $\mathrm{G}_2^1$ & $2$ & $(d,xy)$ \\
        $\mathrm{G}_2^2$ & $2$ & $(c,xy)$ \\
        $\mathrm{G}_2^3$ & $6$ & $(d,xy^{-1})$ \\
        $\mathrm{G}_2^4$ & $2$ & $(d, xey^{-1})$ \\
        $\mathrm{G}_3$ & $3$ & $(dc, xy)$ \\
        $\mathrm{G}_3^1$ & $3$ & $(c, xey)$ \\
        $\mathrm{G}_4$ & $6$ & $(b^2, xy^{-1})$\\
        $\mathrm{G}_4^1$ & $6$ & $(b^2t, xy^{-1})$\\
        $\mathrm{G}_5$ & $4$ & $(st, xy^{-1})$\\
        $\mathrm{G}_5^1$ & $8$ & $(b^2,xvy)$ \\
    \end{tabular}
\end{center}
\begin{center}
    \begin{tabular}{c|c|c}
        Amalgam Type &  Maximal $\ell(T_g,A_h)$ & Generators $(g,h)$ \\
        \hline
        $\mathrm{DjM}_2^1$ & $2$ & $(p,xy)$ \\
        $\mathrm{DjM}_2^2$ & $2$ & $(p,xy)$ \\
        $\mathrm{DjM}_3$ & $4$ & $(q,xy)$ \\
        $\mathrm{DjM}_4^1$ & $6$ & $(r,xy)$ \\
        $\mathrm{DjM}_4^2$ & $6$ & $(r,xy)$\\
        $\mathrm{DjM}_5$ & $8$ & $(s,xy)$\\
    \end{tabular}
\end{center}
\bibliography{reference}{}
\bibliographystyle{alpha}
\end{document}